\documentclass[reqno,final]{amsart}

\newif\ifmonochrome
\monochromefalse

\usepackage{amsmath}
\usepackage{amssymb}
\usepackage{verbatim}
\usepackage{dsfont}
\usepackage{enumerate}
\usepackage[cmtip,arrow,all]{xy}
\usepackage[mathscr]{eucal}
\usepackage{amsthm}
\usepackage{stmaryrd}
\usepackage{amscd}
\usepackage{amsfonts}
\usepackage{latexsym}
\usepackage{amscd}
\usepackage{graphicx}
\usepackage{tikz}
\usepackage{enumitem}
\usepackage{txfonts}
\usepackage{setspace}
\usepackage{wasysym}
\usepackage{todonotes}
\usepackage[breaklinks=true]{hyperref}
\usepackage{pdfpages}
\usepackage{pb-diagram,pb-xy}
\usetikzlibrary{decorations.pathmorphing}

\makeatletter

\ifmonochrome

\tikzset{bluep/.style={blue,dash pattern = on .7 pt off .5 pt,line width=.3pt}}
\tikzset{redp/.style={red,dash pattern = on 2pt off 1pt,line width=.8pt}}
\tikzset{greenp/.style={green,dash pattern = on 1pt off 1pt,line width = .55pt}}

\else

\tikzset{bluep/.style={blue}}
\tikzset{redp/.style={red}}
\tikzset{greenp/.style={green}}

\fi

\newcommand{\arxiv}[1]{\href{http://arxiv.org/abs/#1}{\tt
    arXiv:\nolinkurl{#1}}}

\newtheorem{theorem}{Theorem}[section]
\newtheorem{lemma}[theorem]{Lemma}

\newtheorem{corollary}[theorem]{Corollary}
\theoremstyle{definition}
\newtheorem{definition}[theorem]{Definition}

\newtheorem{remark}[theorem]{Remark}

\@addtoreset{equation}{section}

\DeclareMathOperator{\image}{im}

\newcommand{\cH}{H}
\newcommand{\dH}{AH}

\newcommand{\op}{\textup{op}}
\newcommand{\rev}{\textup{rev}}
\newcommand{\pmd}{\textup{-pmod}}
\newcommand{\rh}{\mathrm{Heis}}
\def\e{e}
\def\h{h}
\def\s{s}
\def\E{E}
\def\H{H}
\def\S{\mathfrak{S}}
\def\p{p}
\def\c{c}
\def\bull{{\scriptstyle\,\oplus\,}}
\def\star{{\scriptstyle\,\ominus\,}}
\def\jonschi{{\!\begin{tikzpicture}[baseline=-1mm]\node at
    (0,0){$\chi$};\end{tikzpicture}\!\!\!}}

\def\ind{\operatorname{ind}}
\def\res{\operatorname{res}}

\def\Add{\operatorname{Add}}
\def\Kar{\operatorname{Kar}}
\def\Sym{\operatorname{Sym}}
\def\SymZ{\operatorname{Sym}}
\def\SymQ{\operatorname{Sym}_\Q}

\def\SYM{{\mathcal{S}ym}}
\def\Heis{{\mathcal{H}eis}}

\def\AH{{\mathcal{AH}}}
\newcommand{\End}{\operatorname{End}}

\newcommand{\unit}{\mathds{1}}
\newcommand{\Hom}{\operatorname{Hom}}

\newcommand{\Z}{\mathbb{Z}}
\newcommand{\N}{\mathbb{N}}

\newcommand{\Q}{\mathbb{Q}}

\def\la{\lambda}

\renewcommand{\k}{\Bbbk}
\renewcommand{\eta}{\operatorname{flip}}

\tikzset{anchorbase/.style={baseline={([yshift=-0.5ex]current bounding box.center)}}}

\def\red#1{{\color{red} #1}}
\def\blue#1{{\color{blue} #1}}
\def\green#1{{\color{green} #1}}

\def\clockright{\begin{tikzpicture}[baseline=-.9mm]
\filldraw[white] (0,0) circle (1.72mm);
\draw[-] (0,-0.18) to[out=180,in=-90] (-.18,0);
\draw[-] (-0.18,0) to[out=90,in=180] (0,0.18);
\draw[->] (0,0.18) to[out=0,in=90] (0.18,0);
\draw[-] (0.178,0.02) to[out=-78,in=0] (0,-0.18);
\end{tikzpicture}}

\def\anticlockleft{\begin{tikzpicture}[baseline=-.9mm]
\filldraw[white] (0,0) circle (1.72mm);
\draw[-] (0,-0.18) to[out=180,in=-102] (-.178,0.02);
\draw[<-] (-0.18,0) to[out=90,in=180] (0,0.18);
\draw[-] (0.18,0) to[out=-90,in=0] (0,-0.18);
\draw[-] (0,0.18) to[out=0,in=90] (0.18,0);
\end{tikzpicture}}

\def\smallclockred{\begin{tikzpicture}
\filldraw[white] (0,0) circle (1mm);
\draw[-,redp] (0,-0.1) to[out=180,in=-90] (-.1,0);
\draw[-,redp] (-0.1,0) to[out=90,in=180] (0,0.1);
\draw[->,redp] (0,0.1) to[out=0,in=90] (0.09,-0.04);
\draw[-,redp] (0.09,-0.02) to[out=-96,in=0] (0,-0.1);
\end{tikzpicture}}

\def\smallanticlockred{\begin{tikzpicture}
\filldraw[white] (0,0) circle (1mm);
\draw[-,redp] (0,-0.1) to[out=180,in=-84] (-.09,-.02);
\draw[<-,redp] (-0.09,-0.04) to[out=90,in=180] (0,0.1);
\draw[-,redp] (0.1,0) to[out=-90,in=0] (0,-0.1);
\draw[-,redp] (0,0.1) to[out=0,in=90] (0.1,0);
\end{tikzpicture}}

\def\clockrightred{\begin{tikzpicture}[baseline=-.9mm]
\filldraw[white] (0,0) circle (1.72mm);
\draw[-,redp] (0,-0.18) to[out=180,in=-90] (-.18,0);
\draw[-,redp] (-0.18,0) to[out=90,in=180] (0,0.18);
\draw[->,redp] (0,0.18) to[out=0,in=90] (0.18,0);
\draw[-,redp] (0.178,0.02) to[out=-78,in=0] (0,-0.18);
\end{tikzpicture}}

\def\clocktopred{\begin{tikzpicture}[baseline=-.9mm]
\filldraw[white] (0,0) circle (1.72mm);
\draw[-,redp] (0,-0.18) to[out=180,in=-90] (-.18,0);
\draw[->,redp] (-0.18,0) to[out=90,in=180] (0,0.18);
\draw[-,redp] (-0.02,0.178) to[out=12,in=90] (0.18,0);
\draw[-,redp] (0.18,0) to[out=-90,in=0] (0,-0.18);
\end{tikzpicture}}

\def\anticlockrightred{\begin{tikzpicture}[baseline=-.9mm]
\filldraw[white] (0,0) circle (1.72mm);
\draw[-,redp] (0,-0.18) to[out=180,in=-90] (-.18,0);
\draw[-,redp] (-0.18,0) to[out=90,in=180] (0,0.18);
\draw[-,redp] (0,0.18) to[out=0,in=78] (0.178,-0.02);
\draw[<-,redp] (0.18,0) to[out=-90,in=0] (0,-0.18);
\end{tikzpicture}}

\def\anticlockleftred{\begin{tikzpicture}[baseline=-.9mm]
\filldraw[white] (0,0) circle (1.72mm);
\draw[-,redp] (0,-0.18) to[out=180,in=-102] (-.178,0.02);
\draw[<-,redp] (-0.18,0) to[out=90,in=180] (0,0.18);
\draw[-,redp] (0.18,0) to[out=-90,in=0] (0,-0.18);
\draw[-,redp] (0,0.18) to[out=0,in=90] (0.18,0);
\end{tikzpicture}}

\def\smallclockblue{\begin{tikzpicture}
\filldraw[white] (0,0) circle (1mm);
\draw[-,bluep] (0,-0.1) to[out=180,in=-90] (-.1,0);
\draw[-,bluep] (-0.1,0) to[out=90,in=180] (0,0.1);
\draw[->,bluep] (0,0.1) to[out=0,in=90] (0.09,-0.04);
\draw[-,bluep] (0.09,-0.02) to[out=-96,in=0] (0,-0.1);
\end{tikzpicture}}

\def\smallanticlockblue{\begin{tikzpicture}
\filldraw[white] (0,0) circle (1mm);
\draw[-,bluep] (0,-0.1) to[out=180,in=-84] (-.09,-.02);
\draw[<-,bluep] (-0.09,-0.04) to[out=90,in=180] (0,0.1);
\draw[-,bluep] (0.1,0) to[out=-90,in=0] (0,-0.1);
\draw[-,bluep] (0,0.1) to[out=0,in=90] (0.1,0);
\end{tikzpicture}}

\def\clockrightblue{\begin{tikzpicture}[baseline=-.9mm]
\filldraw[white] (0,0) circle (1.72mm);
\draw[-,bluep] (0,-0.18) to[out=180,in=-90] (-.18,0);
\draw[-,bluep] (-0.18,0) to[out=90,in=180] (0,0.18);
\draw[->,bluep] (0,0.18) to[out=0,in=90] (0.18,0);
\draw[-,bluep] (0.178,0.02) to[out=-78,in=0] (0,-0.18);
\end{tikzpicture}}

\def\clocktopblue{\begin{tikzpicture}[baseline=-.9mm]
\filldraw[white] (0,0) circle (1.72mm);
\draw[-,bluep] (0,-0.18) to[out=180,in=-90] (-.18,0);
\draw[->,bluep] (-0.18,0) to[out=90,in=180] (0,0.18);
\draw[-,bluep] (-0.02,0.178) to[out=12,in=90] (0.18,0);
\draw[-,bluep] (0.18,0) to[out=-90,in=0] (0,-0.18);
\end{tikzpicture}}

\def\anticlockleftblue{\begin{tikzpicture}[baseline=-.9mm]
\filldraw[white] (0,0) circle (1.72mm);
\draw[-,bluep] (0,-0.18) to[out=180,in=-102] (-.178,0.02);
\draw[<-,bluep] (-0.18,0) to[out=90,in=180] (0,0.18);
\draw[-,bluep] (0.18,0) to[out=-90,in=0] (0,-0.18);
\draw[-,bluep] (0,0.18) to[out=0,in=90] (0.18,0);
\end{tikzpicture}}

\def\bigdot{{\color{white}\bullet}\hspace{.1mm}\!\!\!\pmb\circ}
\def\dot{{\color{white}\bullet}\!\!\!\circ}

\def\up{\:{\begin{tikzpicture}[anchorbase]\draw[->] (0,0) to (0,.32);\end{tikzpicture}}\:}
\def\redup{\:{\begin{tikzpicture}[anchorbase]\draw[->,redp] (0,0) to (0,.32);\end{tikzpicture}}\:}
\def\smallredup{\:{\begin{tikzpicture}[anchorbase]\draw[->,redp] (0,0) to (0,.22);\end{tikzpicture}}\:}
\def\blueup{\:{\begin{tikzpicture}[anchorbase]\draw[->,bluep] (0,0) to (0,.32);\end{tikzpicture}}\:}
\def\smallblueup{\:{\begin{tikzpicture}[anchorbase]\draw[->,bluep] (0,0) to (0,.22);\end{tikzpicture}}\:}
\def\greenup{\:{\begin{tikzpicture}[anchorbase]\draw[->,greenp] (0,0) to (0,.32);\end{tikzpicture}}\:}

\def\down{\:{\begin{tikzpicture}[anchorbase]\draw[<-] (0,0) to (0,.32);\end{tikzpicture}}\:}
\def\reddown{\:{\begin{tikzpicture}[anchorbase]\draw[<-,redp] (0,0) to (0,.32);\end{tikzpicture}}\:}
\def\bluedown{\:{\begin{tikzpicture}[anchorbase]\draw[<-,bluep] (0,0) to (0,.32);\end{tikzpicture}}\:}

\allowdisplaybreaks

\begin{document}

\title[Degenerate Heisenberg category]{The degenerate Heisenberg
  category\\and its Grothendieck ring}

\author[J.~Brundan]{Jonathan Brundan}
\address{Department of Mathematics\\
University of Oregon\\ Eugene\\ OR 97403\\ USA}
\email{brundan@uoregon.edu}

\author[A.~Savage]{Alistair Savage}
\address{
  Department of Mathematics and Statistics\\
  University of Ottawa\\
  Ottawa, ON\\ Canada
}
\urladdr{\href{http://alistairsavage.ca}{alistairsavage.ca}, \textrm{\textit{ORCiD}:} \href{https://orcid.org/0000-0002-2859-0239}{orcid.org/0000-0002-2859-0239}}
\email{alistair.savage@uottawa.ca}

\author[B.~Webster]{Ben Webster}
\address{Department of Pure Mathematics, University of Waterloo \&
  Perimeter Institute for Theoretical Physics\\
Waterloo, ON\\ Canada}
\email{ben.webster@uwaterloo.ca}

\thanks{J.B.\ supported in part by NSF grant DMS-1700905. A.S.\ and B.W.\ are supported by Discovery Grants RGPIN-2017-03854 and RGPIN-2018-03974 from the Natural Sciences and Engineering Research Council of Canada.}
\thanks{This research was supported in part by Perimeter Institute for Theoretical Physics. Research at Perimeter Institute is supported by the Government of Canada through the Department of Innovation, Science and Economic Development Canada and by the Province of Ontario through the Ministry of Research, Innovation and Science.}


\begin{abstract}
The degenerate Heisenberg category
$\Heis_k$ is a strict monoidal category which
was originally introduced in the special case $k=-1$
by Khovanov in 2010.
Khovanov conjectured that the Grothendieck ring
of the additive Karoubi envelope of his category is isomorphic to a
certain $\Z$-form for the
universal enveloping algebra of the infinite-dimensional Heisenberg
Lie algebra specialized at central charge $-1$. We prove this
conjecture and extend it to arbitrary central charge $k \in \Z$.
We also explain how to categorify the comultiplication (generically).
\end{abstract}

\maketitle

\section{Introduction}

Throughout the article, we work over a fixed ground field $\k$ of characteristic zero.
The {degenerate Heisenberg category} $\Heis_k$ of central charge $k \in \Z$
is a strict $\k$-linear
monoidal category which was
introduced originally by Khovanov \cite{K} in the special case $k=-1$, motivated by the
calculus of induction and restriction functors between representations
of the symmetric groups. 
Khovanov's definition was extended to
arbitrary central charge in \cite{MS18, B2}.
The relations of this category are modeled on those of a $\Z$-form $\rh_k$ for a central reduction of
the universal enveloping algebra $U(\mathfrak{h})$ of the
infinite-dimensional Heisenberg Lie algebra.
By \cite{K, MS18}, there is an injective ring homomorphism
\begin{equation}
\gamma_k: \rh_k \rightarrow K_0(\Kar(\Heis_k))
\end{equation}
to the Grothendieck ring of the additive Karoubi
envelope of $\Heis_k$.
In this paper, we prove that $\gamma_k$ is also surjective, so that
$\Heis_k$ {\em categorifies} $\rh_k$, as was conjectured in \cite{K, MS18}.
We also take a first step towards categorification of the comultiplication on $U(\mathfrak{h})$.

To give more precise statements, we need to recall some basic notions.
Let $\SymZ$ be the ring of {\em symmetric functions}; see \cite{Mac}.
It is freely generated either by the {\em elementary symmetric functions}
$\{\e_n\}_{n \geq 1}$ or the {\em complete symmetric funtions} $\{\h_n\}_{n
  \geq 1}$.
We also have the {\em power sums} $\{\p_n\}_{n \geq 1}$ whose images generate
                     $\SymQ := \Q\otimes_{\Z} \SymZ$.
Moreover, $\SymZ$ is a Hopf ring with comultiplication $\delta:\SymZ
\rightarrow \SymZ\otimes_\Z \SymZ, f \mapsto \sum_{(f)}
f_{(1)}\otimes f_{(2)}$ satisfying
\begin{align}\label{comultiplication}
\delta(\h_n) &= \sum_{r=0}^n
\h_{n-r} \otimes \h_{r},
&
\delta(\e_n) &= \sum_{r=0}^n
\e_{n-r} \otimes \e_{r},
&
\delta(\p_n) &= \p_n \otimes 1 + 1 \otimes \p_n,
\end{align}
where $\h_0 = \e_0 = 1$ by convention.
As a $\Z$-module, $\SymZ$ is free with the canonical basis
$\{\s_\lambda\}_{\lambda \in \mathcal P}$ of {\em Schur functions}
indexed by the set $\mathcal P$ of all partitions.

The {\em infinite-dimensional Heisenberg Lie algebra} is the
Lie algebra $\mathfrak{h}$ over $\Q$ 
with basis $\left\{\c, \p_n^{\pm}\:\big|\:n \geq 1\right\}$
and Lie bracket defined from
\begin{align}
[\c,\p_n^{\pm}] &=
[\p_m^+, \p_n^+] = [\p_m^-,\p_n^-] = 0,&
[\p_m^+,\p_n^-] &= \delta_{m,n} n\c.
\end{align}
The central reduction
$U(\mathfrak{h}) / (\c - k)$ of its universal enveloping algebra
may also be realized as 
the {\em Heisenberg
  double}
$\SymQ \#_\Q \SymQ$ with respect
to the bilinear  Hopf
pairing 
\begin{equation} \label{Hopfpair}
  \langle -, - \rangle_k \colon \SymQ \times \SymQ \to \Q,\quad
  \langle \p_m, \p_n \rangle_k = \delta_{m,n}nk.
\end{equation}
By definition, $\SymQ\#_\Q \SymQ$ is
the vector space
$\SymQ \otimes_\Q \SymQ$ with associative multiplication defined by
\[
  (e \otimes f)(g \otimes h) := \sum_{(f),(g)} \left\langle
    f_{(1)}, g_{(2)} \right\rangle_k eg_{(1)} \otimes f_{(2)} h.
\]
The pairing 
of two complete symmetric functions is an integer, as follows for
example by comparing the coefficients appearing in \cite[Th.~5.3]{Sua}
to \cite[(2.2)]{Sua}.  
Thus we can restrict to obtain a biadditive form $\langle-,-\rangle_k:\SymZ \times \SymZ
\to \Z$. The resulting Heisenberg double
\begin{equation}
\rh_k := \SymZ\#_\Z\SymZ
\end{equation}
gives us a natural $\Z$-form for $U(\mathfrak{h}) / (\c-k)
\cong \SymQ \#_\Q \SymQ$.
For $f \in \SymZ$, we write $f^-$ and $f^+$ for the elements
$f \otimes 1$ and $1 \otimes f$ of $\rh_k$, respectively.
Then 
$\rh_k$ is generated as a ring by the elements $\{\h_n^+, \e_n^-\}_{n \geq 0}$ subject to
the relations 
\begin{align}\label{upper}
\h_0^+&=\e_0^-=1,&
\h_m^+ \h_n^+ &= \h_n^+ \h_m^+,&
\e_m^- \e_n^- &= \e_n^- \e_m^-,&
  \h_m^+ \e_n^- &= 
\sum_{r=0}^{\min(m,n)}
                  \binom{\,k\,}{\,r\,}\:\e_{n-r}^- \h_{m-r}^+.
\end{align}
See \cite[Section 5]{Sua} and \cite[Appendix
A]{LRS} where this and other presentations are derived. 
The usual comultiplication on $U(\mathfrak{h})$ descends
to ring homomorphisms
\begin{equation}\label{talent}
\delta_{l|m}:\rh_k \rightarrow \rh_l \otimes_\Z \rh_m,\quad
f^\pm \mapsto \sum_{(f)} (f_{(1)})^\pm \otimes (f_{(2)})^{\pm}
\end{equation}
for $k=l+m$ and $f \in \SymZ$.
The antipode induces 
$\sigma_k:\rh_k \stackrel{\sim}{\rightarrow} (\rh_{-k})^{\op},
s_\lambda^{\pm} \mapsto (-1)^{|\lambda|} s_{\lambda^T}^{\pm}$.
Also there is an
isomorphism
\begin{equation} \label{sedai}
    \omega_k \colon \rh_{k}\stackrel{\sim}{\rightarrow} \rh_{-k},\qquad
    s_\lambda^\pm\mapsto s_{\lambda^T}^\mp.
\end{equation}

The {\em degenerate Heisenberg category} $\Heis_k$ is a
strict $\k$-linear monoidal category 
with two generating objects $\up$ and $\down$ and six generating morphisms
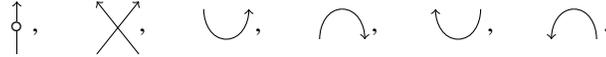
\begin{align*}
\mathord{
\begin{tikzpicture}[baseline = 0]
	\draw[->] (0.08,-.3) to (0.08,.4);
      \node at (0.08,0.05) {$\dot$};
\end{tikzpicture}
}, \qquad
\mathord{
\begin{tikzpicture}[baseline = 0]
	\draw[->] (0.28,-.3) to (-0.28,.4);
	\draw[->] (-0.28,-.3) to (0.28,.4);
\end{tikzpicture}
}, \qquad
\mathord{
\begin{tikzpicture}[baseline = 1mm]
	\draw[<-] (0.4,0.4) to[out=-90, in=0] (0.1,0);
	\draw[-] (0.1,0) to[out = 180, in = -90] (-0.2,0.4);
\end{tikzpicture}
}\:, \qquad
\mathord{
\begin{tikzpicture}[baseline = 1mm]
	\draw[<-] (0.4,0) to[out=90, in=0] (0.1,0.4);
	\draw[-] (0.1,0.4) to[out = 180, in = 90] (-0.2,0);
\end{tikzpicture}
}\:, \qquad
\mathord{
\begin{tikzpicture}[baseline = 1mm]
	\draw[-] (0.4,0.4) to[out=-90, in=0] (0.1,0);
	\draw[->] (0.1,0) to[out = 180, in = -90] (-0.2,0.4);
\end{tikzpicture}
}\:, \qquad
\mathord{
\begin{tikzpicture}[baseline = 1mm]
	\draw[-] (0.4,0) to[out=90, in=0] (0.1,0.4);
	\draw[->] (0.1,0.4) to[out = 180, in = 90] (-0.2,0);
\end{tikzpicture}
}\:.
\end{align*}
A full set of relations between these generating morphisms
is recorded in
Definition~\ref{maindef} below, 
where we adopt the usual string calculus for strict monoidal
categories. 
The relations imply that $\Heis_k$ is {\em strictly pivotal} with 
duality functor $*$ defined on a morphism 
by rotating its string diagram through $180^\circ$.
In particular, the generating objects $\up$ and $\down$ are duals of
each other.
Letting $\S_n$ denote the symmetric group with basic transpositions $s_1,\dots,s_{n-1}$,
there is also an algebra homomorphism 
$\imath_n: \k \S_n
  \rightarrow \End_{\Heis_k}\left(\up^{\otimes n}\right)$,
which sends $s_i$ to the crossing of the $i$th and $(i+1)$th
strings. Note we always number strings in diagrams by $1,2,\dots$ from {\em right to left}.

By the {\em additive Karoubi envelope}
$\Kar(\Heis_k)$ of $\Heis_k$, we mean
the idempotent completion of its additive envelope
$\Add(\Heis_k)$.
Let $K_0(\Kar(\Heis_k))$ be the Grothendieck ring
of the monoidal category $\Kar(\Heis_k)$, i.e., the split
Grothendieck group with multiplication $[X][Y] := [X
\otimes Y]$.
For $\lambda \in \mathcal P$ with $|\lambda|=n$, let $e_\lambda \in \k
\S_n$ be the corresponding Young symmetrizer, so that the left ideal $S(\lambda) := (\k \S_n)
e_\lambda$ is the usual (irreducible) {\em Specht module} for the
symmetric group.
Associated to the idempotent $e_\lambda$, we also have the object
\begin{align}\label{goodtimes}
S_\lambda^+ &:=
\left(\up^{\otimes n}, \imath_n(e_\lambda)\right) \in \Kar(\Heis_k).
\end{align}
Let $S_\lambda^- := (S_\lambda^+)^*$,
and set $\H_n^{\pm} := S_{(n)}^{\pm}$ and
$\E_n^{\pm} := S_{(1^n)}^{\pm}$
for short.
Our first main result is as follows.

\begin{theorem}\label{t1}
There is a ring isomorphism $\gamma_k:\rh_k \stackrel{\sim}{\rightarrow}
K_0(\Kar(\Heis_k))$ such that 
$\s_\lambda^{\pm}\mapsto
[S_\lambda^{\pm}]$ for each $\lambda \in \mathcal{P}$.
In particular,
$\h_n^{\pm} \mapsto [\H_n^{\pm}]$ and
$\e_n^{\pm} \mapsto [\E_n^{\pm}]$.
Also for $X \in \Kar(\Heis_k)$ we have that
$[X] = 0 \Rightarrow X = 0$.
\end{theorem}

This proves extended versions of \cite[Conjecture~1]{K} 
and \cite[Conjecture~4.5]{MS18}. The original conjectures in {\em
  loc.\ cit.} are concerned with the specialization $\Heis_k(\delta)$
of $\Heis_k$
obtained by
evaluating the (strictly central) bubble
$\begin{tikzpicture}[baseline=-1mm]
    \draw[<-] (0,0.2) arc(90:450:0.2);
    \draw (-0.2,0) node {$\dot$} node[anchor=east] {$\scriptstyle{k}$};
  \end{tikzpicture}=
\begin{tikzpicture}[baseline=-1mm]
    \draw[->] (0,0.2) arc(90:450:0.2);
    \draw (0.2,0) node {$\dot$} node[anchor=west] {$\scriptstyle{-k}$};
  \end{tikzpicture}$
at a scalar $\delta \in \k$; see \cite[Theorem 1.4]{B2}.
We will not discuss this specialization further here, but note that
our arguments can be carried out in $\Heis_k(\delta)$ in exactly
the same way as in $\Heis_k$.
Consequently, Theorem~\ref{t1} remains true when $\Heis_k$ is
replaced by $\Heis_k(\delta)$. The specialized version with $k=-1,\delta=0$ or
with $k < 0,\delta\in\Z$ proves
the original conjectures from \cite{K} and \cite{MS18}, respectively.

The main new ingredient needed to prove Theorem~\ref{t1} is to
show that $\gamma_k$ is {\em surjective}. We do this by combining the strategy
proposed by Khovanov in \cite[Section 5]{K} with one additional general result
about Grothendieck groups; see Theorem~\ref{Grothsplit}. This additional result is well known
(and easy to prove) in the setting of finite-dimensional algebras. However, we need it here for
algebras that are not finite-dimensional and, at this level of
generality, we actually could not find it explicitly in the
literature (but see \cite{D} for a related result).

We also prove the following theorem, which categorifies the relations
(\ref{upper}). 
An analogous result categorifying the commutation relations
between $h_m^+$ and $h_n^-$ was recorded in
\cite[Proposition 4.3]{MS18}, where it was used to construct the
homomorphism $\gamma_k$ in the first
place. In our proof of Theorem~\ref{t1} explained in Section~\ref{sp},
we give a new approach
to the construction of $\gamma_k$, thereby making our arguments
completely independent of {\em loc.\ cit.}.
We are then able to exploit Theorem~\ref{t1} to give a considerably simplified
proof of the
categorical relations; see Section~\ref{sfin}.

\begin{theorem}\label{t3}
In $\Kar(\Heis_k)$, there are distinguished isomorphisms 
\begin{align*}
H_m^+
\otimes H_n^+ &\cong H_n^+ \otimes H_m^+,&
H_m^+ \otimes E_n^-
 &\cong 
\bigoplus_{r=0}^{\min(m,n,k)} \bigoplus_{\lambda \in \mathcal P_{r,k}}
E_{n-r}^-
\otimes H_{m-r}^+\quad\text{if $k \geq 0$,}\\
E_m^-\otimes E_n^- &\cong E_n^-\otimes E_m^-,&
E_m^- \otimes H_n^+ &\cong
\bigoplus_{r=0}^{\min(m,n,-k)} \bigoplus_{\lambda \in \mathcal P_{r,-k}}
H_{n-r}^+
\otimes E_{m-r}^-\quad\text{if $k \leq 0$,}
\end{align*}
where $\mathcal P_{r,k}$ denotes the set of all partitions whose Young
diagram fits into an $r \times (k-r)$ rectangle.
\end{theorem}

The other key ingredient making this new approach possible
is a strict monoidal functor
\begin{equation}\label{finalversion}
\Delta_{l|m}:\Kar(\Heis_k) \rightarrow \Kar(\Heis_l \;\overline{\odot}\;\Heis_m)
\end{equation}
for $k=l+m$; see Theorem~\ref{comult}.
Here, $-\odot -$ denotes symmetric product of strict monoidal
categories (see Section~\ref{newsec} for the definition),
and
$\Heis_l \;\overline{\odot}\;\Heis_m$
is the localization of
$\Heis_l \odot \Heis_m$
at the morphism
\begin{equation}\label{sand}
\mathord{
\begin{tikzpicture}[baseline = -1mm]
 	\draw[->,redp] (0.18,-.4) to (0.18,.4);
	\draw[->,bluep] (-0.38,-.4) to (-0.38,.4);
     \node at (0.18,0) {$\red{\dot}$};
\end{tikzpicture}
}-
\mathord{
\begin{tikzpicture}[baseline = -1mm]
 	\draw[->,redp] (0.18,-.4) to (0.18,.4);
	\draw[->,bluep] (-0.38,-.4) to (-0.38,.4);
     \node at (-0.38,0) {$\blue{\dot}$};
\end{tikzpicture}}\:,
\end{equation}
where the left (blue) string comes from $\Heis_l$ and the
right (red) string comes from $\Heis_m$.
The following explains how $\Delta_{l|m}$ categorifies the
comultiplication $\delta_{l|m}$ from (\ref{talent}). 

\begin{theorem}\label{t2}
For any $k=l+m$, there is a commutative diagram
$$
\begin{diagram}
\node{\rh_k}\arrow[2]{s,l,A,J}{\gamma_k}\arrow{e,t}{\delta_{l|m}}\node{\rh_l\otimes_\Z\rh_m}\arrow{se,t,A,J}{\gamma_l\otimes\gamma_m}\\
\node[3]{\hspace{-14mm}K_0(\Kar(\Heis_l))\otimes_\Z
  K_0(\Kar(\Heis_m))}\arrow{sw,b}{\epsilon_{l|m}}\\
\node{K_0(\Kar(\Heis_k))}\arrow{e,b}{[\Delta_{l|m}]}\node{K_0(\Kar(\Heis_l\;\overline\odot\;\Heis_m)),}
\end{diagram}
$$
where
$\epsilon_{l|m}$ is the ring homomorphism induced by the canonical functors
from $\Heis_l$ and
$\Heis_m$ to $\Heis_l \;\overline{\odot}\;\Heis_m$.
\end{theorem}

The categorical comultiplication $\Delta_{l|m}$ allows one to take tensor products
of Heisenberg module categories provided that the morphism (\ref{sand}) acts
invertibly (that is, there is no overlap in the spectrum of the red and blue dots). 
In Section~\ref{sbt}, we give another application of this principle,
namely, an efficient new proof of the basis
theorem for morphism spaces in $\Heis_k$ from \cite[Theorem 1.6]{B2} (where it is
proved by invoking results of \cite{K, MS18} when $k < 0$
and \cite{BCNR} when $k = 0$). The same general idea was first
used in 
\cite{Wunfurling}, and its formulation via categorical comultiplication as developed here has subsequently been applied to establish basis theorems for several other diagrammatic monoidal categories of a similar
nature, including Frobenius and quantum analogs of the Heisenberg
category; see \cite{qheis, Foundations, QFrobHeis}.

\vspace{1mm}
\noindent
{\em Organization.}
We begin in Section~\ref{sgg} by proving the key auxiliary result about Grothendieck rings, namely, that $K_0(R) \cong K_0(eRe) \oplus K_0(R / ReR)$ 
for a unital ring $R$ and an idempotent $e \in R$
assuming that the quotient map $R \rightarrow R / ReR$ splits in a suitable way. Section~\ref{newsec} introduces symmetric groups $\S_n$ for all $n \geq 0$ in the guise of the {\em symmetric category} $\SYM$, which is the free strict $\k$-linear symmetric monoidal category on one object. Then we reformulate a classical result of Frobenius showing that the Grothendieck ring of 
$\Kar(\SYM)$ is isomorphic to the ring $\SymZ$ of symmetric functions. This depends crucially on the assumption that $\k$ is of characteristic zero.
In fact, $\SymZ$ is a Hopf ring, and we also explain here an elementary but non-standard way to categorify the comultiplication on $\SymZ$ via a strict monoidal functor
\begin{equation}
\Delta:
\SYM \rightarrow \Add(\SYM\odot\SYM).
\end{equation}
The {\em degenerate affine Hecke algebra} $AH_n$ appears for the first time in Section~\ref{sdha}. 
This algebra is a filtered deformation of the smash product $\k[x_1,\dots,x_n] \# \S_n$ so that, 
by a result of Quillen, its Grothendieck ring
 may be identified with that of the group algebra $\k S_n$; again this requires the ground field to be of characteristic zero.
In fact, we work with a 
strict $\k$-linear monoidal $\AH$ obtained from  the symmetric category $\SYM$ by adjoining a new generator (the ``dot")
corresponding to the polynomial generators of $AH_n$,
and deduce that $K_0(\Kar(\AH)) \cong \SymZ$.
We have to work rather hard in this section to extend the definition of the monoidal functor $\Delta$ from $\SYM$ to $\AH$. In particular, to do this, it is already necessary to invert the morphism (\ref{sand}).
The {\em degenerate Heisenberg category} $\Heis_k$ is then obtained in Section~\ref{sdhc} by ``rigidifying" $\AH$ following the approach of \cite{B2}, and we are finally able to construct the functor (\ref{finalversion}) by some further explicit but remarkable computations with the defining relations. We use this functor to give our new proof of the basis theorem for morphism spaces in $\Heis_k$ in Section~\ref{sbt}. Finally all of the required ingredients are in place, and we are ready to prove the main results. The proofs of Theorems~\ref{t1} and \ref{t2} appear at the end of Section~\ref{sp}, then Theorem~\ref{t3} is proved at the end of Section~\ref{sfin}. To establish the latter result, in the final section, we also introduce some additional ``thick calculus" which is of independent interest.

\section{A general result about Grothendieck groups}\label{sgg}

In this section, until the final paragraph, all rings and modules are assumed to be unital.
For a ring $R$, we let $K_0(R)$ denote the split Grothendieck group of
the category $R\pmd$ of finitely generated projective left
$R$-modules.
By definition (e.g., see \cite[Definition 1.1.5]{Ros94}), this is the group
completion of the commutative monoid consisting of
isomorphism
classes of finitely generated projectives with respect to the operation $+$
induced by taking direct sums of modules.
We write $[P]$ for the image of the isomorphism class of $P \in R\pmd$
in $K_0(R)$.
According to the definition of group completion, any element of $K_0(R)$
can be written in the form
$[P] - [P']$ for $P,P' \in R\pmd$.
Furthermore $[P] - [P'] = 0$ in $K_0(R)$ if and only if $P \oplus Q \cong P' \oplus Q$ for some $Q \in R\pmd$.
Since $Q$ is finitely generated and projective, it is a direct summand
of a free module of finite rank.  In other words, there exists $Q' \in
R\pmd$ and $n \geq 0$ such that $Q \oplus Q' \cong R^n$.  Hence:
\begin{equation} \label{slap}
  [P] - [P'] = 0 \text{ in } K_0(R)
  P \oplus R^n \cong P' \oplus R^n \text{ for some } n\geq 0.
\end{equation}
The ring $R$ is {\em stably finite} 
if $AB = 1 \Rightarrow BA = 1$ for all matrices $A, B \in M_n(R)$ and all $n \geq 1$.
This is equivalent to the property 
$P \oplus R^n
\cong R^n
\Rightarrow P = 0$ for all $P \in R\pmd$, i.e.,
$[P] = 0\Rightarrow P=0$.

\begin{lemma}\label{beerisgood}
If $R$ is finitely generated as a module over its
center then $R$ is stably finite.
\end{lemma}

\begin{proof}
Suppose that $P \oplus R^n \cong R^n$ for some non-zero $P \in R\pmd$.
Since $P$ is finitely generated over the center $Z$ of $R$,
the Nakayama lemma (\cite[Cor. 4.8]{Eisenbud}) implies that $J(Z)P \ne P$, where $J(Z)$ is the Jacobson radical of $Z$.  It follows that the quotient $P/\mathfrak{m} P$ is non-zero for some maximal ideal 
$\mathfrak{m}$ of $Z$. 
Then we have that 
 $P/\mathfrak{m}P\oplus
(R/\mathfrak{m}R)^n\cong (R/\mathfrak{m}R)^n$ as $R /
\mathfrak{m}R$-modules, hence, as $Z / \mathfrak{m}$-vector spaces.
This is clearly impossible by dimension considerations.
\end{proof}

Suppose $R$ and $S$ are rings, and $M$ is an $(S,R)$-bimodule that is
finitely generated and projective as a left
$S$-module.  Then we have the induced functor
\[
F:  R\pmd \to S\pmd,\quad P \mapsto M \otimes_R P,
\]
which induces a homomorphism of Abelian groups
$[F]:  K_0(R) \to K_0(S)$.
The main result in this section is as follows. Note in \cite[Theorem 2.2(3)]{D} one finds a similar split short exact sequence in $K$-theory, but this is proved under different hypotheses.
 
\begin{theorem} \label{Grothsplit}
  Suppose $R$ is a ring and $e \in R$ is an idempotent.  Let $S := R/ReR$
  and suppose that there exists a unital ring homomorphism $\sigma \colon S
  \to R$ such that $\pi \circ \sigma = \operatorname{id}_S$, where $\pi \colon R \to S$ is
  the quotient map.
Then there is a split short exact sequence of Abelian groups
\begin{equation}\label{theseq}
    0\longrightarrow K_0(eRe) \stackrel{\phi}{\longrightarrow} K_0(R) \stackrel{\psi}{\longrightarrow} K_0(S) \longrightarrow 0,
\end{equation}
where $\phi([P]) := [Re \otimes_{eRe} P]$ and $\psi([Q]) :=
[S\otimes_R Q]$.
Moreover, $R$ is stably finite if and only if both $eRe$ and $S$
are stably finite.
\end{theorem}

The proof will be carried out in the remainder of the section via a
series of lemmas.
We begin with some elementary remarks. First, the map $\phi$
is well defined since the $(R, eRe)$-bimodule $Re$ is finitely
generated and projective as a left $R$-module.
Similarly, the map $\psi$ is well defined since the $(S,R)$-bimodule
$S$
is finitely generated and projective as a left $S$-module.
We may denote this bimodule also by $S_\pi$ to make it clear that the
right $R$-module structure is defined via the homomorphism
$\pi:R \rightarrow S$. Similarly, we have the $(R,S)$-bimodule
$R_\sigma$, which is the left regular $R$-module $R$ with right action of $S$
defined by $rs := r \sigma(s)$.
Note that
\begin{equation}\label{drwho}
S_\pi \otimes_R R_\sigma \cong S_{\pi\circ\sigma} = S
\end{equation}
as an $(S,S)$-bimodule.

\begin{lemma}\label{psisurj}
  The map $\psi$ from (\ref{theseq}) is a split surjection.
\end{lemma}

\begin{proof}
Since $R_\sigma$ is finitely generated and projective as a left $R$-module,
we get a well-defined map
$\theta:K_0(S) \rightarrow K_0(R)$, $[P] \mapsto [R_\sigma\otimes_S P]$.
Then the identity (\ref{drwho}) implies that $\psi\circ\theta =
\operatorname{id}$.
\end{proof}
  
\begin{lemma} \label{phiing}
  The map $\phi$ is injective.
\end{lemma}

\begin{proof}
  As noted above, any element in $K_0(eRe)$ can be written in the form $[P] - [P']$ for some $P, P' \in eRe\pmd$.  Suppose $[P] - [P'] \in \ker(\phi)$.  Then we have that $[Re \otimes_{eRe} P] - [Re \otimes_{eRe} P'] = 0$, so by \eqref{slap} there exists $n \in \N$ such that there is an isomorphism
  \[
    \theta:Re \otimes_{eRe} P \oplus R^n
    \rightarrow Re \otimes_{eRe} P'  \oplus R^n.
  \]
  Writing maps on the right, $\theta$ can  be represented by right multiplication by an invertible $2\times 2$ matrix
  $
  \left[
  \begin{smallmatrix}
    A & B \\
    C & D
  \end{smallmatrix}
  \right]
  $ for $A:Re\otimes_{eRe}P\rightarrow Re\otimes_{eRe} P'$,  a row vector $B:Re\otimes_{eRe} P \rightarrow R^n$,
  a column vector $C:R^n \rightarrow Re\otimes_{eRe} P'$ and an $n\times n$ matrix
  $D \in M_n(R)$. The image $\pi(D)$ in $M_{n}(S)$ is invertible, so, without loss of generality, we can assume $\pi(D)$ is the identity matrix $I_n$. Then we have that $D=I_n-\sum_{k=1}^m A_k e B_k$ for some $m\geq 1$ and $A_k,B_k \in M_n(R)$.  
  Consider the homomorphism
  \[
    \theta' \colon Re \otimes_{eRe} P \oplus R^n \oplus (Re)^n \oplus\cdots\oplus (Re)^n
    \to  Re \otimes_{eRe} P' \oplus R^n\oplus (Re)^{ n}\oplus\cdots\oplus (Re)^n\] 
    (where there are $m$ summands $(Re)^n$ on each side)
    defined by right multiplication by the matrix
  \[
X:=    \begin{bmatrix}
      A & B & 0 & 0 & \cdots & 0\\
      C & I_n & A_1e & A_2e& \cdots & A_m e\\
      0 & eB_1 & e I_n & 0 & \cdots & 0\\
      0 & eB_2 & 0 & e I_n & \cdots &0\\
      \vdots &\vdots & \vdots & \vdots & \ddots & \vdots\\
      0 & eB_m & 0 & 0 &\cdots& e I_n
    \end{bmatrix}.
  \]
  By some obvious elementary row operations, the matrix $X$ can be transformed into 
   the invertible matrix
  \[
    \begin{bmatrix}
      A & B & 0 & 0 & \cdots & 0 \\
      C & D & 0 & 0 & \cdots & 0 \\
      0 & eB_1 & e I_n & 0 & \cdots & 0\\
      0 & eB_2 & 0 & e I_n & \cdots &0\\
      \vdots &\vdots & \vdots & \vdots & \ddots & \vdots\\
      0 & eB_m & 0 & 0 &\cdots& e I_n
    \end{bmatrix}.
  \]
  It follows that the matrix $X$ is invertible. On the other hand, by some other elementary row and column operations, the matrix $X$ can be transformed into a matrix of the form
   \[
    \begin{bmatrix}
      Y_{1,1} & 0 & Y_{1,2} & Y_{1,3} & \cdots & Y_{1,m+1} \\
      0 & I_n & 0 & 0 & \cdots & 0 \\
      Y_{2,1} & 0 & Y_{2,2} & Y_{2,3} & \cdots & Y_{2,m+1}\\
      Y_{3,1} & 0 & Y_{3,2} & Y_{3,3} & \cdots &Y_{3,m+1}\\
      \vdots &\vdots & \vdots & \vdots & \ddots & \vdots\\
      Y_{m+1,1} & 0 & Y_{m+1,2} & Y_{m+1,3} &\cdots& Y_{m+1,m+1}
    \end{bmatrix}.
  \]
  This produces an invertible  matrix $Y = (Y_{i,j})_{i,j=1,\dots,m+1}$ such that right multiplication by $Y$ defines an isomorphism
    \[
    \theta''\colon Re \otimes_{eRe} P \oplus (Re)^n\oplus\cdots\oplus (Re)^n
    \to Re \otimes_{eRe} P' \oplus (Re)^n\oplus\cdots\oplus (Re)^n.
  \]
  Finally, we restrict $\theta''$ to $eRe\otimes_{eRe} P \oplus (eRe)^{n}\oplus\cdots\oplus (eRe)^n$,
  noting that $eRe \otimes_{eRe} P \cong P$ and $eRe\otimes_{eRe} P' \cong P'$,
  to obtain an isomorphism of $eRe$-modules $P \oplus (eRe)^{mn} \cong P' \oplus (eRe)^{mn}$.
  Hence, $[P]-[P'] = 0$ in $K_0(eRe)$ by \eqref{slap}.
\end{proof}

\begin{lemma}\label{bacomplex}
 We have that $\psi\circ \phi = 0$.
\end{lemma}

\begin{proof}
For any right $R$-module $M$, the multiplication map is an isomorphism
$M \otimes_R Re \cong M e$.
Applying this to $M = S_\pi$, we see that $S_\pi \otimes_R Re \cong (S_\pi) e$, which is
zero as $\pi(e) = 0$.
The map $\psi\circ\phi$ is defined by tensoring with this bimodule.
\end{proof}

\begin{lemma} \label{quite}
  If $P \in R\pmd$ and $S_\pi \otimes_R P= 0$, then $P \cong Re
  \otimes_{eRe} V$ for some $V \in eRe\pmd$.
\end{lemma}

\begin{proof}
  Suppose $P\in R\pmd$ and $S_\pi \otimes_R P = 0$.  Let $V := eP$, which is naturally an $eRe$-module.    Consider the homomorphism of $R$-modules
  \[
    \mu \colon Re \otimes_{eRe} V \to P,\quad
    ae \otimes v \mapsto aev.
  \]
Since $0 = S_\pi\otimes_R P = (R / ReR) \otimes_R P \cong P/ ReP$,
it follows that $ReP = P$.
Hence, $\mu$ is surjective.
Since $P$ is projective as a left $R$-module, the map $\mu$ splits,
so we have a homomorphism of $R$-modules $\tau \colon P
\to Re \otimes_{eRe} V$ such that $\mu \circ\tau = \operatorname{id}_P$.
Restricting, we have
  \[
    V = eP
    \xrightarrow{\tau} eRe \otimes_{eRe} V
    \xrightarrow[\cong]{\mu} V.
  \]
  In other words, $\tau|_V$ splits the isomorphism $\mu|_{eRe
    \otimes_{eRe} V}$ and hence must be its inverse.  Thus $e \otimes
  V \subseteq \image \tau$. It follows that $\tau$ is surjective, hence, an isomorphism.
We have now shown that $P \cong Re \otimes_{eRe} V$ as $R$-modules.
It remains to show that $V$ is finitely generated and projective.

Since $P \in R\pmd$, we can choose elements $p_1,\dots,p_m$ that generate $P$ as an $R$-module.  As noted above, we have $P = ReP$.  Hence, for each $i=1,\dotsc,m$, we can write
  \[
    p_i = \sum_{j=1}^{n_i} a_{i,j} e q_{i,j}
  \]
for some $n_i \geq 0$, $a_{i,j}\in R$ and
$q_{i,j} \in P$.
The elements $\{e q_{i,j}\:|\:i=1,\dots,m,
  j=1,\dots,n_i\}$
generate $V$ as an $eRe$-module.  So $V$ is finitely generated.

  To see that $V$ is projective, suppose we have a surjective homomorphism of $eRe$-modules $\theta \colon U \twoheadrightarrow V$.  Then we have an induced surjective homomorphism of $R$-modules
  \[
    \operatorname{id} \otimes \theta \colon Re \otimes_{eRe} U \twoheadrightarrow Re \otimes_{eRe} V \cong P.
  \]
  Since $P$ is projective, this map splits.  So we have a homomorphism
  of $R$-modules $$
\xi \colon Re \otimes_{eRe} V \to Re \otimes_{eRe}
  U$$
such that $(\operatorname{id} \otimes \theta) \circ \xi =
  \operatorname{id}_{Re \otimes_{eRe} V}$.  From this, we see that
the restriction $\xi|_{eRe\otimes_{eRe}V}$ splits the restriction
$(\operatorname{id} \otimes\theta)|_{eRe\otimes_{eRe} U}:
eRe\otimes_{eRe} U \rightarrow eRe\otimes_{eRe} V$.
Under the natural isomorphisms $eRe \otimes_{eRe} U\cong U$ and $eRe \otimes_{eRe} V
\cong
V$, the map
$(\operatorname{id} \otimes\theta)|_{eRe\otimes_{eRe} U}$
corresponds to $\theta$.  So $\theta$ splits.
\end{proof}

\begin{lemma}\label{jonfounditeasiertostatethisseparately}
Suppose $P \in R\pmd$
and let
 $Q := S_\pi\otimes_R P$.
There exists $V \in eRe \pmd$ and $n \geq 0$ such
that
$R_\sigma\otimes_S Q \oplus
(Re)^n
\cong P \oplus Re\otimes_{eRe} V$
as $R$-modules.
\end{lemma}

\begin{proof}
Since tensor is left adjoint to hom, we have a natural isomorphism
  \begin{equation} \label{Frobrep}
    \Hom_R(R_\sigma\otimes_S Q,P) \cong \Hom_S(Q,\Hom_R(R_\sigma, P)).
  \end{equation}
Moreover, $\Hom_R(R_\sigma, P) \cong {_\sigma}P$, meaning the left
$R$-module $P$ viewed as an $S$-module via
the map $\sigma \colon S \to R$.
Since $Q=(R / ReR) \otimes_R P \cong P / ReP$,
we have a short exact sequence
  \[
    0 \longrightarrow  {_\sigma} ReP \longrightarrow {_\sigma} P \longrightarrow Q \longrightarrow 0.
  \]
Since $Q$ is projective as an $S$-module, we have a splitting $\tau
\colon Q \to {_\sigma} P$.
Let $\nu \colon R_\sigma\otimes_S Q \to P$ be the $R$-module
homomorphism corresponding to $\tau$ under
\eqref{Frobrep}, i.e., $\nu(a \otimes q) = a \tau(q)$.

As $S$-modules, we have ${_\sigma}P \cong{_\sigma}ReP \oplus \image \tau$.  Thus, since $\image \nu \supseteq \image \tau$, we have $P = (\image \nu) + ReP$.  Let $p_1,\dotsc,p_m$ generate $P$ as an $R$-module and, for $i=1,\dotsc,m$, write
  \[
    p_i = \nu(u_i) + \sum_{j=1}^n a_{i,j} q_{j},
\]
for some $n \geq 0,\ u_i \in
    R_\sigma\otimes_S Q,\ a_{i,j} \in Re$ and $q_{j} \in eP$.
  Then the map
  \[
    \theta \colon R_\sigma\otimes_S Q \oplus (Re)^n \to P,\quad
    (u,b_1,\dotsc,b_n) \mapsto \nu(u) + \sum_{j=1}^n b_j q_j,
  \]
  is a surjective $R$-module homomorphism.  Since $P$ is projective as an $R$-module, this map splits.  So we have
  \[
    R_\sigma\otimes_S Q \oplus (Re)^n \cong P \oplus (\ker \theta).
  \]
When we apply the functor $S_\pi\otimes_R -$ to the split short exact sequence
$$
0 \longrightarrow \ker \theta \longrightarrow R_\sigma\otimes_S Q \oplus
(Re)^n \stackrel{\theta}{\longrightarrow} P \longrightarrow 0,
$$
we obtain a short exact sequence
$$
0 \longrightarrow S_\pi\otimes_R \ker \theta \longrightarrow S_\pi\otimes_R R_\sigma\otimes_S Q \stackrel{\operatorname{id}_S\otimes\theta}{\longrightarrow} Q \longrightarrow 0.
$$
The composite of $\operatorname{id}_S\otimes\theta$
and the isomorphism $Q \stackrel{\sim}{\rightarrow} S_\pi\otimes_R
R_\sigma \otimes_S Q$ from (\ref{drwho}) is the identity
$\operatorname{id}_Q$,
hence, $\operatorname{id}_S\otimes\theta$ is an isomorphism.
This implies that $S_\pi\otimes_R \ker \theta = 0$.
Finally, we apply Lemma~\ref{quite} to $\ker \theta \in R\pmd$ to
deduce that
$\ker \theta \cong Re \otimes_{eRe} V$ for some $V \in eRe\pmd$.
\end{proof}

\begin{lemma} \label{midexact}
  We have $\ker \psi \subseteq \image \phi$.
\end{lemma}

\begin{proof}
Consider an arbitrary element $[P]-[P'] \in \ker \psi$, where $P, P'
\in R\pmd$.
In $K_0(S)$, we have that $[Q]-[Q'] = 0$ where
$Q := S_\pi \otimes_R P$ and
   $Q':= S_\pi \otimes_R P'$.
By \eqref{slap}, we can assume
  (replacing $P$ and $P'$ by $P\oplus R^n$ and $P' \oplus R^n$ for
  some $n \geq 0$) that $Q \cong Q'$ as $S$-modules.
By the second isomorphism from (\ref{drwho}), we have that
$R_\sigma \otimes_S Q \cong R_{\sigma\circ\pi}\otimes_R P$ and
$R_\sigma\otimes_S Q' \cong R_{\sigma\circ\pi}\otimes_R P'$.
Applying Lemma~\ref{jonfounditeasiertostatethisseparately} twice, we
get $V, V' \in eRe\pmd$ and $n,n'\geq 0$ such that
$$
(Re)^n \oplus R_\sigma\otimes_S Q \cong P \oplus Re\otimes_{eRe} V,
\qquad
(Re)^{n'} \oplus R_\sigma\otimes_S Q'\cong P' \oplus Re\otimes_{eRe} V'.
$$
Since $R_\sigma\otimes_S Q \cong R_\sigma\otimes_S Q'$ as
$R$-modules, we deduce that
$$
[P]-[P'] = (n-n')[Re] -[Re\otimes_{eRe} V] + [Re\otimes_{eRe} V'],
$$
which belongs to
$\image \phi$.
\end{proof}

\begin{proof}[Proof of Theorem~\ref{Grothsplit}]
The fact that (\ref{theseq}) is split exact follows from
Lemmas~\ref{psisurj}, \ref{phiing}, \ref{bacomplex} and
\ref{midexact}.
For the final part, suppose first that $eRe$ and $S$ are both stably finite.
Take $P \in R\pmd$ with $[P] = 0$. Then $[S_\pi\otimes_R P] = 0$ which
implies that 
$S_\pi \otimes_R P  = 0$.
Applying Lemma~\ref{quite}, we deduce that $P \cong Re \otimes_{eRe} V$ for some
$V \in eRe\pmd$. As $[Re\otimes_{eRe} V] = 0$, Lemma~\ref{phiing} 
now gives that
$[V] = 0$. Hence, $V = 0$, so $P=0$ too.

Conversely, suppose that $R$ is stably finite.
Take $V \in eRe\pmd$ with $[V] = 0$. Then $[Re\otimes_{eRe} V] = 0$
which implies that $Re \otimes_{eRe} V = 0$. Then multiply by the
idempotent $e$ to get that $V \cong eRe\otimes_{eRe} V = 0$.
Finally, take $Q \in S\pmd$ with $[Q] = 0$. Then $[S_\pi \otimes_R Q]
= 0$ which implies that $S_\pi \otimes_S Q = 0$. 
Hence, $Q \cong R_\sigma\otimes_R S_\pi \otimes_S Q = 0$. 
\end{proof}

We are going to be working in the remainder of the article 
with (usually monoidal) $\k$-linear categories instead of
rings.  
The data of a $\k$-linear category $\mathcal A$ 
is the same as the data of a {\em locally unital algebra}, i.e., an
associative (but not necessarily unital) $\k$-algebra $A$ equipped with a
system of mutually orthogonal idempotents $\{1_X\:|\:X \in
\mathbb{A}\}$
such that
\begin{equation}\label{dictionary}
A = \bigoplus_{X,Y \in \mathbb{A}} 1_Y A 1_X.
\end{equation}
Under this identification,
$\mathbb{A}$ is the object set of $\mathcal A$,
the idempotent
$1_X$ is the identity endomorphism of object $X$,
$1_Y A 1_X := \Hom_{\mathcal A}(X,Y)$,  and multiplication in $A$ is
induced by composition in $\mathcal A$.
By a module over such a locally unital algebra, we mean a left module $V$
as usual such that $V = \bigoplus_{X \in \mathbb{A}} 1_X V$.
This is
just the same data as a $\k$-linear functor from $\mathcal A$ to
vector spaces. 
Let $A\pmd$ be the category of finitely generated projective
$A$-modules.
Then the Yoneda lemma implies that there is a {\em contravariant}
equivalence of categories
\begin{equation}\label{yon}
\Kar(\mathcal A) \rightarrow A\pmd
\end{equation}
sending an object $X \in \mathcal
A$ to the left ideal $A 1_X$, and a morphism
$f:X\rightarrow Y$ to the homomorphism $A 1_Y \rightarrow A 1_X$
defined by right multiplication.
We get induced a canonical isomorphism
\begin{equation}\label{yoneda}
K_0(\Kar(\mathcal A)) \cong K_0(A),
\end{equation}
where $K_0(A)$ denotes the split Grothendieck group of $A\pmd$.
Providing $A$ is actually unital,
i.e., $\mathcal A$ has only finitely many non-zero objects, 
Theorem~\ref{Grothsplit} can then 
be applied in this situation.

\section{Categorification of symmetric functions}\label{newsec}

It is well known that the ring $\SymZ$ of symmetric functions is
categorified by the representations of the symmetric groups $\S_n$ for all
$n$. In this section, 
we are going to reformulate this classical result in terms of 
monoidal categories. This will give us the opportunity to introduce
 language which will be essential later on.

Let $\SYM$ be the free strict $\k$-linear symmetric monoidal category
generated by one object.
This has a very simple monoidal presentation in terms of
the string calculus for morphisms in strict
monoidal categories; see e.g. \cite[Chapter 2]{TV}. We
represent the horizontal
composition $f \otimes g$
(resp., vertical composition $f \circ g$)
of morphisms $f$ and $g$ 
diagrammatically by drawing $f$ to the left of $g$ (resp., drawing $f$ above
$g$). We denote the unit object by $\unit$
and its identity endomorphism by $1_\unit$.
Then, $\SYM$ is the strict $\k$-linear monoidal category generated by
one object
$\up$ and one morphism
$\mathord{
\begin{tikzpicture}[baseline = -.6mm]
	\draw[->] (0.16,-.16) to (-0.16,.16);
	\draw[->] (-0.16,-.16) to (0.16,.16);
\end{tikzpicture}
}:\up\otimes\up \rightarrow \up\otimes \up
$
subject to the 
relations
\begin{align}\label{symmetric}
\mathord{
\begin{tikzpicture}[baseline = -1mm]
	\draw[->] (0.28,0) to[out=90,in=-90] (-0.28,.6);
	\draw[->] (-0.28,0) to[out=90,in=-90] (0.28,.6);
	\draw[-] (0.28,-.6) to[out=90,in=-90] (-0.28,0);
	\draw[-] (-0.28,-.6) to[out=90,in=-90] (0.28,0);
\end{tikzpicture}
}&=
\mathord{
\begin{tikzpicture}[baseline = -1mm]
	\draw[->] (0.18,-.6) to (0.18,.6);
	\draw[->] (-0.18,-.6) to (-0.18,.6);
\end{tikzpicture}
}\:,
\qquad
\mathord{
\begin{tikzpicture}[baseline = -1mm]
	\draw[<-] (0.45,.6) to (-0.45,-.6);
	\draw[->] (0.45,-.6) to (-0.45,.6);
        \draw[-] (0,-.6) to[out=90,in=-90] (-.45,0);
        \draw[->] (-0.45,0) to[out=90,in=-90] (0,0.6);
\end{tikzpicture}
}
=
\mathord{
\begin{tikzpicture}[baseline = -1mm]
	\draw[<-] (0.45,.6) to (-0.45,-.6);
	\draw[->] (0.45,-.6) to (-0.45,.6);
        \draw[-] (0,-.6) to[out=90,in=-90] (.45,0);
        \draw[->] (0.45,0) to[out=90,in=-90] (0,0.6);
\end{tikzpicture}
}\:.
\end{align}
The objects of $\SYM$ are the tensor powers
$\up^{\otimes n}$ 
of the generating object for $n \in \N$.
There are no non-zero morphisms between $\up^{\otimes m}$ and
$\up^{\otimes n}$ for $m \neq n$.
Moreover, there is an algebra isomorphism
\begin{equation}\label{twangy}
\imath_n:
\k \S_n \stackrel{\sim}{\rightarrow}
\End_{\SYM}(\up^{\otimes n})
\end{equation}
sending the $i$th basic transposition $s_i$ to the crossing of the
$i$th and $(i+1)$th strings (remembering that we number strings
$1,2,\dots$ from right to left).
Thus $\SYM$ assembles the group algebras of all the symmetric groups
into one
convenient package. 

Now we can use the equivalence (\ref{yon}) and the isomorphism
(\ref{yoneda}) to translate the well-known representation
theory of symmetric groups into statements about $\SYM$.
Since we are in characteristic zero, Maschke's theorem
implies that the additive Karoubi envelope
$\Kar(\SYM)$ is a semisimple Abelian category. 
For $\lambda \in \mathcal P$ with $|\la|=n$,
the Specht module $S(\lambda) = (\k \S_n) e_\lambda$ 
corresponds to the indecomposable object
$S_\lambda := \left(\up^{\otimes n}, \imath_n(e_\lambda)\right) \in
\Kar(\SYM)$. 
We set $H_n := S_{(n)}$ and $E_n := S_{(1^n)}$ for short.
Then we see that the classes $\left\{[S_\lambda]\:\big|\:\lambda \in \mathcal
  P\right\}$ give 
a basis for $K_0(\Kar(\SYM))$ as a free $\Z$-module. Moreover,
since taking tensor products of idempotents in $\Kar(\SYM)$
corresponds to the induction product at the level of $\k
\S_n$-modules,
the Littlewood-Richardson rule implies that there is a ring isomorphism
\begin{equation}\label{bargamma}
\gamma:\SymZ \stackrel{\sim}{\rightarrow} K_0(\Kar(\SYM)),
\qquad
s_\lambda \mapsto [S_\lambda],\quad
h_n \mapsto [H_n],\quad
e_n \mapsto [E_n].
\end{equation}
Thus $\SYM$ categorifies the ring of symmetric functions.

In the remainder of the section, we are going to explain how to categorify the comultiplication  (\ref{comultiplication}) on
$\SymZ$.
The usual way to do this is by considering the restriction functors from $\S_n$ to $\S_r
\times \S_{n-r}$ for all $0 \leq r \leq n$.
We are going to formulate the result instead in terms of
a monoidal functor on $\SYM$. 

Given strict $\k$-linear monoidal categories $\mathcal C$ and $\mathcal D$,  we can form their free product $\mathcal C\circledast \mathcal D$ as a strict $\k$-linear monoidal category. This can be defined by a universal property: the category of $\k$-linear monoidal functors $\mathcal C\circledast \mathcal D\to \mathcal{B}$ for any other strict $\k$-linear monoidal category $\mathcal{B}$ is the same as the category of pairs of $\k$-linear monoidal functors $\mathcal C\to \mathcal{B}$ and $\mathcal D\to \mathcal{B}$.  When $\mathcal C$ and $\mathcal D$ are themselves defined by generators and relations, the free product of $\mathcal C$ and $\mathcal D$ may be constructed simply as the strict $\k$-linear monoidal category defined by taking the disjoint union of the given generators and relations of $\mathcal C$ and $\mathcal D$.
The {\em symmetric product}
$\mathcal C \odot \mathcal D$ is the strict $\k$-linear monoidal
category obtained from $\mathcal C \circledast \mathcal D$ by
adjoining isomorphisms $\sigma_{X,Y}:X \otimes Y \stackrel{\sim}{\rightarrow} Y \otimes
X$ 
such that $\sigma_{Y,X} = \sigma_{X,Y}^{-1}$
for each pair of objects $X \in \mathcal C$ and $Y \in \mathcal D$,
subject also to the relations
\begin{align*}
\sigma_{X_1 \otimes X_2, Y} &= (\sigma_{X_1,Y} \otimes 1_{X_2}) \circ
  (1_{X_1} \otimes \sigma_{X_2,Y}),&
\sigma_{X_2,Y} \circ (f \otimes 1_Y)  &= (1_Y \otimes f) \circ
                                        \sigma_{X_1,Y},\\
\sigma_{X, Y_1 \otimes Y_2} &= (1_{Y_1} \otimes \sigma_{X,Y_2}) \circ
(\sigma_{X, Y_1} \otimes 1_{Y_2}),&
\sigma_{X,Y_2} \circ (1_X \otimes g) &= (g \otimes 1_X)\circ \sigma_{X,Y_1}
\end{align*}
for all $X, X_1,X_2 \in\mathcal C, Y, Y_1,Y_2 \in \mathcal D$ and
$f\in \Hom_{\mathcal C}(X_1,X_2), g \in \Hom_{\mathcal
  D}(Y_1,Y_2)$.

The symmetric product $\SYM\odot\SYM$
of two copies of $\SYM$ is
the free strict $\k$-linear symmetric monoidal category generated by two
objects. Diagrammatically it is convenient to use 
different colors,
denoting the
symmetric product instead by $\blue{\SYM}\odot\red{\SYM}$ and using the color
blue (resp., red) for objectsand morphisms in the first (resp., second)
copy of $\SYM$. 
Morphisms may then be represented by linear combinations of string diagrams
colored both blue and red. In these diagrams,
as well as the one-color crossings that are the generating morphisms
of
$\blue{\SYM}$ and $\red{\SYM}$,
we have the additional two-color crossings
\begin{equation}\label{tcc}
\sigma_{\smallblueup, \smallredup} = 
\mathord{
\begin{tikzpicture}[baseline = -.6mm]
	\draw[->,redp] (0.28,-.3) to (-0.28,.3);
	\draw[->,bluep] (-0.28,-.3) to (0.28,.3);
\end{tikzpicture}
}: \blueup\otimes\redup\rightarrow
\redup\otimes\blueup,\qquad
\sigma_{\smallredup, \smallblueup} =
\mathord{
\begin{tikzpicture}[baseline = -.6mm]
	\draw[->,bluep] (0.28,-.3) to (-0.28,.3);
	\draw[->,redp] (-0.28,-.3) to (0.28,.3);
\end{tikzpicture}
}: \redup\otimes\blueup\rightarrow \blueup\otimes\redup,
\end{equation}
which are mutual inverses. The definition of symmetric product gives
braid-like relations allowing all one-color crossings to be commuted across
strings of the other color, for example:
\begin{align}
\mathord{
\begin{tikzpicture}[baseline = -1mm]
	\draw[<-,redp] (0.45,.6) to (-0.45,-.6);
	\draw[->,bluep] (0.45,-.6) to (-0.45,.6);
        \draw[-,redp] (0,-.6) to[out=90,in=-90] (-.45,0);
        \draw[->,redp] (-0.45,0) to[out=90,in=-90] (0,0.6);
\end{tikzpicture}
}
&=
\mathord{
\begin{tikzpicture}[baseline = -1mm]
	\draw[<-,redp] (0.45,.6) to (-0.45,-.6);
	\draw[->,bluep] (0.45,-.6) to (-0.45,.6);
        \draw[-,redp] (0,-.6) to[out=90,in=-90] (.45,0);
        \draw[->,redp] (0.45,0) to[out=90,in=-90] (0,0.6);
\end{tikzpicture}
}\:,&
\mathord{
\begin{tikzpicture}[baseline = -1mm]
	\draw[<-,bluep] (0.45,.6) to (-0.45,-.6);
	\draw[->,bluep] (0.45,-.6) to (-0.45,.6);
        \draw[-,redp] (0,-.6) to[out=90,in=-90] (-.45,0);
        \draw[->,redp] (-0.45,0) to[out=90,in=-90] (0,0.6);
\end{tikzpicture}
}
&=
\mathord{
\begin{tikzpicture}[baseline = -1mm]
	\draw[<-,bluep] (0.45,.6) to (-0.45,-.6);
	\draw[->,bluep] (0.45,-.6) to (-0.45,.6);
        \draw[-,redp] (0,-.6) to[out=90,in=-90] (.45,0);
        \draw[->,redp] (0.45,0) to[out=90,in=-90] (0,0.6);
\end{tikzpicture}
}\:.\label{braidlike}
\end{align}

For $0 \leq r \leq n$ let $\mathcal{P}_{r,n}$ denote the set of size
$\binom{n}{r}$ consisting of 
tuples $\lambda =
(\lambda_1,\dots,\lambda_r) \in \Z^r$ such that $n-r \geq \lambda_1 \geq \cdots \geq \lambda_r
\geq 0$. 
Let
$\min_{r,n}$ (resp., $\max_{r,n}$) be the element $\lambda \in \mathcal{P}_{r,n}$
with $\lambda_1=\cdots=\lambda_r=0$ (resp., $\lambda_1=\cdots=\lambda_r=n-r$).
For any $\lambda \in \mathcal{P}_{r,n}$, we let
\begin{equation}
\up^{\otimes \lambda} := \blueup^{\otimes (n-r-\lambda_1)} \otimes \redup
\otimes \blueup^{\otimes (\lambda_1-\lambda_{2})} \otimes \redup \otimes
\cdots
\otimes \redup \otimes \blueup^{\otimes \lambda_r} \in \blue{\SYM}\odot\red{\SYM};
\end{equation}
in particular, 
$\up^{\otimes \min_{r,n}} = \blueup^{\otimes (n-r)}
\otimes \redup^{\otimes r}$
and $\up^{\otimes \max_{r,n}} = 
\redup^{\otimes r}\otimes \blueup^{\otimes (n-r)}$.
In this way, $\mathcal{P}_{r,n}$ labels
the objects of $\blue{\SYM}\odot\red{\SYM}$ obtained by tensoring $r$ 
generators $\redup$ and $(n-r)$ generators $\blueup$ in
some order.
We denote
the identity endomorphism of $\up^{\otimes \lambda}$ simply by
$1_\lambda$. There is also a unique isomorphism
\begin{equation}\label{babysig}
\sigma_\lambda:   \up^{\otimes \lambda}
\stackrel{\sim}{\rightarrow}
\up^{\otimes \min_{r,n}}\end{equation}
whose
diagram only involves crossings of the form
$\begin{tikzpicture}[baseline = -1mm]
	\draw[->,bluep] (0.18,-.2) to (-0.18,.2);
	\draw[->,redp] (-0.18,-.2) to (0.18,.2);
\end{tikzpicture}\:$; in particular,
$\sigma_{\min_{r,n}} = 1_{\min_{r,n}}$.
To make sense of these definitions,
one can represent an element of
$\mathcal{P}_{r,n}$
by a Young diagram 
with $\lambda_i$ boxes on its $i$th row 
drawn inside an $r \times (n-r)$-rectangle. Then
 $\up^{\otimes \lambda}$ may be seen
by walking southwest along the boundary of the diagram,
recording $\blueup$ for a horizontal step and $\redup$ for a vertical one; 
for example, $(3,3,2,0,0) \in \mathcal P_{5,9}$ is
$$
\lambda = \begin{tikzpicture}[baseline=-6mm]
\draw[-] (0,0) to (.6,0);
\draw[-] (.8,0) to (.8,-1) to (0,-1);
\draw[-] (0,-.6) to (0,0);
\draw[-] (0,-.2) to (.6,-0.2);
\draw[-] (0,-.4) to (.6,-0.4);
\draw[-] (0,-.6) to (.4,-0.6);
\draw[-] (.2,-.6) to (.2,0);
\draw[-] (.4,-.6) to (.4,0);
\draw[-] (.6,-.4) to (.6,0);
\draw[-,thick,blue] (.59,0) to (.807,0);
\draw[-,thick,red] (0.6,-.4) to (.6,0.0138);
\draw[-,thick,blue] (.4,-.4) to (.6138,-.4);
\draw[-,thick,red] (0.4,-.3862) to (.4,-.6);
\draw[-,thick,blue] (.4138,-.6) to (0,-.6);
\draw[-,thick,red] (0,-1.009) to (0,-.5862);
\end{tikzpicture}\:,
\qquad 
\up^{\otimes \lambda} = \blueup\otimes
\redup\otimes\redup\otimes
\blueup\otimes\redup\otimes\blueup\otimes\blueup\otimes\redup\otimes\redup,
\qquad
\sigma_\lambda = 
\begin{tikzpicture}[baseline=4mm]
\draw[<-,bluep] (0,1) to (0,0);
\draw[<-,bluep] (0.25,1) to (0.75,0);
\draw[<-,bluep] (.5,1) to (1.25,0);
\draw[<-,bluep] (.75,1) to (1.5,0);
\draw[<-,redp] (1,1) to (0.25,0);
\draw[<-,redp] (1.25,1) to (0.5,0);
\draw[<-,redp] (1.5,1) to (1,0);
\draw[<-,redp] (1.75,1) to (1.75,0);
\draw[<-,redp] (2,1) to (2,0);
\end{tikzpicture}\:.
$$

We will often identify the group algebra $\k \S_r \otimes_\k \k
\S_{n-r}$ of $\S_r \times \S_{n-r}$ with a
subalgebra of $\k
\S_n$ so that $s_i \otimes 1 \leftrightarrow s_i$ and
$1 \otimes s_j \leftrightarrow s_{r+j}$.
There is an algebra isomorphism
\begin{equation}\label{doublei}
\imath_{r,n}:\k \S_r \otimes_\k \k \S_{n-r} \stackrel{\sim}{\rightarrow}
\End_{\blue{\SYM}\odot\red{\SYM}}\left(\blueup^{\otimes (n-r)} \otimes
\redup^{\otimes r}\right)
\end{equation}
sending $s_i = s_i\otimes 1$ to the crossing of the $i$th
and $(i+1)$th red strings
and $s_{r+j} = 1\otimes s_j$ to the crossing of
the $j$th and $(j+1)$th blue strings.
Combining this isomorphism with the elements 
$\left
\{\sigma^{-1}_\lambda \circ \sigma_\mu\:\big|\lambda,\mu
\in \mathcal P_{r,n}\right\}$, which
give the matrix units, we see that
\begin{equation}\label{matrix}
\End_{\Add(\blue{\SYM}\odot\red{\SYM})}\left((\blueup\oplus\redup)^{\otimes
  n}\right)
\cong \bigoplus_{r=0}^n \operatorname{Mat}_{\binom{n}{r}}\left(\k \S_{r}
\otimes_\k \k \S_{n-r}\right).
\end{equation}
Using (\ref{yon})--(\ref{yoneda}) too, we conclude 
that
$K_0(\Kar(\blue{\SYM}\odot\red{\SYM})) \cong \SymZ \otimes_{\Z} \SymZ$. An explicit isomorphism
is given by the composition
\begin{equation*}
\SymZ\otimes_{\Z}\SymZ\stackrel{\gamma\otimes\gamma}{\longrightarrow}
K_0(\Kar(\blue{\SYM}))\otimes_{\Z}K_0(\Kar(\red{\SYM}))
\stackrel{\epsilon}{\longrightarrow}
K_0(\Kar(\blue{\SYM}\odot\red{\SYM}))
\end{equation*}
where the second map $\epsilon$ is induced by the inclusions of
$\blue{\SYM}$ and $\red{\SYM}$ into $\blue{\SYM}\odot\red{\SYM}$.

Now we are ready to define a strict $\k$-linear monoidal functor
\begin{equation}\label{monoo}
\Delta:
\SYM \rightarrow \Add(\blue{\SYM}\odot\red{\SYM})
\end{equation}
that sends the generating object $\up$ to $\blueup\oplus\redup$,
and is defined on 
the generating morphism by
\begin{equation}\label{toom}
\mathord{
\begin{tikzpicture}[baseline = -.6mm]
	\draw[->] (0.28,-.3) to (-0.28,.3);
	\draw[->] (-0.28,-.3) to (0.28,.3);
\end{tikzpicture}
}\mapsto
\mathord{
\begin{tikzpicture}[baseline = -.6mm]
	\draw[->,bluep] (0.28,-.3) to (-0.28,.3);
	\draw[->,bluep] (-0.28,-.3) to (0.28,.3);
\end{tikzpicture}
}+
\mathord{
\begin{tikzpicture}[baseline = -.6mm]
	\draw[->,redp] (0.28,-.3) to (-0.28,.3);
	\draw[->,redp] (-0.28,-.3) to (0.28,.3);
\end{tikzpicture}
}
+
\mathord{
\begin{tikzpicture}[baseline = -.6mm]
	\draw[->,redp] (0.28,-.3) to (-0.28,.3);
	\draw[->,bluep] (-0.28,-.3) to (0.28,.3);
\end{tikzpicture}
}+
\mathord{
\begin{tikzpicture}[baseline = -.6mm]
	\draw[->,bluep] (0.28,-.3) to (-0.28,.3);
	\draw[->,redp] (-0.28,-.3) to (0.28,.3);
\end{tikzpicture}
}.
\end{equation}
The right-hand side of this, which is a $4\times
4$ matrix in
$\End_{\Add(\blue{\SYM}\odot\red{\SYM})}(\blueup\otimes\blueup\oplus\blueup\otimes\redup\oplus\redup\otimes\blueup\oplus\redup\otimes\redup)$,
is the morphism defining the
symmetric braiding on
the object $\blueup\oplus\redup$ of $\Add(\blue{\SYM}\odot\red{\SYM})$
with respect to its canonical symmetric monoidal structure as the additive
envelope of the symmetric monoidal category $\blue{\SYM}\odot\red{\SYM}$.
The fact that $\Delta$ is well defined is immediate from the universal
property of $\SYM$ as the free symmetric monoidal category on one
object; alternatively, one can directly verify that the defining
relations (\ref{symmetric}) are satisfied.
To compute $\Delta$ 
on a more general diagram $D$, one just
has to sum over all diagrams obtained from $D$ by
coloring the strings red or blue in all possible ways.

\begin{remark}\label{alts}
Similarly, there is a monoidal functor
$\SYM
\rightarrow
\Add\left(
\blue{\SYM} \odot
\red{\SYM}\odot
\green{\SYM}\right)$
to the triple symmetric product
which sends $\up$ to $\blueup\oplus\redup\oplus\greenup$.
Identifying
$\blue{\SYM}\odot
\red{\SYM}\odot \green{\SYM}$ with
$(\blue{\SYM} \odot \red{\SYM})\odot \green{\SYM}$
and
$\blue{\SYM} \odot (\red{\SYM}\odot \green{\SYM})$,
this agrees with both of the compositions
$\left(\Delta\odot\operatorname{Id}\right)
\circ
\Delta$
and
$\left(\operatorname{Id} \odot\Delta\right) \circ
\Delta$. In other words, the categorical
comultiplication is coassociative.
\end{remark}

The functor $\Delta$ extends to a
monoidal functor
$\Delta:\Kar(\SYM)\rightarrow \Kar(\blue{\SYM}\odot\red{\SYM})$,
which in turn
induces 
$[\Delta]:K_0(\Kar(\SYM))\rightarrow
K_0(\Kar(\blue{\SYM}\odot\red{\SYM}))$. Note that $[\Delta]$ is
automatically a ring homomorphism; the analogous statement in the more
traditional approach via restriction functors requires an application
of the Mackey theorem at this point.
We claim moreover that 
\begin{equation}\label{shower}
\begin{diagram}
\node{\SymZ}\arrow{s,l,A,J}{\gamma}\arrow{e,t}{\delta}\node{\SymZ\otimes_\Z\SymZ\:}\arrow{s,r,A,J}{\epsilon\circ\gamma\otimes\gamma}\\
\node{K_0(\Kar(\SYM))}\arrow{e,b}{[\Delta]}\node{K_0(\Kar(\blue{\SYM}
  \odot \red{\SYM}))}
\end{diagram}
\end{equation}
 commutes, i.e.,
$\Delta$ categorifies the comultiplication $\delta$ on $\SymZ$.
This is a consequence of the following theorem, bearing in mind that the
complete symmetric functions $h_n$ generate $\SymZ$.

\begin{theorem}\label{trivial}
For each $n \geq 0$, we have that 
\begin{align}
\Delta(H_n) &\cong \bigoplus_{r=0}^n \blue{H_{n-r}}
\otimes \red{H_r},&
\Delta(E_n) &\cong \bigoplus_{r=0}^n \blue{E_{n-r}}
\otimes \red{E_r}.
\end{align}
\end{theorem}

\begin{proof}
For the isomorphism involving $H_n$,
it suffices to show that 
the idempotents 
$\Delta(\imath_n(e_{(n)}))$ and 
$\sum_{r=0}^n
\imath_{r,n}(e_{(r)}\otimes e_{(n-r)})$ which
define
the objects $\Delta(H_n)$ and 
$\bigoplus_{r=0}^n \blue{H_{n-r}}
\otimes \red{H_r}$ 
are conjugate.
Thus, we need to construct morphisms $u$ and $v$ in
$\Kar(\blue{\SYM}\odot\red{\SYM})$
such that $
u \circ v = 
\Delta(\imath_n(e_{(n)}))$ and 
$v \circ u = \sum_{r=0}^n
\imath_{r,n}(e_{(r)}\otimes e_{(n-r)})$.
To do this, notice 
for any $\lambda,\mu \in \mathcal P_{r,n}$
that 
$$
1_\mu \circ \Delta(\imath_n(e_{(n)})) \circ 1_\lambda =
\binom{\,n\,}{\,r\,}^{-1}\sigma_\mu^{-1} \circ 
\imath_{r,n}(e_{(r)}\otimes e_{(n-r)}) \circ \sigma_\lambda.
$$
It follows that $\Delta(\imath_n(e_{(n)})) = u \circ v$ where
\begin{align*}
u &:=\sum_{r=0}^n \binom{\,n\,}{\,r\,}^{-1}\sum_{\mu \in \mathcal P_{r,n}}
    \sigma_\mu^{-1} \circ 
\imath_{r,n}(e_{(r)}\otimes e_{(n-r)}),&
v &:=\sum_{r=0}^n \sum_{\lambda \in \mathcal P_{r,n}}
         \imath_{r,n}(e_{(r)}\otimes e_{(n-r)})\circ \sigma_\lambda.
\end{align*}
Finally it is easy to check that 
$v \circ u = \sum_{r=0}^n \imath_{r,n}(e_{(r)}\otimes e_{(n-r)})$.

To establish the isomorphism involving $E_n$,
one needs to show instead that there are morphisms $u$ and $v$ such that
$u \circ v = 
\Delta(\imath_n(e_{(1^n)}))$ and 
$v \circ u = \sum_{r=0}^n
\imath_{r,n}(e_{(1^r)}\otimes e_{(1^{n-r})})$.
These are given by similar formulae to the above,
replacing $e_{(m)}$ by $e_{(1^m)}$ and
$\sigma_\nu$ 
by $(-1)^{|\nu|}\sigma_\nu$ everywhere.
\end{proof}

\section{The degenerate affine Hecke category}\label{sdha}

The {\em degenerate affine Hecke algebra} $\dH_n$ is the
vector space $\k[x_1,\dots,x_n] \otimes_\k \k \S_n$
viewed as an associative algebra with
multiplication defined so that $\k[x_1,\dots,x_n]$ and $\k \S_n$
are subalgebras, and in addition 
$s_i f = s_i(f) \:s_i + \partial_i(f)$
for $f \in \k[x_1,\dots,x_n]$ and $i=1,\dots,n-1$,
where $\partial_i$ is the {\em Demazure operator}
\begin{equation}\label{demaz}
\partial_i(f) := \frac{f - s_i(f)}{x_{i+1}-x_i}.
\end{equation}
Also recall that the {\em center} of $AH_n$ is the subalgebra  $\Sym_n
:= \k[x_1,\dots,x_n]^{\S_n}$ 
of symmetric  polynomials; see e.g. \cite[Theorem 3.3.1]{Kbook}. The algebra $AH_n$ is finitely generated as a $\Sym_n$-module.
The following theorem was proved by Khovanov in \cite{K}; its proof uses the assumption that $\k$ is of characteristic zero in an essential way.

\begin{theorem}\label{starting}
The inclusion $\k \S_n \hookrightarrow AH_n$ induces an isomorphism
$K_0(\k \S_n) \stackrel{\sim}{\rightarrow} K_0(AH_n)$.
More generally, the same assertion holds 
when $\k \S_n$ is replaced with $\k \S_{n_1}\otimes_\k\cdots\otimes_\k
\k \S_{n_r}$ and $AH_n$ is replaced with $AH_{n_1}\otimes_\k \cdots\otimes_\k
AH_{n_r} \otimes_\k B$ for any $n_1,\dots,n_r \geq 0$ and any
polynomial algebra $B$ (possibly of infinite rank).
\end{theorem}

\begin{proof}
This is explained in \cite[Section 5.2]{K};
see especially \cite[(40)]{K}\footnote{This is (44) in the preprint version available at \url{https://arxiv.org/abs/1009.3295v1}.}.
The argument in {\em
  loc.\ cit.} depends ultimately
on a result of Quillen \cite[Theorem 7]{Q}. In order to be able to
apply Quillen's result, one needs to know that the degenerate affine Hecke algebra $AH_n$
is a filtered deformation of the smash product $\k[x_1,\dots,x_n] \#
\S_n$, which is a positively graded algebra with 
degree zero component given by the semisimple algebra $\k \S_n$.
When $B$ is of infinite rank, one needs to know also that taking 
$K_0$ commutes with direct limits \cite[Theorem 1.2.5]{Ros94}.
\end{proof}

The first part of this theorem implies that one can categorify the ring $\SymZ$ using the algebra $AH_n$ in
place of $\k \S_n$. Of course, we are going to translate this into the language of monoidal categories.
Let $\mathcal{AH}$ be the strict $\k$-linear monoidal category
obtained from the category $\SYM$ from the previous section by
adjoining an additional generating morphism 
$\mathord{
\begin{tikzpicture}[baseline = -1mm]
	\draw[->] (0.08,-.2) to (0.08,.2);
      \node at (0.08,0) {$\dot$};
\end{tikzpicture}
}
:\up\rightarrow\up
$
subject to the additional relations
\begin{align}\label{heckea}
\mathord{
\begin{tikzpicture}[baseline = -1mm]
	\draw[<-] (0.25,.3) to (-0.25,-.3);
	\draw[->] (0.25,-.3) to (-0.25,.3);
     \node at (-0.12,-0.145) {$\dot$};
\end{tikzpicture}
}
&=
\mathord{
\begin{tikzpicture}[baseline = -1mm]
	\draw[<-] (0.25,.3) to (-0.25,-.3);
	\draw[->] (0.25,-.3) to (-0.25,.3);
     \node at (0.12,0.135) {$\dot$};
\end{tikzpicture}}
+\:\mathord{
\begin{tikzpicture}[baseline = -1mm]
 	\draw[->] (0.08,-.3) to (0.08,.3);
	\draw[->] (-0.28,-.3) to (-0.28,.3);
\end{tikzpicture}
}
\:,&
\mathord{
\begin{tikzpicture}[baseline = -1mm]
	\draw[<-] (0.25,.3) to (-0.25,-.3);
	\draw[->] (0.25,-.3) to (-0.25,.3);
     \node at (-0.12,0.135) {$\dot$};
\end{tikzpicture}
}
&=
\mathord{
\begin{tikzpicture}[baseline = -1mm]
	\draw[<-] (0.25,.3) to (-0.25,-.3);
	\draw[->] (0.25,-.3) to (-0.25,.3);
     \node at (0.12,-0.145) {$\dot$};
\end{tikzpicture}}
+\:\mathord{
\begin{tikzpicture}[baseline = -1mm]
 	\draw[->] (0.08,-.3) to (0.08,.3);
	\draw[->] (-0.28,-.3) to (-0.28,.3);
\end{tikzpicture}
}
\:.
\end{align}
In fact, in the presence of the quadratic relation in $\SYM$, the two relations in (\ref{heckea}) are equivalent.
We denote the $a$th power of
$\mathord{
\begin{tikzpicture}[baseline = -1.3mm]
	\draw[->] (0.08,-.2) to (0.08,.2);
      \node at (0.08,0) {$\dot$};
\end{tikzpicture}
}$
under vertical composition
by
labeling the dot with the multiplicity $a \in \N$.
Just like for $\SYM$, there are no non-zero morphisms between $\up^{\otimes m}$ and
$\up^{\otimes n}$ for $m \neq n$. Moreover, replacing (\ref{twangy}),
there is an algebra
isomorphism
\begin{equation}
\imath_n:AH_n \stackrel{\sim}{\rightarrow} \End_{\AH}(\up^{\otimes n})
\end{equation}
sending $s_i$ to the crossing of the $i$th and $(i+1)$th strings and
$x_j$ to the dot on the $j$th string. Using
(\ref{yon})--(\ref{yoneda}) and Theorem~\ref{starting}, we deduce that the
canonical monoidal embedding $\SYM \rightarrow \AH$ induces a ring
isomorphism
$K_0(\Kar(\SYM)) \stackrel{\sim}{\rightarrow} K_0(\Kar(\AH))$.
Thus, we can reformulate (\ref{bargamma}): there is a ring isomorphism
\begin{equation}\label{gamma}
\gamma:\SymZ\stackrel{\sim}{\rightarrow}
K_0(\Kar(\AH)),\qquad
s_\lambda \mapsto [S_\lambda],\quad h_n \mapsto [H_n], \quad e_n \mapsto [E_n],
\end{equation}
viewing $S_\lambda, H_n$ and $E_n$ now as objects of $\Kar(\AH)$.

The next obvious 
question is whether the monoidal functor $\Delta$
from (\ref{monoo}) can be upgraded from $\SYM$ to $\AH$ too.
To do this, it turns out that we need to localize.

Consider the symmetric product $\blue{\AH}\odot\red{\AH}$. This is
generated by the objects and morphisms from two copies of
$\AH$, one drawn in blue and the other in red, plus the additional
two-color crossings as in (\ref{tcc}). As well as (\ref{braidlike}),
dots of one color commute across strings of the other:
\begin{align}\label{dotlike}
\mathord{
\begin{tikzpicture}[baseline = -1mm]
	\draw[<-,redp] (0.25,.3) to (-0.25,-.3);
	\draw[->,bluep] (0.25,-.3) to (-0.25,.3);
     \node at (-0.12,0.135) {$\blue{\dot}$};
\end{tikzpicture}
}
&=
\mathord{
\begin{tikzpicture}[baseline = -1mm]
	\draw[<-,redp] (0.25,.3) to (-0.25,-.3);
	\draw[->,bluep] (0.25,-.3) to (-0.25,.3);
     \node at (0.12,-0.135) {$\blue{\dot}$};
\end{tikzpicture}
}\,&
\mathord{
\begin{tikzpicture}[baseline = -1mm]
	\draw[<-,redp] (0.25,.3) to (-0.25,-.3);
	\draw[->,bluep] (0.25,-.3) to (-0.25,.3);
     \node at (0.12,0.135) {$\red{\dot}$};
\end{tikzpicture}
}
&=
\mathord{
\begin{tikzpicture}[baseline = -1mm]
	\draw[<-,redp] (0.25,.3) to (-0.25,-.3);
	\draw[->,bluep] (0.25,-.3) to (-0.25,.3);
     \node at (-0.12,-0.135) {$\red{\dot}$};
\end{tikzpicture}
}\:.
\end{align}
Given a diagram $D$ representing a morphism in
$\blue{\AH}\odot\red{\AH}$
and two generic
points in this diagram, one on a red string and the other on a blue
string, we will denote the morphism represented by
($D$ with an extra
dot at the red point)
$-$ ($D$ with an extra dot at the blue point)
by joining the two points with
a dotted line;
this line may pass
willy nilly through
other strings in the diagram as needed. For example:
\begin{align}\label{dashit}
\mathord{
\begin{tikzpicture}[baseline = -1mm]
 	\draw[->,redp] (0.18,-.4) to (0.18,.4);
	\draw[->,bluep] (-0.38,-.4) to (-0.38,.4);
	\draw[-,dotted] (-0.38,0.01) to (0.18,0.01);
     \node at (0.18,0) {$\dot$};
     \node at (-0.38,0) {$\dot$};
\end{tikzpicture}
}&= \:\;
\mathord{
\begin{tikzpicture}[baseline = -1mm]
 	\draw[->,redp] (0.18,-.4) to (0.18,.4);
	\draw[->,bluep] (-0.38,-.4) to (-0.38,.4);
     \node at (0.18,0) {$\red{\dot}$};
\end{tikzpicture}
}-
\mathord{
\begin{tikzpicture}[baseline = -1mm]
 	\draw[->,redp] (0.18,-.4) to (0.18,.4);
	\draw[->,bluep] (-0.38,-.4) to (-0.38,.4);
     \node at (-0.38,0) {$\blue{\dot}$};
\end{tikzpicture}}
\:,&
\mathord{
\begin{tikzpicture}[baseline = -1mm]
 	\draw[->,bluep] (0.18,-.4) to (0.18,.4);
	\draw[->,redp] (-0.38,-.4) to (-0.38,.4);
	\draw[-,dotted] (-0.38,0.01) to (0.18,0.01);
     \node at (0.18,0) {$\dot$};
     \node at (-0.38,0) {$\dot$};
\end{tikzpicture}
}&=
\mathord{
\begin{tikzpicture}[baseline = -1mm]
 	\draw[->,bluep] (0.18,-.4) to (0.18,.4);
	\draw[->,redp] (-0.38,-.4) to (-0.38,.4);
     \node at (-0.38,0) {$\red{\dot}$};
\end{tikzpicture}
}\:\;-\;\:
\mathord{
\begin{tikzpicture}[baseline = -1mm]
 	\draw[->,bluep] (0.18,-.4) to (0.18,.4);
	\draw[->,redp] (-0.38,-.4) to (-0.38,.4);
     \node at (0.18,0) {$\blue{\dot}$};
\end{tikzpicture}}=
\mathord{
\begin{tikzpicture}[baseline = -1mm]
	\draw[->,redp] (0.28,0) to[out=90,in=-90] (-0.28,.6);
	\draw[->,bluep] (-0.28,0) to[out=90,in=-90] (0.28,.6);
	\draw[-,bluep] (0.28,-.6) to[out=90,in=-90] (-0.28,0);
	\draw[-,redp] (-0.28,-.6) to[out=90,in=-90] (0.28,0);
	\draw[-,dotted] (-0.28,0.01) to (0.28,0.01);
     \node at (-0.28,0) {$\dot$};
     \node at (0.28,0) {$\dot$};
\end{tikzpicture}
}\:.
\end{align}
Let
$\blue{\AH} \;\overline{\odot}\; \red{\AH}$ be the strict
$\k$-linear monoidal category obtained from
$\blue{\AH}\odot\red{\AH}$ by localizing at
$\mathord{
\begin{tikzpicture}[baseline = -1mm]
 	\draw[->,redp] (0.18,-.2) to (0.18,.25);
	\draw[->,bluep] (-0.38,-.2) to (-0.38,.25);
	\draw[-,dotted] (-0.38,0.01) to (0.18,0.01);
     \node at (0.18,0) {$\dot$};
     \node at (-0.38,0.01) {$\dot$};
\end{tikzpicture}}
$.
This means that we adjoin a two-sided inverse to this morphism, which
we denote as a dumbbell
\begin{equation}\label{dumb}
\mathord{
\begin{tikzpicture}[baseline = -1mm]
 	\draw[->,redp] (0.18,-.4) to (0.18,.4);
	\draw[->,bluep] (-0.38,-.4) to (-0.38,.4);
	\draw[-] (-0.38,0.01) to (0.18,0.01);
     \node at (0.18,0) {$\dot$};
     \node at (-0.38,0) {$\dot$};
\end{tikzpicture}
}:=
\left(\mathord{
\begin{tikzpicture}[baseline = -1mm]
 	\draw[->,redp] (0.18,-.4) to (0.18,.4);
	\draw[->,bluep] (-0.38,-.4) to (-0.38,.4);
	\draw[-,dotted] (-0.38,0.01) to (0.18,0.01);
     \node at (0.18,0) {$\dot$};
     \node at (-0.38,0) {$\dot$};
\end{tikzpicture}
}\right)^{-1}.
\end{equation}
By the commuting relations, the morphism
$\mathord{
\begin{tikzpicture}[baseline = -1mm]
 	\draw[->,bluep] (0.18,-.2) to (0.18,.25);
	\draw[->,redp] (-0.38,-.2) to (-0.38,.25);
	\draw[-,dotted] (-0.38,0.01) to (0.18,0.01);
     \node at (0.18,0) {$\dot$};
     \node at (-0.38,0.01) {$\dot$};
\end{tikzpicture}}
$ is also invertible
in
$\blue{\AH} \;\overline{\odot}\; \red{\AH}$,
with two-sided inverse
\begin{equation*}
\mathord{
\begin{tikzpicture}[baseline = -1mm]
 	\draw[->,bluep] (0.18,-.4) to (0.18,.4);
	\draw[->,redp] (-0.38,-.4) to (-0.38,.4);
	\draw[-] (-0.38,0.01) to (0.18,0.01);
     \node at (0.18,0) {$\dot$};
     \node at (-0.38,0) {$\dot$};
\end{tikzpicture}
}:=\left(\mathord{
\begin{tikzpicture}[baseline = -1mm]
 	\draw[->,bluep] (0.18,-.4) to (0.18,.4);
	\draw[->,redp] (-0.38,-.4) to (-0.38,.4);
	\draw[-,dotted] (-0.38,0.01) to (0.18,0.01);
     \node at (0.18,0) {$\dot$};
     \node at (-0.38,0) {$\dot$};
\end{tikzpicture}
}\right)^{-1}=
\mathord{
\begin{tikzpicture}[baseline = -1mm]
	\draw[->,redp] (0.28,0) to[out=90,in=-90] (-0.28,.6);
	\draw[->,bluep] (-0.28,0) to[out=90,in=-90] (0.28,.6);
	\draw[-,bluep] (0.28,-.6) to[out=90,in=-90] (-0.28,0);
	\draw[-,redp] (-0.28,-.6) to[out=90,in=-90] (0.28,0);
	\draw[-] (-0.28,0.01) to (0.28,0.01);
     \node at (0.28,0) {$\dot$};
     \node at (-0.28,0) {$\dot$};
\end{tikzpicture}
}\:.
\end{equation*}
We can also introduce more general
dumbbells that cross over other strings:
let
\begin{align*}
\mathord{
\begin{tikzpicture}[baseline = -1mm]
 	\draw[->,redp] (0.38,-.4) to (0.38,.4);
\draw[-,thick,red] (-.1,-.4) to (-.1,.4);
\draw[-,thick,dashed,blue] (-.1,-.4) to (-.1,.4);
	\draw[->,bluep] (-0.58,-.4) to (-0.58,.4);
	\draw[-] (-0.58,0.01) to (0.38,0.01);
     \node at (0.38,0) {$\dot$};
\node at (-.1,-.5) {$\scriptstyle X$};
     \node at (-0.58,0) {$\dot$};
\end{tikzpicture}
}&:= \left(
\mathord{
\begin{tikzpicture}[baseline = -1mm]
 	\draw[->,redp] (0.38,-.4) to (0.38,.4);
\draw[-,thick,red] (-.1,-.4) to (-.1,.4);
\draw[-,thick,dashed,blue] (-.1,-.4) to (-.1,.4);
	\draw[->,bluep] (-0.58,-.4) to (-0.58,.4);
	\draw[-,dotted] (-0.58,0.01) to (0.38,0.01);
     \node at (0.38,0) {$\dot$};
\node at (-.1,-.5) {$\scriptstyle X$};
     \node at (-0.58,0) {$\dot$};
\end{tikzpicture}
}\right)^{-1}\
,&
\mathord{
\begin{tikzpicture}[baseline = -1mm]
 	\draw[->,bluep] (0.38,-.4) to (0.38,.4);
\draw[-,thick,red] (-.1,-.4) to (-.1,.4);
\draw[-,thick,dashed,blue] (-.1,-.4) to (-.1,.4);
	\draw[->,redp] (-0.58,-.4) to (-0.58,.4);
	\draw[-] (-0.58,0.01) to (0.38,0.01);
     \node at (0.38,0) {$\dot$};
\node at (-.1,-.5) {$\scriptstyle X$};
     \node at (-0.58,0) {$\dot$};
\end{tikzpicture}
}&:= \left(
\mathord{
\begin{tikzpicture}[baseline = -1mm]
 	\draw[->,bluep] (0.38,-.4) to (0.38,.4);
\draw[-,thick,red] (-.1,-.4) to (-.1,.4);
\draw[-,thick,dashed,blue] (-.1,-.4) to (-.1,.4);
	\draw[->,redp] (-0.58,-.4) to (-0.58,.4);
	\draw[-,dotted] (-0.58,0.01) to (0.38,0.01);
     \node at (0.38,0) {$\dot$};
\node at (-.1,-.5) {$\scriptstyle X$};
     \node at (-0.58,0) {$\dot$};
\end{tikzpicture}
}\right)^{-1}
\end{align*}
for any object $X \in \blue{\AH}\;\overline{\odot}\;\red{\AH}$, where the two-colored vertical line represents $1_X$.
To see that this makes sense, one needs to prove that this morphism is indeed
invertible; this follows easily from the
commuting relations. For example, if $X
= \blueup\otimes\redup\otimes\blueup$ then
$$
\mathord{
\begin{tikzpicture}[baseline = -1mm]
 	\draw[->,redp] (0.86,-.4) to (0.86,.4);
 	\draw[->,bluep] (0.38,-.4) to (0.38,.4);
\draw[->,redp] (-.1,-.4) to (-.1,.4);
	\draw[->,bluep] (-0.58,-.4) to (-0.58,.4);
	\draw[->,bluep] (-1.06,-.4) to (-1.06,.4);
	\draw[-,dotted] (-1.06,0.01) to (0.86,0.01);
     \node at (0.86,0) {$\dot$};
     \node at (-1.06,0) {$\dot$};
\end{tikzpicture}
}
=
\mathord{
\begin{tikzpicture}[baseline = -1mm]
 	\draw[-,redp] (0.86,-.8) to[out=90,in=-90] (-0.1,0);
 	\draw[->,redp] (-0.1,-0) to[out=90,in=-90] (0.86,.8);
 	\draw[-,bluep] (0.38,-.8) to[out=90,in=-90] (0.86,0);
 	\draw[->,bluep] (0.86,0) to[out=90,in=-90] (0.38,.8);
\draw[-,redp] (-.1,-.8) to [out=90,in=-90] (-1.06,0);
\draw[->,redp] (-1.06,0) to[out=90,in=-90] (-.1,.8);
	\draw[->,bluep] (0.38,0) to[out=90,in=-90] (-0.58,.8);
	\draw[-,bluep] (-0.58,-.8) to[out=90,in=-90] (0.38,0);
	\draw[->,bluep] (-.58,0) to [out=90,in=-90] (-1.06,.8);
	\draw[-,bluep] (-1.06,-.8) to [in=-90,out=90] (-.58,0);
	\draw[-,dotted] (-.58,0.01) to (-0.1,0.01);
     \node at (-0.58,0) {$\dot$};
     \node at (-.1,0) {$\dot$};
\end{tikzpicture}
}
\qquad\text{with inverse}\qquad
\mathord{
\begin{tikzpicture}[baseline = -1mm]
 	\draw[->,redp] (0.86,-.4) to (0.86,.4);
 	\draw[->,bluep] (0.38,-.4) to (0.38,.4);
\draw[->,redp] (-.1,-.4) to (-.1,.4);
	\draw[->,bluep] (-0.58,-.4) to (-0.58,.4);
	\draw[->,bluep] (-1.06,-.4) to (-1.06,.4);
	\draw[-] (-1.06,0.01) to (0.86,0.01);
     \node at (0.86,0) {$\dot$};
     \node at (-1.06,0) {$\dot$};
\end{tikzpicture}
}
=
\mathord{
\begin{tikzpicture}[baseline = -1mm]
 	\draw[-,redp] (0.86,-.8) to[out=90,in=-90] (-0.1,0);
 	\draw[->,redp] (-0.1,-0) to[out=90,in=-90] (0.86,.8);
 	\draw[-,bluep] (0.38,-.8) to[out=90,in=-90] (0.86,0);
 	\draw[->,bluep] (0.86,0) to[out=90,in=-90] (0.38,.8);
\draw[-,redp] (-.1,-.8) to [out=90,in=-90] (-1.06,0);
\draw[->,redp] (-1.06,0) to[out=90,in=-90] (-.1,.8);
	\draw[->,bluep] (0.38,0) to[out=90,in=-90] (-0.58,.8);
	\draw[-,bluep] (-0.58,-.8) to[out=90,in=-90] (0.38,0);
	\draw[->,bluep] (-.58,0) to [out=90,in=-90] (-1.06,.8);
	\draw[-,bluep] (-1.06,-.8) to [in=-90,out=90] (-.58,0);
	\draw[-] (-.58,0.01) to (-0.1,0.01);
     \node at (-0.58,0) {$\dot$};
     \node at (-.1,0) {$\dot$};
\end{tikzpicture}
}\:.
$$
Note also that $\blue{\AH}\;\overline{\odot}\;\red{\AH}$ has a
monoidal involution
\begin{equation}\label{Sigma}
\eta:
\blue{\AH}\;\overline{\odot}\;\red{\AH}
\rightarrow \blue{\AH}\;\overline{\odot}\;\red{\AH}
\end{equation}
which is defined on diagrams by switching the colors blue and red
then multiplying by $(-1)^z$ where $z$ is the total
number of dumbbells in the picture.

There are several other useful relations in
$\blue{\AH}\;\overline{\odot}\;\red{\AH}$.
Composing the definition (\ref{dashit}) on the top with the dumbbell, we get
that
\begin{equation*}
\mathord{
\begin{tikzpicture}[baseline = -1mm]
 	\draw[->,redp] (0.18,-.4) to (0.18,.4);
	\draw[->,bluep] (-0.38,-.4) to (-0.38,.4);
	\draw[-] (-0.38,.11) to (0.18,.11);
     \node at (0.18,.1) {$\dot$};
     \node at (-0.38,.1) {$\dot$};
     \node at (0.18,-.15) {$\red{\dot}$};
\end{tikzpicture}
}
=
\mathord{
\begin{tikzpicture}[baseline = -1mm]
 	\draw[->,redp] (0.18,-.4) to (0.18,.4);
	\draw[->,bluep] (-0.38,-.4) to (-0.38,.4);
	\draw[-] (-0.38,0.11) to (0.18,0.11);
     \node at (0.18,0.1) {$\dot$};
     \node at (-0.38,0.1) {$\dot$};
     \node at (-0.38,-.15) {$\blue{\dot}$};
\end{tikzpicture}
}
+\;
\mathord{
\begin{tikzpicture}[baseline = -1mm]
 	\draw[->,redp] (0.18,-.4) to (0.18,.4);
	\draw[->,bluep] (-0.38,-.4) to (-0.38,.4);
\end{tikzpicture}}
\:,
\end{equation*}
which gives a way to teleport dots across dumbbells modulo a correction term.
More generally:
\begin{equation}\label{teleporting}
\mathord{
\begin{tikzpicture}[baseline = -1mm]
 	\draw[->,redp] (0.18,-.4) to (0.18,.4);
	\draw[->,bluep] (-0.38,-.4) to (-0.38,.4);
	\draw[-] (-0.38,.11) to (0.18,.11);
     \node at (0.18,.1) {$\dot$};
     \node at (-0.38,.1) {$\dot$};
     \node at (0.18,-.15) {$\red{\dot}$};
     \node at (0.34,-.15) {$\scriptstyle\red{a}$};
\end{tikzpicture}
}
=
\mathord{
\begin{tikzpicture}[baseline = -1mm]
 	\draw[->,redp] (0.18,-.4) to (0.18,.4);
	\draw[->,bluep] (-0.38,-.4) to (-0.38,.4);
	\draw[-] (-0.38,0.11) to (0.18,0.11);
     \node at (0.18,0.1) {$\dot$};
     \node at (-0.38,0.1) {$\dot$};
     \node at (-0.38,-.15) {$\blue{\dot}$};
     \node at (-0.54,-.15) {$\scriptstyle\blue{a}$};
\end{tikzpicture}
}
+\sum_{\substack{b,c \geq 0 \\ b+c=a-1}}
\mathord{
\begin{tikzpicture}[baseline = -1mm]
 	\draw[->,redp] (0.18,-.4) to (0.18,.4);
	\draw[->,bluep] (-0.38,-.4) to (-0.38,.4);
     \node at (-0.38,0) {$\blue{\dot}$};
     \node at (-0.54,0) {$\scriptstyle\blue{b}$};
     \node at (0.18,0) {$\red{\dot}$};
     \node at (0.34,0) {$\scriptstyle\red{c}$};
\end{tikzpicture}}
\:.
\end{equation}
Dots commute with dumbbells:
\begin{align*}
\mathord{
\begin{tikzpicture}[baseline = -1mm]
 	\draw[->,redp] (0.18,-.4) to (0.18,.4);
	\draw[->,bluep] (-0.38,-.4) to (-0.38,.4);
	\draw[-] (-0.38,.11) to (0.18,.11);
     \node at (0.18,.1) {$\dot$};
     \node at (-0.38,.1) {$\dot$};
     \node at (0.18,-.15) {$\red{\dot}$};
\end{tikzpicture}
}
&=\mathord{
\begin{tikzpicture}[baseline = -1mm]
 	\draw[->,redp] (0.18,-.4) to (0.18,.4);
	\draw[->,bluep] (-0.38,-.4) to (-0.38,.4);
	\draw[-] (-0.38,-.14) to (0.18,-.14);
     \node at (0.18,-.15) {$\dot$};
     \node at (-0.38,-.15) {$\dot$};
     \node at (0.18,.1) {$\red{\dot}$};
\end{tikzpicture}
}\:,
&
\mathord{
\begin{tikzpicture}[baseline = -1mm]
 	\draw[->,redp] (0.18,-.4) to (0.18,.4);
	\draw[->,bluep] (-0.38,-.4) to (-0.38,.4);
	\draw[-] (-0.38,0.11) to (0.18,0.11);
     \node at (0.18,0.1) {$\dot$};
     \node at (-0.38,0.1) {$\dot$};
     \node at (-0.38,-.15) {$\blue{\dot}$};
\end{tikzpicture}
}
&=
\mathord{
\begin{tikzpicture}[baseline = -1mm]
 	\draw[->,redp] (0.18,-.4) to (0.18,.4);
	\draw[->,bluep] (-0.38,-.4) to (-0.38,.4);
	\draw[-] (-0.38,-0.14) to (0.18,-0.14);
     \node at (0.18,-0.15) {$\dot$};
     \node at (-0.38,-0.15) {$\dot$};
     \node at (-0.38,.1) {$\blue{\dot}$};
\end{tikzpicture}
}\:.
\end{align*}
To see this, compose on top and bottom with $\mathord{
\begin{tikzpicture}[baseline = -1mm]
 	\draw[->,redp] (0.18,-.2) to (0.18,.25);
	\draw[->,bluep] (-0.38,-.2) to (-0.38,.25);
	\draw[-,dotted] (-0.38,0.01) to (0.18,0.01);
     \node at (0.18,0) {$\dot$};
     \node at (-0.38,0.01) {$\dot$};
\end{tikzpicture}}
$.
Similarly, different dumbbells commute with each other.
Also, dumbbells commute past two-color crossings:
\begin{align*}
\mathord{
\begin{tikzpicture}[baseline = 0]
	\draw[->,bluep] (0.28,-.3) to (-0.28,.4);
	\draw[->,redp] (-0.28,-.3) to (0.28,.4);
	\draw[-] (-0.15,-.14) to (0.15,-.14);
     \node at (0.15,-.15) {$\dot$};
     \node at (-0.15,-.15) {$\dot$};
\end{tikzpicture}
}
&=
\mathord{
\begin{tikzpicture}[baseline = 0]
	\draw[->,bluep] (0.28,-.3) to (-0.28,.4);
	\draw[->,redp] (-0.28,-.3) to (0.28,.4);
	\draw[-] (-0.13,.21) to (0.13,.21);
     \node at (0.13,.2) {$\dot$};
     \node at (-0.13,.2) {$\dot$};
\end{tikzpicture}
}\:,&
\mathord{
\begin{tikzpicture}[baseline = 0]
	\draw[->,redp] (-0.74,-.3) to (-0.28,.4);
	\draw[->,bluep] (-0.28,-.3) to (-0.74,.4);
	\draw[->,redp] (0.18,-.3) to (0.18,.4);
	\draw[-] (-0.62,.21) to (0.18,.21);
     \node at (-0.62,.2) {$\dot$};
     \node at (0.18,.2) {$\dot$};
\end{tikzpicture}
}
&=
\mathord{
\begin{tikzpicture}[baseline = 0]
	\draw[->,redp] (-0.74,-.3) to (-0.28,.4);
	\draw[->,bluep] (-0.28,-.3) to (-0.74,.4);
	\draw[->,redp] (0.18,-.3) to (0.18,.4);
	\draw[-] (-0.42,-.09) to (0.18,-.09);
     \node at (-0.42,-.1) {$\dot$};
     \node at (0.18,-.1) {$\dot$};
\end{tikzpicture}
}\:,&
\mathord{
\begin{tikzpicture}[baseline = 0]
	\draw[->,bluep] (-0.74,-.3) to (-0.28,.4);
	\draw[->,redp] (-0.28,-.3) to (-0.74,.4);
	\draw[->,redp] (0.18,-.3) to (0.18,.4);
	\draw[-] (-0.4,.21) to (0.18,.21);
     \node at (-0.4,.2) {$\dot$};
     \node at (0.18,.2) {$\dot$};
\end{tikzpicture}
}
&=
\mathord{
\begin{tikzpicture}[baseline = 0]
	\draw[->,bluep] (-0.74,-.3) to (-0.28,.4);
	\draw[->,redp] (-0.28,-.3) to (-0.74,.4);
	\draw[->,redp] (0.18,-.3) to (0.18,.4);
	\draw[-] (-0.6,-.09) to (0.18,-.09);
     \node at (-0.6,-.1) {$\dot$};
     \node at (0.18,-.1) {$\dot$};
\end{tikzpicture}
}\:.
\end{align*}
For one-color crossings, we have the following more complicated
commutation relations:
\begin{align}
\mathord{
\begin{tikzpicture}[baseline = 0]
	\draw[->,bluep] (-0.74,-.3) to (-0.28,.4);
	\draw[->,bluep] (-0.28,-.3) to (-0.74,.4);
	\draw[->,redp] (0.18,-.3) to (0.18,.4);
	\draw[-] (-0.62,.21) to (0.18,.21);
     \node at (-0.62,.2) {$\dot$};
     \node at (0.18,.2) {$\dot$};
\end{tikzpicture}
}
&=
\mathord{
\begin{tikzpicture}[baseline = 0]
	\draw[->,bluep] (-0.74,-.3) to (-0.28,.4);
	\draw[->,bluep] (-0.28,-.3) to (-0.74,.4);
	\draw[->,redp] (0.18,-.3) to (0.18,.4);
	\draw[-] (-0.42,-.09) to (0.18,-.09);
     \node at (-0.42,-.1) {$\dot$};
     \node at (0.18,-.1) {$\dot$};
\end{tikzpicture}
}+
\mathord{
\begin{tikzpicture}[baseline = 0]
	\draw[->,bluep] (-0.74,-.3) to (-0.74,.4);
	\draw[->,bluep] (-0.28,-.3) to (-0.28,.4);
	\draw[->,redp] (0.18,-.3) to (0.18,.4);
	\draw[-] (-0.74,.11) to (0.18,.11);
     \node at (-0.74,.1) {$\dot$};
     \node at (0.18,.1) {$\dot$};
	\draw[-] (-0.28,-.09) to (0.18,-.09);
     \node at (-0.28,-.1) {$\dot$};
     \node at (0.18,-.1) {$\dot$};
\end{tikzpicture}
}
\:,&
\mathord{
\begin{tikzpicture}[baseline = 0]
	\draw[->,bluep] (-0.74,-.3) to (-0.28,.4);
	\draw[->,bluep] (-0.28,-.3) to (-0.74,.4);
	\draw[->,redp] (0.18,-.3) to (0.18,.4);
	\draw[-] (-0.6,-.09) to (0.18,-.09);
     \node at (-0.6,-.1) {$\dot$};
     \node at (0.18,-.1) {$\dot$};
\end{tikzpicture}
}
&=
\mathord{
\begin{tikzpicture}[baseline = 0]
	\draw[->,bluep] (-0.74,-.3) to (-0.28,.4);
	\draw[->,bluep] (-0.28,-.3) to (-0.74,.4);
	\draw[->,redp] (0.18,-.3) to (0.18,.4);
	\draw[-] (-0.4,.21) to (0.18,.21);
     \node at (-0.4,.2) {$\dot$};
     \node at (0.18,.2) {$\dot$};
\end{tikzpicture}
}+
\mathord{
\begin{tikzpicture}[baseline = 0]
	\draw[->,bluep] (-0.74,-.3) to (-0.74,.4);
	\draw[->,bluep] (-0.28,-.3) to (-0.28,.4);
	\draw[->,redp] (0.18,-.3) to (0.18,.4);
	\draw[-] (-0.74,-.09) to (0.18,-.09);
     \node at (-0.74,-.1) {$\dot$};
     \node at (0.18,-.1) {$\dot$};
	\draw[-] (-0.28,.11) to (0.18,.11);
     \node at (-0.28,.1) {$\dot$};
     \node at (0.18,.1) {$\dot$};
\end{tikzpicture}
}
\:.\label{c1}
\end{align}
For example, to prove the first one of these,
one just needs to compose on the top with
$\mathord{
\begin{tikzpicture}[baseline = -1mm]
 	\draw[->,redp] (0.28,-.2) to (0.28,.25);
	\draw[->,bluep] (-0.48,-.2) to (-0.48,.25);
	\draw[->,bluep] (-0.1,-.2) to (-0.1,.25);
	\draw[-,dotted] (-0.48,0.01) to (0.28,0.01);
     \node at (0.28,0) {$\dot$};
     \node at (-0.48,0.01) {$\dot$};
\end{tikzpicture}}
$
and on the bottom with
$\mathord{
\begin{tikzpicture}[baseline = -1mm]
 	\draw[->,redp] (0.28,-.2) to (0.28,.25);
	\draw[->,bluep] (-0.48,-.2) to (-0.48,.25);
	\draw[->,bluep] (-0.1,-.2) to (-0.1,.25);
	\draw[-,dotted] (-0.1,0.01) to (0.28,0.01);
     \node at (0.28,0) {$\dot$};
     \node at (-0.1,0.01) {$\dot$};
\end{tikzpicture}}
$, then apply
the dot-sliding relation from (\ref{heckea}).
Let us also record the mirror images of the last set of relations
under $\eta$:
\begin{align}
\mathord{
\begin{tikzpicture}[baseline = 0]
	\draw[->,redp] (-0.74,-.3) to (-0.28,.4);
	\draw[->,redp] (-0.28,-.3) to (-0.74,.4);
	\draw[->,bluep] (0.18,-.3) to (0.18,.4);
	\draw[-] (-0.62,.21) to (0.18,.21);
     \node at (-0.62,.2) {$\dot$};
     \node at (0.18,.2) {$\dot$};
\end{tikzpicture}
}
&=
\mathord{
\begin{tikzpicture}[baseline = 0]
	\draw[->,redp] (-0.74,-.3) to (-0.28,.4);
	\draw[->,redp] (-0.28,-.3) to (-0.74,.4);
	\draw[->,bluep] (0.18,-.3) to (0.18,.4);
	\draw[-] (-0.42,-.09) to (0.18,-.09);
     \node at (-0.42,-.1) {$\dot$};
     \node at (0.18,-.1) {$\dot$};
\end{tikzpicture}
}-
\mathord{
\begin{tikzpicture}[baseline = 0]
	\draw[->,redp] (-0.74,-.3) to (-0.74,.4);
	\draw[->,redp] (-0.28,-.3) to (-0.28,.4);
	\draw[->,bluep] (0.18,-.3) to (0.18,.4);
	\draw[-] (-0.74,.11) to (0.18,.11);
     \node at (-0.74,.1) {$\dot$};
     \node at (0.18,.1) {$\dot$};
	\draw[-] (-0.28,-.09) to (0.18,-.09);
     \node at (-0.28,-.1) {$\dot$};
     \node at (0.18,-.1) {$\dot$};
\end{tikzpicture}
}
\:,&
\mathord{
\begin{tikzpicture}[baseline = 0]
	\draw[->,redp] (-0.74,-.3) to (-0.28,.4);
	\draw[->,redp] (-0.28,-.3) to (-0.74,.4);
	\draw[->,bluep] (0.18,-.3) to (0.18,.4);
	\draw[-] (-0.6,-.09) to (0.18,-.09);
     \node at (-0.6,-.1) {$\dot$};
     \node at (0.18,-.1) {$\dot$};
\end{tikzpicture}
}
&=
\mathord{
\begin{tikzpicture}[baseline = 0]
	\draw[->,redp] (-0.74,-.3) to (-0.28,.4);
	\draw[->,redp] (-0.28,-.3) to (-0.74,.4);
	\draw[->,bluep] (0.18,-.3) to (0.18,.4);
	\draw[-] (-0.4,.21) to (0.18,.21);
     \node at (-0.4,.2) {$\dot$};
     \node at (0.18,.2) {$\dot$};
\end{tikzpicture}
}-
\mathord{
\begin{tikzpicture}[baseline = 0]
	\draw[->,redp] (-0.74,-.3) to (-0.74,.4);
	\draw[->,redp] (-0.28,-.3) to (-0.28,.4);
	\draw[->,bluep] (0.18,-.3) to (0.18,.4);
	\draw[-] (-0.74,-.09) to (0.18,-.09);
     \node at (-0.74,-.1) {$\dot$};
     \node at (0.18,-.1) {$\dot$};
	\draw[-] (-0.28,.11) to (0.18,.11);
     \node at (-0.28,.1) {$\dot$};
     \node at (0.18,.1) {$\dot$};
\end{tikzpicture}
}
\:.\label{c2}
\end{align}
We have done this in order to stress that the signs are different when
the colors are this way around!
Note also in (\ref{c1})--(\ref{c2}) that the vertical string on the right
hand side could also be drawn on the left hand side; the resulting
relations also hold thanks to the commuting relations.

\begin{theorem}\label{comultah}
There is a strict $\k$-linear monoidal functor
$\Delta:
\AH \rightarrow
\Add\left(\blue{\AH} \;\overline{\odot}\; \red{\AH}\right)$
such that $\up \mapsto \blueup \oplus \redup$ and
\begin{align}\label{com0old}
\mathord{
\begin{tikzpicture}[baseline = -.6mm]
	\draw[->] (0.08,-.3) to (0.08,.3);
      \node at (0.08,0) {$\dot$};
\end{tikzpicture}
}
&\mapsto
\mathord{
\begin{tikzpicture}[baseline = -.6mm]
	\draw[->,bluep] (0.08,-.3) to (0.08,.3);
      \node at (0.08,0) {$\color{blue}\dot$};
\end{tikzpicture}
}
+
\mathord{
\begin{tikzpicture}[baseline = -.6mm]
	\draw[->,redp] (0.08,-.3) to (0.08,.3);
      \node at (0.08,0) {$\color{red}\dot$};
\end{tikzpicture}
}\:,
&\mathord{
\begin{tikzpicture}[baseline = -.6mm]
	\draw[->] (0.28,-.3) to (-0.28,.3);
	\draw[->] (-0.28,-.3) to (0.28,.3);
\end{tikzpicture}
}&\mapsto
\mathord{
\begin{tikzpicture}[baseline = -.6mm]
	\draw[->,bluep] (0.28,-.3) to (-0.28,.3);
	\draw[->,bluep] (-0.28,-.3) to (0.28,.3);
\end{tikzpicture}
}+
\mathord{
\begin{tikzpicture}[baseline = -.6mm]
	\draw[->,redp] (0.28,-.3) to (-0.28,.3);
	\draw[->,redp] (-0.28,-.3) to (0.28,.3);
\end{tikzpicture}
}
+
\mathord{
\begin{tikzpicture}[baseline = -.6mm]
	\draw[->,redp] (0.28,-.3) to (-0.28,.3);
	\draw[->,bluep] (-0.28,-.3) to (0.28,.3);
\end{tikzpicture}
}+
\mathord{
\begin{tikzpicture}[baseline = -.6mm]
	\draw[->,bluep] (0.28,-.3) to (-0.28,.3);
	\draw[->,redp] (-0.28,-.3) to (0.28,.3);
\end{tikzpicture}
}-
\mathord{
\begin{tikzpicture}[baseline = -.6mm]
	\draw[->,redp] (0.28,-.3) to (-0.28,.3);
	\draw[->,bluep] (-0.28,-.3) to (0.28,.3);
	\draw[-] (-0.15,-.16) to (0.15,-.16);
     \node at (0.15,-.17) {$\dot$};
     \node at (-0.15,-.17) {$\dot$};
\end{tikzpicture}
}
+
\mathord{
\begin{tikzpicture}[baseline = -.6mm]
	\draw[->,bluep] (0.28,-.3) to (-0.28,.3);
	\draw[->,redp] (-0.28,-.3) to (0.28,.3);
	\draw[-] (-0.15,-.16) to (0.15,-.16);
     \node at (0.15,-.17) {$\dot$};
     \node at (-0.15,-.17) {$\dot$};
\end{tikzpicture}
}
-\mathord{
\begin{tikzpicture}[baseline =-.6mm]
 	\draw[->,redp] (0.2,-.3) to (0.2,.3);
	\draw[->,bluep] (-0.2,-.3) to (-0.2,.3);
	\draw[-] (-0.2,0.01) to (0.2,0.01);
     \node at (0.2,0.0) {$\dot$};
     \node at (-0.2,0.0) {$\dot$};
\end{tikzpicture}
}+
\mathord{
\begin{tikzpicture}[baseline = -0.6mm]
 	\draw[->,bluep] (0.2,-.3) to (0.2,.3);
	\draw[->,redp] (-0.2,-.3) to (-0.2,.3);
	\draw[-] (-0.2,0.01) to (0.2,0.01);
     \node at (0.2,0.0) {$\dot$};
     \node at (-0.2,0.0) {$\dot$};
\end{tikzpicture}
}\:.
\end{align}
In addition, we have that $\Delta = \eta \circ \Delta$
(extending $\eta$ to the additive
envelope
in the obvious way).
\end{theorem}

\begin{remark}\label{altss}
This categorical comultiplication is coassociative like in Remark~\ref{alts}.
\end{remark}

\begin{proof}[Proof of Theorem~\ref{comultah}]
We just need to check that the
defining relations from (\ref{symmetric}) and (\ref{heckea}) are satisfied in
$\blue{\AH} \;\overline{\odot}\; \red{\AH}$. 
For the quadratic relation, the image of the crossing squared is
\begin{align*}
&\mathord{
\begin{tikzpicture}[baseline = -1mm]
	\draw[->,bluep] (-.2,-.4) to [out=90,in=-90] (0.2,0) to[out=90,in=-90] (-0.2,.4);
	\draw[->,bluep] (.2,-.4) to [out=90,in=-90] (-0.2,0) to[out=90,in=-90] (0.2,.4);
\end{tikzpicture}
}
\:+\:
\mathord{
\begin{tikzpicture}[baseline = -1mm]
	\draw[->,redp] (-.2,-.4) to [out=90,in=-90] (0.2,0) to[out=90,in=-90] (-0.2,.4);
	\draw[->,bluep] (.2,-.4) to [out=90,in=-90] (-0.2,0) to[out=90,in=-90] (0.2,.4);
\end{tikzpicture}
}
+
\mathord{
\begin{tikzpicture}[baseline = -1mm]
	\draw[->,redp] (-.2,-.4) to [out=90,in=-90] (0.2,0) to[out=90,in=-90] (-0.2,.4);
	\draw[->,bluep] (.2,-.4) to [out=90,in=-90] (-0.2,0) to[out=90,in=-90] (0.2,.4);
\draw[-] (.17,-.31) to (-.17,-.31);
      \node at (-.17,-0.32) {$\dot$};
   \node at (.17,-0.32) {$\dot$};
\end{tikzpicture}
}
-
\mathord{
\begin{tikzpicture}[baseline = -1mm]
	\draw[->,redp] (0.2,-.4) to (-0.2,.4);
	\draw[->,bluep] (-0.2,-.4) to (0.2,.4);
\draw[-] (.09,-.18) to (-.09,-.18);
      \node at (-.09,-0.19) {$\dot$};
   \node at (.09,-0.19) {$\dot$};
\end{tikzpicture}}
-
\mathord{
\begin{tikzpicture}[baseline = -1mm]
	\draw[->,redp] (-.2,-.4) to [out=90,in=-90] (0.2,0) to[out=90,in=-90] (-0.2,.4);
	\draw[->,bluep] (.2,-.4) to [out=90,in=-90] (-0.2,0) to[out=90,in=-90] (0.2,.4);
\draw[-] (.17,.09) to (-.17,.09);
      \node at (-.17,0.08) {$\dot$};
   \node at (.17,0.08) {$\dot$};
\end{tikzpicture}
}
-
\mathord{
\begin{tikzpicture}[baseline = -1mm]
	\draw[->,redp] (-.2,-.4) to [out=90,in=-90] (0.2,0) to[out=90,in=-90] (-0.2,.4);
	\draw[->,bluep] (.2,-.4) to [out=90,in=-90] (-0.2,0) to[out=90,in=-90] (0.2,.4);
\draw[-] (.17,-.31) to (-.17,-.31);
      \node at (-.17,-0.32) {$\dot$};
   \node at (.17,-0.32) {$\dot$};
\draw[-] (.17,.09) to (-.17,.09);
      \node at (-.17,0.08) {$\dot$};
   \node at (.17,0.08) {$\dot$};
\end{tikzpicture}
}
+
\mathord{
\begin{tikzpicture}[baseline = -1mm]
	\draw[->,redp] (0.2,-.4) to (-0.2,.4);
	\draw[->,bluep] (-0.2,-.4) to (0.2,.4);
\draw[-] (.15,-.31) to (-.15,-.31);
      \node at (-.15,-0.32) {$\dot$};
   \node at (.15,-0.32) {$\dot$};
\draw[-] (.09,-.18) to (-.09,-.18);
      \node at (-.09,-0.19) {$\dot$};
   \node at (.09,-0.19) {$\dot$};
\end{tikzpicture}}
+
\mathord{
\begin{tikzpicture}[baseline = -1mm]
	\draw[->,redp] (0.2,-.4) to (-0.2,.4);
	\draw[->,bluep] (-0.2,-.4) to (0.2,.4);
\draw[-] (.09,.18) to (-.15,.18);
      \node at (-.09,0.17) {$\dot$};
   \node at (.09,0.17) {$\dot$};
\end{tikzpicture}}
-
\mathord{
\begin{tikzpicture}[baseline = -1mm]
	\draw[->,redp] (0.2,-.4) to (-0.2,.4);
	\draw[->,bluep] (-0.2,-.4) to (0.2,.4);
\draw[-] (.09,.18) to (-.09,.18);
      \node at (-.09,0.17) {$\dot$};
   \node at (.09,0.17) {$\dot$};
\draw[-] (.09,-.18) to (-.09,-.18);
      \node at (-.09,-0.19) {$\dot$};
   \node at (.09,-0.19) {$\dot$};
\end{tikzpicture}}
+\mathord{
\begin{tikzpicture}[baseline = -1mm]
	\draw[->,bluep] (0.2,-.4) to (0.2,.4);
	\draw[->,redp] (-0.2,-.4) to (-0.2,.4);
\draw[-] (.2,.18) to (-.2,.18);
      \node at (-.2,0.17) {$\dot$};
   \node at (.2,0.17) {$\dot$};
\draw[-] (.2,-.18) to (-.2,-.18);
      \node at (-.2,-0.19) {$\dot$};
   \node at (.2,-0.19) {$\dot$};
\end{tikzpicture}
}\\
+\:&\mathord{
\begin{tikzpicture}[baseline = -1mm]
	\draw[->,redp] (-.2,-.4) to [out=90,in=-90] (0.2,0) to[out=90,in=-90] (-0.2,.4);
	\draw[->,redp] (.2,-.4) to [out=90,in=-90] (-0.2,0) to[out=90,in=-90] (0.2,.4);
\end{tikzpicture}
}
\:+\:
\mathord{
\begin{tikzpicture}[baseline = -1mm]
	\draw[->,bluep] (-.2,-.4) to [out=90,in=-90] (0.2,0) to[out=90,in=-90] (-0.2,.4);
	\draw[->,redp] (.2,-.4) to [out=90,in=-90] (-0.2,0) to[out=90,in=-90] (0.2,.4);
\end{tikzpicture}
}
-
\mathord{
\begin{tikzpicture}[baseline = -1mm]
	\draw[->,bluep] (-.2,-.4) to [out=90,in=-90] (0.2,0) to[out=90,in=-90] (-0.2,.4);
	\draw[->,redp] (.2,-.4) to [out=90,in=-90] (-0.2,0) to[out=90,in=-90] (0.2,.4);
\draw[-] (.17,-.31) to (-.17,-.31);
      \node at (-.17,-0.32) {$\dot$};
   \node at (.17,-0.32) {$\dot$};
\end{tikzpicture}
}
+
\mathord{
\begin{tikzpicture}[baseline = -1mm]
	\draw[->,bluep] (0.2,-.4) to (-0.2,.4);
	\draw[->,redp] (-0.2,-.4) to (0.2,.4);
\draw[-] (.09,-.18) to (-.09,-.18);
      \node at (-.09,-0.19) {$\dot$};
   \node at (.09,-0.19) {$\dot$};
\end{tikzpicture}}
+
\mathord{
\begin{tikzpicture}[baseline = -1mm]
	\draw[->,bluep] (-.2,-.4) to [out=90,in=-90] (0.2,0) to[out=90,in=-90] (-0.2,.4);
	\draw[->,redp] (.2,-.4) to [out=90,in=-90] (-0.2,0) to[out=90,in=-90] (0.2,.4);
\draw[-] (.17,.09) to (-.17,.09);
      \node at (-.17,0.08) {$\dot$};
   \node at (.17,0.08) {$\dot$};
\end{tikzpicture}
}
-
\mathord{
\begin{tikzpicture}[baseline = -1mm]
	\draw[->,bluep] (-.2,-.4) to [out=90,in=-90] (0.2,0) to[out=90,in=-90] (-0.2,.4);
	\draw[->,redp] (.2,-.4) to [out=90,in=-90] (-0.2,0) to[out=90,in=-90] (0.2,.4);
\draw[-] (.17,-.31) to (-.17,-.31);
      \node at (-.17,-0.32) {$\dot$};
   \node at (.17,-0.32) {$\dot$};
\draw[-] (.17,.09) to (-.17,.09);
      \node at (-.17,0.08) {$\dot$};
   \node at (.17,0.08) {$\dot$};
\end{tikzpicture}
}
+
\mathord{
\begin{tikzpicture}[baseline = -1mm]
	\draw[->,bluep] (0.2,-.4) to (-0.2,.4);
	\draw[->,redp] (-0.2,-.4) to (0.2,.4);
\draw[-] (.15,-.31) to (-.15,-.31);
      \node at (-.15,-0.32) {$\dot$};
   \node at (.15,-0.32) {$\dot$};
\draw[-] (.09,-.18) to (-.09,-.18);
      \node at (-.09,-0.19) {$\dot$};
   \node at (.09,-0.19) {$\dot$};
\end{tikzpicture}}
-
\mathord{
\begin{tikzpicture}[baseline = -1mm]
	\draw[->,bluep] (0.2,-.4) to (-0.2,.4);
	\draw[->,redp] (-0.2,-.4) to (0.2,.4);
\draw[-] (.09,.18) to (-.15,.18);
      \node at (-.09,0.17) {$\dot$};
   \node at (.09,0.17) {$\dot$};
\end{tikzpicture}}
-
\mathord{
\begin{tikzpicture}[baseline = -1mm]
	\draw[->,bluep] (0.2,-.4) to (-0.2,.4);
	\draw[->,redp] (-0.2,-.4) to (0.2,.4);
\draw[-] (.09,.18) to (-.09,.18);
      \node at (-.09,0.17) {$\dot$};
   \node at (.09,0.17) {$\dot$};
\draw[-] (.09,-.18) to (-.09,-.18);
      \node at (-.09,-0.19) {$\dot$};
   \node at (.09,-0.19) {$\dot$};
\end{tikzpicture}}
+\mathord{
\begin{tikzpicture}[baseline = -1mm]
	\draw[->,redp] (0.2,-.4) to (0.2,.4);
	\draw[->,bluep] (-0.2,-.4) to (-0.2,.4);
\draw[-] (.2,.18) to (-.2,.18);
      \node at (-.2,0.17) {$\dot$};
   \node at (.2,0.17) {$\dot$};
\draw[-] (.2,-.18) to (-.2,-.18);
      \node at (-.2,-0.19) {$\dot$};
   \node at (.2,-0.19) {$\dot$};
\end{tikzpicture}
}\:.
\end{align*}
This expression is a shorthand for a $4 \times 4$ matrix.
We must show that it equals the $4 \times 4$ identity matrix
$\:\mathord{
\begin{tikzpicture}[baseline = 0]
	\draw[->,bluep] (0.1,-.1) to (0.1,.3);
	\draw[->,bluep] (-0.2,-.1) to (-0.2,.3);
\end{tikzpicture}
}\;+\;
\mathord{
\begin{tikzpicture}[baseline = 0]
	\draw[->,bluep] (0.1,-.1) to (0.1,.3);
	\draw[->,redp] (-0.2,-.1) to (-0.2,.3);
\end{tikzpicture}
}\;+\;
\mathord{
\begin{tikzpicture}[baseline = 0]
	\draw[->,redp] (0.1,-.1) to (0.1,.3);
	\draw[->,bluep] (-0.2,-.1) to (-0.2,.3);
\end{tikzpicture}
}\;+\;
\mathord{
\begin{tikzpicture}[baseline = 0]
	\draw[->,redp] (0.1,-.1) to (0.1,.3);
	\draw[->,redp] (-0.2,-.1) to (-0.2,.3);
\end{tikzpicture}
}\:$.
Looking at
the 16 individual matrix entries (most of which are zero), the proof
reduces to verifying the following three identities
\begin{align*}
\mathord{
\begin{tikzpicture}[baseline = -1mm]
	\draw[->,bluep] (-.2,-.4) to [out=90,in=-90] (0.2,0) to[out=90,in=-90] (-0.2,.4);
	\draw[->,bluep] (.2,-.4) to [out=90,in=-90] (-0.2,0) to[out=90,in=-90] (0.2,.4);
\end{tikzpicture}
}&=
\mathord{
\begin{tikzpicture}[baseline = -1mm]
	\draw[->,bluep] (-.2,-.4) to (-0.2,.4);
	\draw[->,bluep] (.2,-.4) to (0.2,.4);
\end{tikzpicture}
}\:,&
\mathord{
\begin{tikzpicture}[baseline = -1mm]
	\draw[->,redp] (-.2,-.4) to [out=90,in=-90] (0.2,0) to[out=90,in=-90] (-0.2,.4);
	\draw[->,bluep] (.2,-.4) to [out=90,in=-90] (-0.2,0) to[out=90,in=-90] (0.2,.4);
\end{tikzpicture}
}
+
\mathord{
\begin{tikzpicture}[baseline = -1mm]
	\draw[->,redp] (-.2,-.4) to [out=90,in=-90] (0.2,0) to[out=90,in=-90] (-0.2,.4);
	\draw[->,bluep] (.2,-.4) to [out=90,in=-90] (-0.2,0) to[out=90,in=-90] (0.2,.4);
\draw[-] (.17,-.31) to (-.17,-.31);
      \node at (-.17,-0.32) {$\dot$};
   \node at (.17,-0.32) {$\dot$};
\end{tikzpicture}
}
-
\mathord{
\begin{tikzpicture}[baseline = -1mm]
	\draw[->,redp] (-.2,-.4) to [out=90,in=-90] (0.2,0) to[out=90,in=-90] (-0.2,.4);
	\draw[->,bluep] (.2,-.4) to [out=90,in=-90] (-0.2,0) to[out=90,in=-90] (0.2,.4);
\draw[-] (.17,.09) to (-.17,.09);
      \node at (-.17,0.08) {$\dot$};
   \node at (.17,0.08) {$\dot$};
\end{tikzpicture}
}
-
\mathord{
\begin{tikzpicture}[baseline = -1mm]
	\draw[->,redp] (-.2,-.4) to [out=90,in=-90] (0.2,0) to[out=90,in=-90] (-0.2,.4);
	\draw[->,bluep] (.2,-.4) to [out=90,in=-90] (-0.2,0) to[out=90,in=-90] (0.2,.4);
\draw[-] (.17,-.31) to (-.17,-.31);
      \node at (-.17,-0.32) {$\dot$};
   \node at (.17,-0.32) {$\dot$};
\draw[-] (.17,.09) to (-.17,.09);
      \node at (-.17,0.08) {$\dot$};
   \node at (.17,0.08) {$\dot$};
\end{tikzpicture}
}
+\mathord{
\begin{tikzpicture}[baseline = -1mm]
	\draw[->,bluep] (0.2,-.4) to (0.2,.4);
	\draw[->,redp] (-0.2,-.4) to (-0.2,.4);
\draw[-] (.2,.18) to (-.2,.18);
      \node at (-.2,0.17) {$\dot$};
   \node at (.2,0.17) {$\dot$};
\draw[-] (.2,-.18) to (-.2,-.18);
      \node at (-.2,-0.19) {$\dot$};
   \node at (.2,-0.19) {$\dot$};
\end{tikzpicture}
}&=\mathord{
\begin{tikzpicture}[baseline = -1mm]
	\draw[->,redp] (-.2,-.4) to (-0.2,.4);
	\draw[->,bluep] (.2,-.4) to (0.2,.4);
\end{tikzpicture}
}
\:,&
-
\mathord{
\begin{tikzpicture}[baseline = -1mm]
	\draw[->,redp] (0.2,-.4) to (-0.2,.4);
	\draw[->,bluep] (-0.2,-.4) to (0.2,.4);
\draw[-] (.09,-.18) to (-.09,-.18);
      \node at (-.09,-0.19) {$\dot$};
   \node at (.09,-0.19) {$\dot$};
\end{tikzpicture}}
+
\mathord{
\begin{tikzpicture}[baseline = -1mm]
	\draw[->,redp] (0.2,-.4) to (-0.2,.4);
	\draw[->,bluep] (-0.2,-.4) to (0.2,.4);
\draw[-] (.15,-.31) to (-.15,-.31);
      \node at (-.15,-0.32) {$\dot$};
   \node at (.15,-0.32) {$\dot$};
\draw[-] (.09,-.18) to (-.09,-.18);
      \node at (-.09,-0.19) {$\dot$};
   \node at (.09,-0.19) {$\dot$};
\end{tikzpicture}}
+
\mathord{
\begin{tikzpicture}[baseline = -1mm]
	\draw[->,redp] (0.2,-.4) to (-0.2,.4);
	\draw[->,bluep] (-0.2,-.4) to (0.2,.4);
\draw[-] (.09,.18) to (-.15,.18);
      \node at (-.09,0.17) {$\dot$};
   \node at (.09,0.17) {$\dot$};
\end{tikzpicture}}
-
\mathord{
\begin{tikzpicture}[baseline = -1mm]
	\draw[->,redp] (0.2,-.4) to (-0.2,.4);
	\draw[->,bluep] (-0.2,-.4) to (0.2,.4);
\draw[-] (.09,.18) to (-.09,.18);
      \node at (-.09,0.17) {$\dot$};
   \node at (.09,0.17) {$\dot$};
\draw[-] (.09,-.18) to (-.09,-.18);
      \node at (-.09,-0.19) {$\dot$};
   \node at (.09,-0.19) {$\dot$};
\end{tikzpicture}}
=0
\end{align*}
together with the mirror images of these identities under $\eta$.
All are obviously true by commuting relations.
For the dot sliding relation (\ref{heckea}),
one computes the entries of the $4\times 4$ matrices involved to see that
the proof reduces to
checking the following
\begin{align*}
\mathord{
\begin{tikzpicture}[baseline = -1mm]
	\draw[<-,bluep] (0.25,.3) to (-0.25,-.3);
	\draw[->,bluep] (0.25,-.3) to (-0.25,.3);
     \node at (-0.12,-0.145) {$\blue{\dot}$};
\end{tikzpicture}
}
=
\mathord{
\begin{tikzpicture}[baseline = -1mm]
	\draw[<-,bluep] (0.25,.3) to (-0.25,-.3);
	\draw[->,bluep] (0.25,-.3) to (-0.25,.3);
     \node at (0.12,0.135) {$\blue{\dot}$};
\end{tikzpicture}}
+\:\mathord{
\begin{tikzpicture}[baseline = -1mm]
 	\draw[->,bluep] (0.08,-.3) to (0.08,.3);
	\draw[->,bluep] (-0.28,-.3) to (-0.28,.3);
\end{tikzpicture}
}\:,\qquad
\mathord{
\begin{tikzpicture}[baseline = -1mm]
	\draw[<-,bluep] (0.25,.3) to (-0.25,-.3);
	\draw[->,redp] (0.25,-.3) to (-0.25,.3);
     \node at (-0.12,-0.145) {$\blue{\dot}$};
\end{tikzpicture}
}
-
\mathord{
\begin{tikzpicture}[baseline = -1mm]
	\draw[<-,bluep] (0.25,.3) to (-0.25,-.3);
	\draw[->,redp] (0.25,-.3) to (-0.25,.3);
	\draw[-] (-0.12,-.135) to (0.12,-.135);
     \node at (-0.12,-0.145) {$\dot$};
     \node at (0.12,-0.145) {$\dot$};
     \node at (0.12,0.135) {$\blue{\dot}$};
\end{tikzpicture}
}
=
\mathord{
\begin{tikzpicture}[baseline = -1mm]
	\draw[<-,bluep] (0.25,.3) to (-0.25,-.3);
	\draw[->,redp] (0.25,-.3) to (-0.25,.3);
     \node at (0.12,0.135) {$\blue{\dot}$};
\end{tikzpicture}
}
-\mathord{
\begin{tikzpicture}[baseline = -1mm]
	\draw[<-,bluep] (0.25,.3) to (-0.25,-.3);
	\draw[->,redp] (0.25,-.3) to (-0.25,.3);
	\draw[-] (-0.12,.155) to (0.12,.155);
     \node at (-0.12,0.145) {$\dot$};
     \node at (0.12,0.145) {$\dot$};
     \node at (-0.12,-0.145) {$\blue{\dot}$};
\end{tikzpicture}
}
\:,
\qquad
\mathord{
\begin{tikzpicture}[baseline = -1mm]
 	\draw[->,bluep] (0.08,-.3) to (0.08,.3);
	\draw[->,redp] (-0.28,-.3) to (-0.28,.3);
	\draw[-] (-0.28,.06) to (0.08,.06);
     \node at (-0.28,0.05) {$\dot$};
     \node at (0.08,0.05) {$\dot$};
     \node at (-0.28,-0.16) {$\red{\dot}$};
\end{tikzpicture}
}
=\:
\mathord{
\begin{tikzpicture}[baseline = -1mm]
 	\draw[->,bluep] (0.08,-.3) to (0.08,.3);
	\draw[->,redp] (-0.28,-.3) to (-0.28,.3);
	\draw[-] (-0.28,-.15) to (0.08,-.15);
     \node at (-0.28,-0.16) {$\dot$};
     \node at (0.08,-0.16) {$\dot$};
     \node at (0.08,0.05) {$\blue{\dot}$};
\end{tikzpicture}
}
+\:
\mathord{
\begin{tikzpicture}[baseline = -1mm]
 	\draw[->,bluep] (0.08,-.3) to (0.08,.3);
	\draw[->,redp] (-0.28,-.3) to (-0.28,.3);
\end{tikzpicture}
}
\:,
\end{align*}
together with their mirror images under $\eta$.
Again these are all clear; use (\ref{teleporting}) for the last one.
The braid relation may be checked by a similar sort of
calculation although this is quite lengthy since it involves $8
\times 8$ matrices; this is where one needs (\ref{c1})--(\ref{c2}).
\end{proof}

The goal now is to show that the new version of the functor
$\Delta$ also categorifies $\delta$ by 
establishing an analog of (\ref{shower}).
The canonical functors 
$\blue{\AH} \rightarrow 
\blue{\AH}
 \;\overline{\odot}\; \red{\AH}$
and
$\red{\AH} \rightarrow 
\blue{\AH}
  \;\overline{\odot}\; \red{\AH}$ induce a ring homomorphism
$\epsilon:K_0(\Kar(\blue{\AH})) \otimes_{\Z} K_0(\Kar(\red{\AH}))
\rightarrow 
K_0(\Kar(\blue{\AH}
  \;\overline{\odot}\; \red{\AH}))$.
We claim that
\begin{equation}\label{shower2}
\begin{diagram}
\node{\SymZ}\arrow{s,l,A,J}{\gamma}\arrow{e,t}{\delta}\node{\SymZ\otimes_\Z\SymZ\:}\arrow{s,r}{\epsilon\circ\gamma\otimes\gamma}\\
\node{K_0(\Kar(\AH))}\arrow{e,b}{[\Delta]}\node{K_0(\Kar(\blue{\AH}
  \;\overline{\odot}\; \red{\AH}))}
\end{diagram}
\end{equation}
commutes. This follows from the next theorem.

\begin{theorem}\label{upgraded}
For each $n \geq 0$, we have that 
\begin{align}
\Delta(H_n) &\cong \bigoplus_{r=0}^n \blue{H_{n-r}}
\otimes \red{H_r},&
\Delta(E_n) &\cong \bigoplus_{r=0}^n \blue{E_{n-r}}
\otimes \red{E_r}.
\end{align}
\end{theorem}

In comparison to
Theorem~\ref{trivial}, the proof of Theorem~\ref{upgraded} is rather
non-trivial,
 and it will occupy the remainder of the
section.
We will need the isomorphisms $\sigma_\lambda\:(\lambda
\in \mathcal P_{r,n})$ from 
(\ref{babysig}), viewed now as morphisms in $\blue{\AH}\;\overline{\odot}\;\red{\AH}$.
Let us also identify $AH_r \otimes_\k AH_{n-r}$ with a subalgebra of $AH_n$
so that
 $s_i \otimes 1 \leftrightarrow s_i,
x_i \otimes 1 \leftrightarrow x_i, 1 \otimes s_j \leftrightarrow
s_{r+j}$ and $1 \otimes x_j \leftrightarrow x_{r+j}$.
Let 
$AH_r \;\overline{\otimes}_\k\; AH_{n-r}$ be the Ore localization of $AH_r
\otimes_\k AH_{n-r}$ at the central element
\begin{equation}
z_{r,n} := \prod_{i=1}^r \prod_{j=1}^{n-r} (x_i - x_{r+j}).
\end{equation}
Generalizing (\ref{doublei}), there is an algebra isomorphism
\begin{equation}
\imath_{r,n}:AH_r \;\overline{\otimes}_\k\; AH_{n-r}
\rightarrow \End_{\blue{\AH}\;\overline{\odot}\;\red{\AH}}(\blueup^{\otimes (n-r)} \otimes \redup^{\otimes r})
\end{equation}
sending $s_i = s_i \otimes 1$ and $s_{r+j} = 1\otimes s_j$ to the same diagrams as before, and $x_i = x_i \otimes
1$ and $x_{r+j}=1\otimes x_j$ to dots on the $i$th red string or $j$th
blue string, respectively.
To see this, we just observe that the analogous isomorphism before
localizing is obvious; then it follows for the localized versions too
since 
all dumbbells make sense
in $AH_r \;\overline{\otimes}_\k\; AH_{n-r}$, and conversely the image of $z_{r,n}$ is invertible in 
the endomorphism algebra.
Just like in (\ref{matrix}), we then get that
\begin{equation}\label{anothermatrix}
\End_{\Add(\blue{\AH}\;\overline{\odot}\;\red{\AH})}\left((\blueup\oplus\redup)^{\otimes n}\right) \cong 
\bigoplus_{r=0}^n \operatorname{Mat}_{\binom{n}{r}}\left(AH_r
 \; \overline{\otimes}_\k\; AH_{n-r}\right).
\end{equation}
For $\lambda \in \mathcal P_{r,n}$ and $1 \leq i \leq r, 1 \leq j \leq
r-n$, we let 
\begin{equation}\label{grayarea}
\varepsilon_{i,j}(\lambda) := \left\{\begin{array}{rl}
1&\text{if $j \leq \lambda_i$}\\
-1&\text{if $j > \lambda_i$.}
\end{array}\right.
\end{equation}
Thus it is 1 or $-1$ according to whether $(i,j)$ is inside or outside
of the Young diagram of $\lambda$.
Also let
\begin{equation}\label{regret}
y_{i,j} := 
(x_{r+1-i} - x_{r+j})^{-1} \in AH_r \;\overline{\otimes}_{\k}\; AH_{n-r}.
\end{equation}
Numbering strings $1,\dots,n$ from right to left as usual, $\imath_{r,n}(y_{i,j})$ is the
dumbbell between the $(r+1-i)$th and
$(r+j)$th strings; alternatively, numbering strings from the center (with
red to the right and blue to the left)
it joins the $i$th red
string to the $j$th blue string.
The key observation needed to prove Theorem~\ref{upgraded} is as follows.

\begin{lemma}\label{government}
For $0 \leq r \leq n$ and $\lambda,\mu \in \mathcal P_{r,n}$, we have
that
\begin{align*}
 1_\mu \circ \Delta(\imath_n(e_{(n)})) \circ 1_\lambda &= \binom{\,n\,}{\,r\,}^{-1}
\sigma_\mu^{-1} \circ \imath_{r,n}\left(e_{(r)} \otimes
  e_{(n-r)}\right)\circ\imath_{r,n}\bigg(\!\prod_{\substack{1 \leq i \leq r\\1 \leq j
    \leq n-r}}\!
\left(1 + \varepsilon_{i,j}(\lambda)y_{i,j}\right)\bigg)
\circ \sigma_\lambda,\\
 1_\mu \circ \Delta(\imath_n(e_{(1^n)})) \circ 1_\lambda &= 
(-1)^{|\lambda|+|\mu|}\:\binom{\,n\,}{\,r\,}^{-1}\!\!\!\!\sigma_\mu^{-1}\circ \imath_{r,n}\bigg(\!\!\prod_{\substack{1 \leq i \leq r\\1 \leq j
    \leq n-r}}\!\!
\left(1 - \varepsilon_{i,j}(\mu)y_{i,j}\right)
\bigg)\circ\imath_{r,n}\left(e_{(1^r)} \otimes
  e_{(1^{n-r})}\right)
\circ \sigma_\lambda.
\end{align*}
\end{lemma}

\begin{proof}
Note that $\Delta(\imath_2(e_{(2)}))$
and
$\Delta(\imath_2(e_{(1^2)}))$
are equal to
\begin{align*}
&\textstyle\frac{1}{2}\left(\:
\begin{tikzpicture}[baseline = -0.6mm]
 	\draw[->,bluep] (0.16,-.3) to (0.16,.3);
	\draw[->,bluep] (-0.16,-.3) to (-0.16,.3);
\end{tikzpicture}
+\begin{tikzpicture}[baseline = -.6mm]
	\draw[->,bluep] (0.2,-.3) to (-0.2,.3);
	\draw[->,bluep] (-0.2,-.3) to (0.2,.3);
\end{tikzpicture}\:
\right)
+
\frac{1}{2}\left(\:
\begin{tikzpicture}[baseline = -0.6mm]
 	\draw[->,redp] (0.16,-.3) to (0.16,.3);
	\draw[->,redp] (-0.16,-.3) to (-0.16,.3);
\end{tikzpicture}
+\begin{tikzpicture}[baseline = -.6mm]
	\draw[->,redp] (0.2,-.3) to (-0.2,.3);
	\draw[->,redp] (-0.2,-.3) to (0.2,.3);
\end{tikzpicture}\:
\right)
+
\frac{1}{2}\left(\:
\begin{tikzpicture}[baseline = -0.6mm]
 	\draw[->,redp] (0.16,-.3) to (0.16,.3);
	\draw[->,bluep] (-0.16,-.3) to (-0.16,.3);
\end{tikzpicture}
+\begin{tikzpicture}[baseline = -.6mm]
	\draw[->,redp] (0.2,-.3) to (-0.2,.3);
	\draw[->,bluep] (-0.2,-.3) to (0.2,.3);
\end{tikzpicture}\:
\right)
\circ \left(\: \begin{tikzpicture}[baseline = -0.6mm]
 	\draw[->,redp] (0.16,-.3) to (0.16,.3);
	\draw[->,bluep] (-0.16,-.3) to (-0.16,.3);
\end{tikzpicture}
-\!
\begin{tikzpicture}[baseline = -0.6mm]
 	\draw[->,redp] (0.16,-.3) to (0.16,.3);
	\draw[->,bluep] (-0.16,-.3) to (-0.16,.3);
	\draw[-] (-0.16,0.01) to (0.16,0.01);
     \node at (0.16,0.0) {$\dot$};
     \node at (-0.16,0.0) {$\dot$};
\end{tikzpicture}
\right)
+
\frac{1}{2}\left(\:
\begin{tikzpicture}[baseline = -0.6mm]
 	\draw[->,bluep] (0.16,-.3) to (0.16,.3);
	\draw[->,redp] (-0.16,-.3) to (-0.16,.3);
\end{tikzpicture}
+\begin{tikzpicture}[baseline = -.6mm]
	\draw[->,bluep] (0.2,-.3) to (-0.2,.3);
	\draw[->,redp] (-0.2,-.3) to (0.2,.3);
\end{tikzpicture}\:
\right)
\circ
\left(
\: \begin{tikzpicture}[baseline = -0.6mm]
 	\draw[->,bluep] (0.16,-.3) to (0.16,.3);
	\draw[->,redp] (-0.16,-.3) to (-0.16,.3);
\end{tikzpicture}
+\!
\begin{tikzpicture}[baseline = -0.6mm]
 	\draw[->,bluep] (0.16,-.3) to (0.16,.3);
	\draw[->,redp] (-0.16,-.3) to (-0.16,.3);
	\draw[-] (-0.16,0.01) to (0.16,0.01);
     \node at (0.16,0.0) {$\dot$};
     \node at (-0.16,0.0) {$\dot$};
\end{tikzpicture}
\right),\\
&\textstyle\frac{1}{2}\left(\:
\begin{tikzpicture}[baseline = -0.6mm]
 	\draw[->,bluep] (0.16,-.3) to (0.16,.3);
	\draw[->,bluep] (-0.16,-.3) to (-0.16,.3);
\end{tikzpicture}-\begin{tikzpicture}[baseline = -.6mm]
	\draw[->,bluep] (0.2,-.3) to (-0.2,.3);
	\draw[->,bluep] (-0.2,-.3) to (0.2,.3);
\end{tikzpicture}\:
\right)
+
\frac{1}{2}\left(\:
\begin{tikzpicture}[baseline = -0.6mm]
 	\draw[->,redp] (0.16,-.3) to (0.16,.3);
	\draw[->,redp] (-0.16,-.3) to (-0.16,.3);
\end{tikzpicture}
-\begin{tikzpicture}[baseline = -.6mm]
	\draw[->,redp] (0.2,-.3) to (-0.2,.3);
	\draw[->,redp] (-0.2,-.3) to (0.2,.3);
\end{tikzpicture}\:
\right)
+
\frac{1}{2}
\left(\: \begin{tikzpicture}[baseline = -0.6mm]
 	\draw[->,redp] (0.16,-.3) to (0.16,.3);
	\draw[->,bluep] (-0.16,-.3) to (-0.16,.3);
\end{tikzpicture}
+\!
\begin{tikzpicture}[baseline = -0.6mm]
 	\draw[->,redp] (0.16,-.3) to (0.16,.3);
	\draw[->,bluep] (-0.16,-.3) to (-0.16,.3);
	\draw[-] (-0.16,0.01) to (0.16,0.01);
     \node at (0.16,0.0) {$\dot$};
     \node at (-0.16,0.0) {$\dot$};
\end{tikzpicture}
\right)
\circ
\left(\:
\begin{tikzpicture}[baseline = -0.6mm]
 	\draw[->,redp] (0.16,-.3) to (0.16,.3);
	\draw[->,bluep] (-0.16,-.3) to (-0.16,.3);
\end{tikzpicture}
-\begin{tikzpicture}[baseline = -.6mm]
	\draw[->,bluep] (0.2,-.3) to (-0.2,.3);
	\draw[->,redp] (-0.2,-.3) to (0.2,.3);
\end{tikzpicture}
\:
\right)
+
\frac{1}{2}
\left(
\: \begin{tikzpicture}[baseline = -0.6mm]
 	\draw[->,bluep] (0.16,-.3) to (0.16,.3);
	\draw[->,redp] (-0.16,-.3) to (-0.16,.3);
\end{tikzpicture}
-\!
\begin{tikzpicture}[baseline = -0.6mm]
 	\draw[->,bluep] (0.16,-.3) to (0.16,.3);
	\draw[->,redp] (-0.16,-.3) to (-0.16,.3);
	\draw[-] (-0.16,0.01) to (0.16,0.01);
     \node at (0.16,0.0) {$\dot$};
     \node at (-0.16,0.0) {$\dot$};
\end{tikzpicture}
\right)
\circ
\left(\:
\begin{tikzpicture}[baseline = -0.6mm]
 	\draw[->,bluep] (0.16,-.3) to (0.16,.3);
	\draw[->,redp] (-0.16,-.3) to (-0.16,.3);
\end{tikzpicture}
-\begin{tikzpicture}[baseline = -.6mm]
	\draw[->,redp] (0.2,-.3) to (-0.2,.3);
	\draw[->,bluep] (-0.2,-.3) to (0.2,.3);
\end{tikzpicture}\:
\right),
\end{align*}
respectively.
The lemma in the case $n=2$ follows from these formulae.
For the general case, we proceed by induction on $|\mu|-|\lambda|$. We just explain the proof for the first formula, since the second is similar.

In the base case when $\mu = \min_{r,n}$ (so $
1_\mu = \sigma_\mu=\blueup^{\otimes (n-r)}\otimes\redup^{\otimes r}$) 
and $\lambda = \max_{r,n}$ (so
$1_\lambda = \redup^{\otimes r}\otimes\blueup^{\otimes (n-r)}$),
we have that
$$
e_{(n)} =
\frac{1}{n!} 
\sum_{\tau \in \S_r \times \S_{n-r}} \sum_{\sigma \in D}
\tau \sigma
$$ where 
$D$ denotes the set of minimal length
$\S_r \times \S_{n-r} \backslash \S_n$-coset representatives.
For $\tau \in \S_r \times \S_{n-r}$, we have that $1_\mu \circ \Delta(\imath_n(\tau)) = \sigma_\mu^{-1}
\circ \imath_{r,n}(\tau) \circ 1_\mu$.
Thus, we see that
$$
1_\mu \circ \Delta(\imath_n(e_{(n)})) \circ 1_\lambda =\binom{\,n\,}{\,r\,}^{-1}
\sum_{\sigma \in D}
\sigma_\mu^{-1} \circ \imath_{r,n}(e_{(r)}\otimes e_{(n-r)}) \circ 1_\mu \circ
\Delta(\imath_n(\sigma)) \circ 1_\lambda.
$$
Since $\lambda$ is maximal,
the term $1_\mu \circ \Delta(\imath_n(\sigma))\circ 1_\lambda$ here can only be
non-zero when $\sigma$ is the longest coset representative.
Moreover, when computing
$\Delta(\imath_n(\sigma))$,
we must replace each crossing 
$\begin{tikzpicture}[baseline = -.6mm]
	\draw[->] (0.2,-.25) to (-0.2,.25);
	\draw[->] (-0.2,-.25) to (0.2,.25);
\end{tikzpicture}$
in a reduced word for $\imath_n(\sigma)$
with
$\mathord{
\begin{tikzpicture}[baseline = -.6mm]
	\draw[->,bluep] (0.2,-.25) to (-0.2,.25);
	\draw[->,redp] (-0.2,-.25) to (0.2,.25);
\end{tikzpicture}
}+
\mathord{
\begin{tikzpicture}[baseline = -.6mm]
	\draw[->,bluep] (0.2,-.25) to (-0.2,.25);
	\draw[->,redp] (-0.2,-.25) to (0.2,.25);
	\draw[-] (-0.12,-.16) to (0.12,-.16);
     \node at (0.12,-.17) {$\dot$};
     \node at (-0.12,-.17) {$\dot$};
\end{tikzpicture}
}$, i.e., the terms from the expression in (\ref{com0old}) 
that are colored
$\blueup\redup$ at the top and
$\redup\blueup$ at the bottom.
We conclude for this longest $\sigma$ that
$$
1_\mu \circ \Delta(\imath_n(\sigma)) \circ 1_\lambda
=
\imath_{r,n}\bigg(\prod_{\substack{1 \leq i \leq r\\1 \leq j
    \leq n-r}}
\left(1 + y_{i,j}\right)\bigg)
\circ \sigma_\lambda.
$$
Since $\varepsilon_{i,j}(\lambda) = 1$ for all $i$ and $j$, this
checks the base case.

For the induction step, take $\mu,\lambda \in \mathcal P_{r,n}$ 
such that either $\mu$ is not minimal or $\lambda$ is not maximal,
and consider 
$X := 1_\mu \circ \Delta(\imath_n(e_{(n)})) \circ
1_\lambda$.
If $\mu$ is not minimal, we let $\nu \in \mathcal P_{r,n}$ be obtained from $\mu$ by removing a
box. Let $j$ be the unique index such that that $\sigma_\mu^{-1} = 
\left(\:\begin{tikzpicture}[baseline = -.6mm]
	\draw[->,redp] (0.2,-.25) to (-0.2,.25);
	\draw[->,bluep] (-0.2,-.25) to (0.2,.25);
\end{tikzpicture}\:\right)_j
\circ \sigma_\nu^{-1}$,
where the subscript indicates we are applying the crossing to the
$j$th and $(j+1)$th strings.
The induction hypothesis gives us a formula for
$Y := 1_\nu \circ
\Delta(\imath_n(e_{(n)})) \circ 1_\lambda$, reducing the problem to 
showing that
$X =
\left(\:\begin{tikzpicture}[baseline = -.6mm]
	\draw[->,redp] (0.2,-.25) to (-0.2,.25);
	\draw[->,bluep] (-0.2,-.25) to (0.2,.25);
\end{tikzpicture}\:\right)_j
\circ Y$.
To see this, 
we apply $1_\mu \circ \Delta(\imath_n(-))\circ
1_\lambda$ to the identity 
$e_{(n)}=\frac{1}{2}(1+s_{j}) e_{(n)}$ to deduce that
$$
\textstyle 
X = \frac{1}{2}
\left(
\: \begin{tikzpicture}[baseline = -0.6mm]
 	\draw[->,bluep] (0.16,-.3) to (0.16,.3);
	\draw[->,redp] (-0.16,-.3) to (-0.16,.3);
\end{tikzpicture}
+\!
\begin{tikzpicture}[baseline = -0.6mm]
 	\draw[->,bluep] (0.16,-.3) to (0.16,.3);
	\draw[->,redp] (-0.16,-.3) to (-0.16,.3);
	\draw[-] (-0.16,0.01) to (0.16,0.01);
     \node at (0.16,0.0) {$\dot$};
     \node at (-0.16,0.0) {$\dot$};
\end{tikzpicture}
\right)_j
\circ X+ \frac{1}{2}
\left(\:\begin{tikzpicture}[baseline = -.6mm]
	\draw[->,redp] (0.2,-.3) to (-0.2,.3);
	\draw[->,bluep] (-0.2,-.3) to (0.2,.3);
\end{tikzpicture}-
\begin{tikzpicture}[baseline = -.6mm]
	\draw[->,redp] (0.2,-.3) to (-0.2,.3);
	\draw[->,bluep] (-0.2,-.3) to (0.2,.3);
	\draw[-] (-0.11,-.14) to (0.11,-.14);
     \node at (0.11,-.15) {$\dot$};
     \node at (-0.11,-.15) {$\dot$};
\end{tikzpicture}
\right)_j
\circ Y.
$$
Hence,
$$
\left(
\: \begin{tikzpicture}[baseline = -0.6mm]
 	\draw[->,bluep] (0.16,-.3) to (0.16,.3);
	\draw[->,redp] (-0.16,-.3) to (-0.16,.3);
\end{tikzpicture}
-\!
\begin{tikzpicture}[baseline = -0.6mm]
 	\draw[->,bluep] (0.16,-.3) to (0.16,.3);
	\draw[->,redp] (-0.16,-.3) to (-0.16,.3);
	\draw[-] (-0.16,0.01) to (0.16,0.01);
     \node at (0.16,0.0) {$\dot$};
     \node at (-0.16,0.0) {$\dot$};
\end{tikzpicture}
\right)_j
\circ X
=
\left(
\: \begin{tikzpicture}[baseline = -0.6mm]
 	\draw[->,bluep] (0.16,-.3) to (0.16,.3);
	\draw[->,redp] (-0.16,-.3) to (-0.16,.3);
\end{tikzpicture}
-\!
\begin{tikzpicture}[baseline = -0.6mm]
 	\draw[->,bluep] (0.16,-.3) to (0.16,.3);
	\draw[->,redp] (-0.16,-.3) to (-0.16,.3);
	\draw[-] (-0.16,0.01) to (0.16,0.01);
     \node at (0.16,0.0) {$\dot$};
     \node at (-0.16,0.0) {$\dot$};
\end{tikzpicture}
\right)_j
\circ
\left(\:
\begin{tikzpicture}[baseline = -.6mm]
	\draw[->,redp] (0.2,-.3) to (-0.2,.3);
	\draw[->,bluep] (-0.2,-.3) to (0.2,.3);
\end{tikzpicture}\:\right)_j
\circ Y.
$$
In view of the isomorphism (\ref{anothermatrix}), this morphism space is free as a module over the integral domain $\k[x_1,\dots,x_n]_{z_{r,n}}$, so
it is permissible to
cancel the first term, and this
gives the desired formula.
Instead, if $\lambda \in \mathcal P_{r,n}$ is not maximal,
we let $\kappa$ be obtained from $\lambda$ by adding a box,
and define $j$ so that $\sigma_\lambda = \sigma_\kappa \circ \left(\:\begin{tikzpicture}[baseline = -.6mm]
	\draw[->,redp] (0.2,-.25) to (-0.2,.25);
	\draw[->,bluep] (-0.2,-.25) to (0.2,.25);
\end{tikzpicture}\:\right)_j$.
Let $Z := 1_\mu \circ
\Delta(\imath_n(e_{(n)})) \circ 1_\kappa$.
Then we need to show that
$$
X\circ
\left(\:
\: \begin{tikzpicture}[baseline = -0.6mm]
 	\draw[->,redp] (0.16,-.3) to (0.16,.3);
	\draw[->,bluep] (-0.16,-.3) to (-0.16,.3);
\end{tikzpicture}
+\!
\begin{tikzpicture}[baseline = -0.6mm]
 	\draw[->,redp] (0.16,-.3) to (0.16,.3);
	\draw[->,bluep] (-0.16,-.3) to (-0.16,.3);
	\draw[-] (-0.16,0.01) to (0.16,0.01);
     \node at (0.16,0.0) {$\dot$};
     \node at (-0.16,0.0) {$\dot$};
\end{tikzpicture}
\right)_j
 = 
Z
\circ 
\left(\:\begin{tikzpicture}[baseline = -.6mm]
	\draw[->,redp] (0.2,-.25) to (-0.2,.25);
	\draw[->,bluep] (-0.2,-.25) to (0.2,.25);
\end{tikzpicture}\:\right)_j
\circ \left(\:
\: \begin{tikzpicture}[baseline = -0.6mm]
 	\draw[->,redp] (0.16,-.3) to (0.16,.3);
	\draw[->,bluep] (-0.16,-.3) to (-0.16,.3);
\end{tikzpicture}
-\!
\begin{tikzpicture}[baseline = -0.6mm]
 	\draw[->,redp] (0.16,-.3) to (0.16,.3);
	\draw[->,bluep] (-0.16,-.3) to (-0.16,.3);
	\draw[-] (-0.16,0.01) to (0.16,0.01);
     \node at (0.16,0.0) {$\dot$};
     \node at (-0.16,0.0) {$\dot$};
\end{tikzpicture}
\right)_j
,
$$
which follows by applying
$1_\mu \circ \Delta(\imath_n(-))\circ
1_\lambda$ to the identity 
$e_{(n)}=e_{(n)}\frac{1}{2}(1+s_{j})$.
\end{proof}

From the defining relations, one sees that $s_i f e_{(n)} = (s_i \bull f)
e_{(n)}$ and $s_i f e_{(1^n)} = -(s_i \star f) e_{(1^n)}$,
where 
\begin{equation}\label{dotact}
s_i \bull f := s_i(f) + \partial_i(f),
\qquad
s_i \star f := s_i(f) - \partial_i(f).
\end{equation} 
Transporting the left action of $AH_n$  on $AH_n e_{(n)}$
through the linear isomorphism
$\k[x_1,\dots,x_n] \stackrel{\sim}{\rightarrow} AH_n e_{(n)},
f \mapsto f e_{(n)},$
we deduce that
$\k[x_1,\dots,x_n]$
is a left $AH_n$-module  with
$\k[x_1,\dots,x_n]$ acting by left multiplication and
$\S_n$ acting by 
$\bull$.
By degree considerations, the space of
$\S_n$-fixed points with respect to the action $\bull$ is the same as the fixed
points with respect to the usual action, i.e., we recover the
subalgebra $\Sym_n$ of $\k[x_1,\dots,x_n]$.
This shows that the {\em spherical subalgebra} $e_{(n)} \dH_n e_{(n)}$
of $\dH_n$ is 
$\Sym_n$.
Moreover, for any $f \in \k[x_1,\dots,x_n]$, we have that
\begin{equation}\label{numbers}
e_{(n)} f e_{(n)} = \frac{1}{n!} \sum_{\pi \in \S_n} e_{(n)} \pi f e_{(n)} =
e_{(n)} \left( \frac{1}{n!}\sum_{\pi \in \S_n} \pi \bull f \right) e_{(n)}.
\end{equation}
Similarly, one sees that $e_{(1^n)} \dH_n e_{(1^n)} = \Sym_n$ and
\begin{equation}\label{numbers2}
e_{(1^n)} f e_{(1^n)} = 
e_{(1^n)} \left( \frac{1}{n!}\sum_{\pi \in \S_n} \pi \star f \right) e_{(1^n)}.
\end{equation}
The $\bull$ and $\star$ actions extend to actions on $\k(x_1,\dots,x_n)$,
with the simple transpositions 
satisfying the same formulae (\ref{dotact}).

\begin{lemma}\label{wow}
For $0 \leq r \leq n$, we have that
\begin{align*}
\sum_{\pi \in \S_r \times \S_{n-r}} \!\!\!\pi \bull \bigg(\sum_{\lambda
  \in \mathcal P_{r,n}} \prod_{\substack{1 \leq i \leq r \\ 1 \leq j
    \leq n-r}}
\left(1 + \varepsilon_{i,j}(\lambda)y_{i,j}\right)\bigg)&
=n!= \sum_{\pi \in \S_r \times \S_{n-r}} \!\!\!\pi \star \bigg(\sum_{\mu
  \in \mathcal P_{r,n}} \prod_{\substack{1 \leq i \leq r \\ 1 \leq j
    \leq n-r}}
\left(1 - \varepsilon_{i,j}(\mu)y_{i,j}\right)\bigg).
\end{align*}
\end{lemma}

\begin{proof}
We just explain the proof of the first equality;  the second then
follows by considering the automorphism $x_i \mapsto -x_i$ of
$\k(x_1,\dots,x_n)$.
Proceed by induction on $n$. For the induction step, we partition
$\mathcal P_{r,n}$ as $A \sqcup B$ as suggested by the diagram:
$$
A\leftrightarrow
\begin{tikzpicture}[baseline=-5mm]
\draw[-] (0,0) to (1,0) to (1,-.8) to (0,-.8) to (0,0);
\draw[-,dashed] (0,-.1) to (1,-.1);
\draw[-] (.2,-.8) to (.2,-.6) to (.6,-.6) to (.6,-.3) to (1,-.3);
      \node at (0.3,-0.35) {$+$};
      \node at (0.82,-0.56) {$-$};
\end{tikzpicture}\:,\qquad
B\leftrightarrow
\begin{tikzpicture}[baseline=-5mm]
\draw[-] (0,0) to (1,0) to (1,-.8) to (0,-.8) to (0,0);
\draw[-,dashed] (.9,0) to (.9,-.8);
\draw[-] (.2,-.8) to (.2,-.6) to (.5,-.6) to (.5,-.3) to (.7,-.3) to (.7,0);
      \node at (0.25,-0.3) {$+$};
      \node at (0.7,-0.56) {$-$};
\end{tikzpicture}\:.
$$
 Thus, $A$ consists of $\lambda \in \mathcal P_{r,n}$ such
that $\lambda_1 = n-r$, and $B$ consists of $\lambda \in
\mathcal P_{r,n}$ such that $\lambda_1 < n-r$.
The expression we are trying to compute then splits as a sum $X  + Y$
where for $X$ we take the second sum just over $\lambda \in A$ and for
$Y$ we take it over $\lambda \in B$.
Using the induction hypothesis plus the observation that 
$\{1,s_{m-1},s_{m-2}s_{m-1},\dots,s_1,\dots,s_{m-1}\}$ is a set of
$\S_m / \S_{m-1}$-coset representatives, we see that
\begin{align*}
X&=
  (n-1)! (1+s_{r-1}+\cdots+s_1\cdots s_{r-1}) \bull
\prod_{j=1}^{n-r} \left(1+y_{1,j}\right),\\
Y&= 
(n-1)!
(1+s_{n-1}+\cdots+s_{r+1}\cdots s_{n-1})\bull
\prod_{i=1}^r \left(1-y_{i,n-r}\right).
\end{align*}
It remains to show that $X+Y=n!$.

From \eqref{regret} and \eqref{c1}--\eqref{c2}, we obtain the following identities
for $1 \leq i \leq r, 1 \leq j \leq n-r$:
\begin{align} \label{shred1}
s_{r+1-q} y_{i,j}&=\left\{\begin{array}{ll}
y_{i+1,j} s_{r+1-q}-y_{i+1,j}y_{i,j}&\text{if $i+1=q\leq r$,}\\
y_{i,j} s_{r+1-q} &\text{if $i+1 < q \leq r$;}
\end{array}\right.\\
s_{r+q}y_{i,j}
&=\left\{\begin{array}{ll}
y_{i,j-1} s_{r+q}+y_{i,j-1}y_{i,j}&\text{if $1 \leq q=j-1$,}\\
y_{i,j} s_{r+q} &\text{if $1 \leq q < j-1$.}
\end{array}\right.\label{shred2}
\end{align}
For $m \geq 1$, let $C_m$ be the set of sequences
$\left((i_1,j_1),\dots,(i_m,j_m)\right) \in
\left(\{1,\dots,r\}\times\{1,\dots,n-r\}\right)^{m}$
such that either $i_q > i_{q+1}, j_q = j_{q+1}$ or $i_q=i_{q+1},j_q <
j_{q+1}$ for each $q=1,\dots,m-1$.
Such a sequence may be visualized as a ``hook'' drawn inside the $r
\times (n-r)$ rectangle, e.g., if $r=4,n=9$ then 
$((4,1),(4,2), (2,2), (2,4)) \in C_4$ is
$\begin{tikzpicture}[anchorbase]
\draw[-] (0,0.2) to (1,0.2) to (1,-.6) to (0,-.6) to (0,0.2);
\draw[-](.1,-.5) to (.3,-.5) to (.3,-.1) to (.7,-.1);
      \node at (0.1,-0.5) {$\bullet$};
      \node at (0.3,-0.5) {$\bullet$};
      \node at (0.3,-0.1) {$\bullet$};
      \node at (0.7,-0.1) {$\bullet$};
\end{tikzpicture}
$\:.
Using \eqref{shred1} and induction on $i=1,\dots,r$, one shows that
$$
s_{r+1-i} \cdots s_{r-2} s_{r-1} \bull \prod_{j=1}^{n-r}
\left(1+y_{1,j}\right)
=
1-\sum_{m \geq 1}
\sum_{\substack{\left((i_1,j_1),\dots,(i_m,j_m)\right)\in C_m\\i_1 =
    i}}
(-1)^{|\{i_1,\dots,i_m\}|} y_{i_1,j_1}\cdots y_{i_m,j_m}.
$$
Hence:
\begin{equation}\label{eggs1}
X= r(n-1)!
-(n-1)!\sum_{m \geq 1}
\sum_{\left((i_1,j_1),\dots,(i_m,j_m)\right)\in C_m}
(-1)^{|\{i_1,\dots,i_m\}|} y_{i_1,j_1}\cdots y_{i_m,j_m}.
\end{equation}
Similarly, using \eqref{shred2} and induction on $j=n-r,\dots,1$,
one shows that
$$
s_{r+j} \dots s_{n-2} s_{n-1}\bull 
\prod_{i=1}^r \left(1-y_{i,n-r}\right)
=
1+\sum_{m \geq 1}
\sum_{\substack{\left((i_1,j_1),\dots,(i_m,j_m)\right)\in C_m\\j_1 =
    j}}
(-1)^{|\{i_1,\dots,i_m\}|} y_{i_1,j_1}\cdots y_{i_m,j_m}.
$$
Hence:
\begin{equation}\label{eggs2}
Y = (n-r)(n-1)!+(n-1)!
\sum_{m \geq 1}
\sum_{\left((i_1,j_1),\dots,(i_m,j_m)\right)\in C_m}
(-1)^{|\{i_1,\dots,i_m\}|} y_{i_1,j_1}\cdots y_{i_m,j_m}.
\end{equation}
Adding the identities (\ref{eggs1}) and (\ref{eggs2}) gives that $X+Y = n!$.
\end{proof}

For later reference, let us also discuss the space $e_{(1^n)} AH_n e_{(n)}$.
For $\lambda =(\lambda_1,\dots,\lambda_n)\in \N^n$, 
let $x^\lambda := x_1^{\lambda_1}\cdots x_n^{\lambda_n}$ and
$A_\lambda := \sum_{\pi \in \S_n} (-1)^{\ell(\pi)} \pi(x^\lambda)$.
Setting $\rho := (n-1,\dots,1,0) \in \N^n$, the symmetric polynomial
\begin{equation}\label{chidef}
\jonschi_\lambda := A_{\lambda+\rho}\big / A_\rho \in \Sym_n
\end{equation}
is the usual Schur polynomial in $n$ variables
when $\lambda_1 \geq \cdots \geq \lambda_n$; on the other hand, it is zero if
$\lambda+\rho$ has a repeated entry.
We have that $e_{(1^n)} (\ker \partial_i) e_{(n)} = 0$, hence,
$e_{(1^n)} s_i(f) e_{(n)} = - e_{(1^n)} f e_{(n)}$.
Since $\k[x_1,\dots,x_n] = \left(\ker\partial_1+\cdots+\ker\partial_{n-1}\right) \oplus
\Sym_n x^\rho$,
we deduce that $e_{(1^n)} AH_n e_{(n)}$ is a free $\Sym_n$-module
generated by $e_{(1^n)} x^\rho e_{(n)}$.
Moreover,
\begin{equation}\label{thickcapsneedthis}
e_{(1^n)} x^{\lambda} e_{(n)} =
\jonschi_{\lambda-\rho} e_{(1^n)} x^\rho
e_{(n)}
= e_{(1^n)} x^\rho
e_{(n)}
\jonschi_{\lambda-\rho}
\end{equation}
for any $\lambda \in \N^n$.
Similar statements hold when $e_{(n)}$ and $e_{(1^n)}$ are
interchanged.

\begin{proof}[Proof of Theorem~\ref{upgraded}]
Consider first the statement about $H_n$.
Exactly like in the proof of Theorem~\ref{trivial}, we need to
construct
morphisms $u$ and $v$ in 
$\Kar(\blue{\AH}\;\overline{\odot}\;\red{\AH})$
such that $
u \circ v = 
\Delta(\imath_n(e_{(n)}))$ and 
$v \circ u = \sum_{r=0}^n
\imath_{r,n}(e_{(r)}\otimes e_{(n-r)})$.
We set
\begin{align*}
u &:=\sum_{r=0}^n 
\binom{\,n\,}{\,r\,}^{-1} 
\sum_{\mu \in \mathcal P_{r,n}}
    \sigma_\mu^{-1} \circ 
\imath_{r,n}(e_{(r)}\otimes e_{(n-r)}),\\
v &:=
\sum_{r=0}^n 
         \sum_{\lambda \in \mathcal P_{r,n}}
\imath_{r,n}\left(e_{(r)}\otimes e_{(n-r)}\right)\circ
\imath_{r,n}
\bigg(\prod_{\substack{1 \leq i \leq r \\ 1 \leq j \leq
  n-r}}\left(1+\varepsilon_{i,j}(\lambda) y_{i,j}\right)\bigg) \circ\sigma_\lambda.
\end{align*}
Lemma~\ref{government} implies that
$u \circ v = \Delta(\imath_n(e_{(n)}))$.
Also
\begin{align*}
v \circ u &= \!
\sum_{r=0}^n \binom{\,n\,}{\,r\,}^{-1}
            \imath_{r,n}
\left(e_{(r)}\otimes e_{(n-r)}\right)\circ
\imath_{r,n}\bigg(\sum_{\lambda \in \mathcal P_{r,n}}
\prod_{\substack{1 \leq i \leq r \\ 1 \leq j \leq
  n-r}}\left(1+\varepsilon_{i,j}(\lambda) y_{i,j}\right)
  \bigg)
\circ \imath_{r,n}\left(e_{(r)}\otimes e_{(n-r)}\right).
\end{align*}
Using the analog of
(\ref{numbers}) for $AH_r \otimes AH_{n-r}$, this equals
$$
\sum_{r=0}^n 
            \imath_{r,n}\left(
e_{(r)}\otimes e_{(n-r)}\right)
\circ \frac{1}{n!}
\imath_{r,n}\bigg(\sum_{\pi \in \S_r \times \S_{n-r}}
\pi \bull 
\bigg(\sum_{\lambda \in \mathcal P_{r,n}}
\prod_{\substack{1 \leq i \leq r \\ 1 \leq j \leq
  n-r}}\left(1+\varepsilon_{i,j}(\lambda) y_{i,j}\right) \bigg)\bigg)
\circ \imath_{r,n}\left(e_{(r)}\otimes e_{(n-r)}\right).
$$
Then we use Lemma~\ref{wow} to see that this equals
the required
$\sum_{r=0}^n \imath_{r,n}(e_{(r)}\otimes e_{(n-r)})$.

For the statement about $E_n$,
we need
morphisms $u$ and $v$
such that $
u \circ v = 
\Delta(\imath_n(e_{(1^n)}))$ and 
$v \circ u = \sum_{r=0}^n
\imath_{r,n}(e_{(1^r)}\otimes e_{(1^{n-r})})$.
One takes
\begin{align*}
u &:=\sum_{r=0}^n \binom{\,n\,}{\,r\,}^{-1} 
\sum_{\mu \in \mathcal P_{r,n}}
 (-1)^{|\mu|}   \sigma_\mu^{-1} \circ 
\imath_{r,n}
\bigg(\prod_{\substack{1 \leq i \leq r \\ 1 \leq j \leq
  n-r}}\left(1-\varepsilon_{i,j}(\mu) y_{i,j}\right)\bigg) \circ
\imath_{r,n}(e_{(1^r)}\otimes e_{(1^{n-r})}),\\
v &:=
\sum_{r=0}^n 
         \sum_{\lambda \in \mathcal P_{r,n}}
(-1)^{|\lambda|}
\imath_{r,n}\left(e_{(1^r)}\otimes e_{(1^{n-r})}\right)\circ
\sigma_\lambda.
\end{align*}
The proof then proceeds like in the previous paragraph, using
(\ref{numbers2}) instead of (\ref{numbers}).
\end{proof}

\section{The degenerate Heisenberg category}\label{sdhc}

Although for us $\k$ is a field of characteristic zero, the following definition
makes sense for $\k$ that is any commutative ring. Moreover, all of the results recorded in this section are valid for any $\k$, including the definition of the categorical comultiplication in Theorem~\ref{comult} (but excluding (\ref{magic}) since $n!$ needs to be invertible for the underlying idempotents to be defined).

\begin{definition}[{\cite[Theorem 1.2]{B2}}]
\label{maindef}
The {\em (degenerate)  Heisenberg category}
$\Heis_k$
of central charge $k \in \Z$
is the strict $\k$-linear monoidal category
generated by objects
$\up$ and $\down$
and
morphisms
\begin{align*}
\mathord{
\begin{tikzpicture}[baseline = 0]
	\draw[->] (0.08,-.3) to (0.08,.4);
      \node at (0.08,0.05) {$\dot$};
\end{tikzpicture}
}
&:\up\rightarrow\up,
&\mathord{
\begin{tikzpicture}[baseline = 1mm]
	\draw[<-] (0.4,0.4) to[out=-90, in=0] (0.1,0);
	\draw[-] (0.1,0) to[out = 180, in = -90] (-0.2,0.4);
\end{tikzpicture}
}&:\unit\rightarrow\down\otimes\up
\:,
&\mathord{
\begin{tikzpicture}[baseline = 1mm]
	\draw[<-] (0.4,0) to[out=90, in=0] (0.1,0.4);
	\draw[-] (0.1,0.4) to[out = 180, in = 90] (-0.2,0);
\end{tikzpicture}
}&:\up\otimes\down\rightarrow\unit\:,\\
\mathord{
\begin{tikzpicture}[baseline = 0]
	\draw[->] (0.28,-.3) to (-0.28,.4);
	\draw[->] (-0.28,-.3) to (0.28,.4);
\end{tikzpicture}
}&:\up\otimes\up \rightarrow \up\otimes \up
\:,&
\mathord{
\begin{tikzpicture}[baseline = 1mm]
	\draw[-] (0.4,0.4) to[out=-90, in=0] (0.1,0);
	\draw[->] (0.1,0) to[out = 180, in = -90] (-0.2,0.4);
\end{tikzpicture}
}&:\unit\rightarrow\up\otimes \down
\:,
&
\mathord{
\begin{tikzpicture}[baseline = 1mm]
	\draw[-] (0.4,0) to[out=90, in=0] (0.1,0.4);
	\draw[->] (0.1,0.4) to[out = 180, in = 90] (-0.2,0);
\end{tikzpicture}
}&:\down\otimes\up\rightarrow\unit
\end{align*}
subject to certain relations.
To record these, define the sideways crossings
\begin{align}\label{rotate}
\mathord{
\begin{tikzpicture}[baseline = 0]
	\draw[<-] (0.28,-.3) to (-0.28,.4);
	\draw[->] (-0.28,-.3) to (0.28,.4);
\end{tikzpicture}
}
&:=
\mathord{
\begin{tikzpicture}[baseline = 0]
	\draw[->] (0.3,-.5) to (-0.3,.5);
	\draw[-] (-0.2,-.2) to (0.2,.3);
        \draw[-] (0.2,.3) to[out=50,in=180] (0.5,.5);
        \draw[->] (0.5,.5) to[out=0,in=90] (0.9,-.5);
        \draw[-] (-0.2,-.2) to[out=230,in=0] (-0.6,-.5);
        \draw[-] (-0.6,-.5) to[out=180,in=-90] (-0.9,.5);
\end{tikzpicture}
}\:,&
\mathord{
\begin{tikzpicture}[baseline = 0]
	\draw[->] (0.28,-.3) to (-0.28,.4);
	\draw[<-] (-0.28,-.3) to (0.28,.4);
\end{tikzpicture}
}
&:=
\mathord{
\begin{tikzpicture}[baseline = 0]
	\draw[<-] (0.3,.5) to (-0.3,-.5);
	\draw[-] (-0.2,.2) to (0.2,-.3);
        \draw[-] (0.2,-.3) to[out=130,in=180] (0.5,-.5);
        \draw[-] (0.5,-.5) to[out=0,in=270] (0.9,.5);
        \draw[-] (-0.2,.2) to[out=130,in=0] (-0.6,.5);
        \draw[->] (-0.6,.5) to[out=180,in=-270] (-0.9,-.5);
\end{tikzpicture}
}\:,
\end{align}
and
introduce
the fake bubbles for $a \leq k$ or $a \leq -k$, respectively, by setting
\begin{align}\label{f1}
\mathord{
\begin{tikzpicture}[baseline = 1.25mm]
  \draw[->] (0.2,0.2) to[out=90,in=0] (0,.4);
  \draw[-] (0,0.4) to[out=180,in=90] (-.2,0.2);
\draw[-] (-.2,0.2) to[out=-90,in=180] (0,0);
  \draw[-] (0,0) to[out=0,in=-90] (0.2,0.2);
   \node at (0.2,0.2) {$\dot$};
   \node at (0.7,0.2) {$\scriptstyle{a-k-1}$};
\end{tikzpicture}}
&:=
\det\left(
\mathord{
\begin{tikzpicture}[baseline = 1.25mm]
  \draw[<-] (0,0.4) to[out=180,in=90] (-.2,0.2);
  \draw[-] (0.2,0.2) to[out=90,in=0] (0,.4);
 \draw[-] (-.2,0.2) to[out=-90,in=180] (0,0);
  \draw[-] (0,0) to[out=0,in=-90] (0.2,0.2);
   \node at (-0.2,0.2) {$\dot$};
   \node at (-0.65,0.2) {$\scriptstyle{i-j+k}$};
\end{tikzpicture}
}\:
\right)_{i,j=1,\dots,a},
&\mathord{
\begin{tikzpicture}[baseline = 1.25mm]
  \draw[<-] (0,0.4) to[out=180,in=90] (-.2,0.2);
  \draw[-] (0.2,0.2) to[out=90,in=0] (0,.4);
 \draw[-] (-.2,0.2) to[out=-90,in=180] (0,0);
  \draw[-] (0,0) to[out=0,in=-90] (0.2,0.2);
   \node at (-0.2,0.2) {$\dot$};
   \node at (-0.7,0.2) {$\scriptstyle{a+k-1}$};
\end{tikzpicture}
}&:=
-\det\left(-\,\mathord{
\begin{tikzpicture}[baseline = 1.25mm]
  \draw[->] (0.2,0.2) to[out=90,in=0] (0,.4);
  \draw[-] (0,0.4) to[out=180,in=90] (-.2,.2);
\draw[-] (-.2,0.2) to[out=-90,in=180] (0,0);
  \draw[-] (0,0) to[out=0,in=-90] (0.2,0.2);
   \node at (0.2,0.2) {$\dot$};
   \node at (0.65,0.2) {$\scriptstyle{i-j-k}$};
\end{tikzpicture}
}\right)_{i,j=1,\dots,a},
\end{align}
interpreting the determinants as $\delta_{a,0}$ in case $a \leq 0$. 
Then the relations are as follows:
\begin{align}
\mathord{
\begin{tikzpicture}[baseline = -1mm]
	\draw[->] (0.28,0) to[out=90,in=-90] (-0.28,.6);
	\draw[->] (-0.28,0) to[out=90,in=-90] (0.28,.6);
	\draw[-] (0.28,-.6) to[out=90,in=-90] (-0.28,0);
	\draw[-] (-0.28,-.6) to[out=90,in=-90] (0.28,0);
\end{tikzpicture}
}&=
\mathord{
\begin{tikzpicture}[baseline = -1mm]
	\draw[->] (0.18,-.6) to (0.18,.6);
	\draw[->] (-0.18,-.6) to (-0.18,.6);
\end{tikzpicture}
}\:,
\qquad\mathord{
\begin{tikzpicture}[baseline = -1mm]
	\draw[<-] (0.45,.6) to (-0.45,-.6);
	\draw[->] (0.45,-.6) to (-0.45,.6);
        \draw[-] (0,-.6) to[out=90,in=-90] (-.45,0);
        \draw[->] (-0.45,0) to[out=90,in=-90] (0,0.6);
\end{tikzpicture}
}
=
\mathord{
\begin{tikzpicture}[baseline = -1mm]
	\draw[<-] (0.45,.6) to (-0.45,-.6);
	\draw[->] (0.45,-.6) to (-0.45,.6);
        \draw[-] (0,-.6) to[out=90,in=-90] (.45,0);
        \draw[->] (0.45,0) to[out=90,in=-90] (0,0.6);
\end{tikzpicture}
}\:,&
\mathord{
\begin{tikzpicture}[baseline = -1mm]
	\draw[<-] (0.25,.3) to (-0.25,-.3);
	\draw[->] (0.25,-.3) to (-0.25,.3);
 \node at (-0.12,-0.145) {$\dot$};
\end{tikzpicture}
}
&=
\mathord{
\begin{tikzpicture}[baseline = -1mm]
	\draw[<-] (0.25,.3) to (-0.25,-.3);
	\draw[->] (0.25,-.3) to (-0.25,.3);
     \node at (0.12,0.135) {$\dot$};
\end{tikzpicture}}
+\:\mathord{
\begin{tikzpicture}[baseline = -1mm]
 	\draw[->] (0.08,-.3) to (0.08,.3);
	\draw[->] (-0.28,-.3) to (-0.28,.3);
\end{tikzpicture}
}
\:,\label{hecke}\\
\label{rightadj}
\mathord{
\begin{tikzpicture}[baseline = 0]
  \draw[->] (0.3,0) to (0.3,.4);
	\draw[-] (0.3,0) to[out=-90, in=0] (0.1,-0.4);
	\draw[-] (0.1,-0.4) to[out = 180, in = -90] (-0.1,0);
	\draw[-] (-0.1,0) to[out=90, in=0] (-0.3,0.4);
	\draw[-] (-0.3,0.4) to[out = 180, in =90] (-0.5,0);
  \draw[-] (-0.5,0) to (-0.5,-.4);
\end{tikzpicture}
}
&=
\mathord{\begin{tikzpicture}[baseline=0]
  \draw[->] (0,-0.4) to (0,.4);
\end{tikzpicture}
}\:,
&\mathord{
\begin{tikzpicture}[baseline = 0]
  \draw[->] (0.3,0) to (0.3,-.4);
	\draw[-] (0.3,0) to[out=90, in=0] (0.1,0.4);
	\draw[-] (0.1,0.4) to[out = 180, in = 90] (-0.1,0);
	\draw[-] (-0.1,0) to[out=-90, in=0] (-0.3,-0.4);
	\draw[-] (-0.3,-0.4) to[out = 180, in =-90] (-0.5,0);
  \draw[-] (-0.5,0) to (-0.5,.4);
\end{tikzpicture}
}
&=
\mathord{\begin{tikzpicture}[baseline=0]
  \draw[<-] (0,-0.4) to (0,.4);
\end{tikzpicture}
}\:,\\
\label{bubbles}
\mathord{
\begin{tikzpicture}[baseline = 1.25mm]
  \draw[<-] (0,0.4) to[out=180,in=90] (-.2,0.2);
  \draw[-] (0.2,0.2) to[out=90,in=0] (0,.4);
 \draw[-] (-.2,0.2) to[out=-90,in=180] (0,0);
  \draw[-] (0,0) to[out=0,in=-90] (0.2,0.2);
   \node at (-0.2,0.2) {$\dot$};
   \node at (-0.7,0.2) {$\scriptstyle{a+k-1}$};
\end{tikzpicture}
}&=-\delta_{a,0} 1_\unit
\text{ if $-k < a \leq 0$,}&\mathord{
\begin{tikzpicture}[baseline = 1.25mm]
  \draw[->] (0.2,0.2) to[out=90,in=0] (0,.4);
  \draw[-] (0,0.4) to[out=180,in=90] (-.2,0.2);
\draw[-] (-.2,0.2) to[out=-90,in=180] (0,0);
  \draw[-] (0,0) to[out=0,in=-90] (0.2,0.2);
   \node at (0.2,0.2) {$\dot$};
   \node at (0.7,0.2) {$\scriptstyle{a-k-1}$};
\end{tikzpicture}
}&=
\delta_{a,0} 1_\unit
\text{ if $k < a \leq 0$,}\\
\label{curls}
\mathord{
\begin{tikzpicture}[baseline = -0.5mm]
	\draw[<-] (0,0.6) to (0,0.3);
	\draw[-] (0,0.3) to [out=-90,in=180] (.3,-0.2);
	\draw[-] (0.3,-0.2) to [out=0,in=-90](.5,0);
	\draw[-] (0.5,0) to [out=90,in=0](.3,0.2);
	\draw[-] (0.3,.2) to [out=180,in=90](0,-0.3);
	\draw[-] (0,-0.3) to (0,-0.6);
\end{tikzpicture}
}&=\delta_{k,0}
\:\mathord{
\begin{tikzpicture}[baseline = -0.5mm]
	\draw[<-] (0,0.6) to (0,-0.6);
\end{tikzpicture}
}
\:\text{ if $k \geq 0$,}
&\mathord{
\begin{tikzpicture}[baseline = -0.5mm]
	\draw[<-] (0,0.6) to (0,0.3);
	\draw[-] (0,0.3) to [out=-90,in=0] (-.3,-0.2);
	\draw[-] (-0.3,-0.2) to [out=180,in=-90](-.5,0);
	\draw[-] (-0.5,0) to [out=90,in=180](-.3,0.2);
	\draw[-] (-0.3,.2) to [out=0,in=90](0,-0.3);
	\draw[-] (0,-0.3) to (0,-0.6);
\end{tikzpicture}
}&=
\delta_{k,0}
\:\mathord{
\begin{tikzpicture}[baseline = -0.5mm]
	\draw[<-] (0,0.6) to (0,-0.6);
\end{tikzpicture}
}
\:\text{ if $k \leq 0$},
\\
\mathord{
\begin{tikzpicture}[baseline = 0mm]
	\draw[-] (0.28,0) to[out=90,in=-90] (-0.28,.6);
	\draw[->] (-0.28,0) to[out=90,in=-90] (0.28,.6);
	\draw[-] (0.28,-.6) to[out=90,in=-90] (-0.28,0);
	\draw[<-] (-0.28,-.6) to[out=90,in=-90] (0.28,0);
\end{tikzpicture}
}
&=
\mathord{
\begin{tikzpicture}[baseline = 0]
	\draw[->] (0.08,-.6) to (0.08,.6);
	\draw[<-] (-0.28,-.6) to (-0.28,.6);
\end{tikzpicture}
}
+\sum_{a,b \geq 0}
\mathord{
\begin{tikzpicture}[baseline = 1mm]
  \draw[<-] (0,0.4) to[out=180,in=90] (-.2,0.2);
  \draw[-] (0.2,0.2) to[out=90,in=0] (0,.4);
 \draw[-] (-.2,0.2) to[out=-90,in=180] (0,0);
  \draw[-] (0,0) to[out=0,in=-90] (0.2,0.2);
   \node at (-0.2,0.2) {$\dot$};
   \node at (-.78,0.2) {$\scriptstyle{-a-b-2}$};
\end{tikzpicture}
}
\mathord{
\begin{tikzpicture}[baseline=-.5mm]
	\draw[<-] (0.3,0.6) to[out=-90, in=0] (0,.1);
	\draw[-] (0,.1) to[out = 180, in = -90] (-0.3,0.6);
      \node at (0.44,-0.3) {$\scriptstyle{b}$};
	\draw[-] (0.3,-.6) to[out=90, in=0] (0,-0.1);
	\draw[->] (0,-0.1) to[out = 180, in = 90] (-0.3,-.6);
      \node at (0.27,-0.3) {$\dot$};
   \node at (0.27,0.3) {$\dot$};
   \node at (.43,.3) {$\scriptstyle{a}$};
\end{tikzpicture}}\:,&
\mathord{
\begin{tikzpicture}[baseline = 0mm]
	\draw[->] (0.28,0) to[out=90,in=-90] (-0.28,.6);
	\draw[-] (-0.28,0) to[out=90,in=-90] (0.28,.6);
	\draw[<-] (0.28,-.6) to[out=90,in=-90] (-0.28,0);
	\draw[-] (-0.28,-.6) to[out=90,in=-90] (0.28,0);
\end{tikzpicture}
}
&=\mathord{
\begin{tikzpicture}[baseline = 0]
	\draw[<-] (0.08,-.6) to (0.08,.6);
	\draw[->] (-0.28,-.6) to (-0.28,.6);
\end{tikzpicture}
}
+\sum_{a,b \geq 0}
\mathord{
\begin{tikzpicture}[baseline=-0.9mm]
	\draw[-] (0.3,0.6) to[out=-90, in=0] (0,0.1);
	\draw[->] (0,0.1) to[out = 180, in = -90] (-0.3,0.6);
      \node at (-0.4,0.3) {$\scriptstyle{a}$};
      \node at (-0.25,0.3) {$\dot$};
	\draw[<-] (0.3,-.6) to[out=90, in=0] (0,-0.1);
	\draw[-] (0,-0.1) to[out = 180, in = 90] (-0.3,-.6);
   \node at (-0.27,-0.4) {$\dot$};
   \node at (-.45,-.35) {$\scriptstyle{b}$};
\end{tikzpicture}}
\mathord{
\begin{tikzpicture}[baseline = 1mm]
  \draw[->] (0.2,0.2) to[out=90,in=0] (0,.4);
  \draw[-] (0,0.4) to[out=180,in=90] (-.2,0.2);
\draw[-] (-.2,0.2) to[out=-90,in=180] (0,0);
  \draw[-] (0,0) to[out=0,in=-90] (0.2,0.2);
   \node at (0.2,0.2) {$\dot$};
   \node at (.72,0.2) {$\scriptstyle{-a-b-2}$};
\end{tikzpicture}
}
\:.\label{sideways}
\end{align}
\end{definition}

We note that the Heisenberg category introduced originally by Khovanov in \cite{K} may be obtained from $\Heis_{-1}$ by 
evaluating the bubble
$\begin{tikzpicture}[baseline = 1.25mm]
  \draw[->] (0.2,0.2) to[out=90,in=0] (0,.4);
  \draw[-] (0,0.4) to[out=180,in=90] (-.2,0.2);
\draw[-] (-.2,0.2) to[out=-90,in=180] (0,0);
  \draw[-] (0,0) to[out=0,in=-90] (0.2,0.2);
   \node at (0.2,0.2) {$\dot$};
\end{tikzpicture}$ at zero; see the discussion after the statement of Theorem~\ref{t1} in the introduction.

As explained in the proof of \cite[Theorem 1.2]{B2}, the defining
relations of $\Heis_k$ imply that
the
following
is an isomorphism
in $\Add(\Heis_k)$:
\begin{align}
\label{invrel}
\left\{\quad
\begin{array}{ll}
\left[
\begin{array}{r}
\mathord{
\begin{tikzpicture}[baseline = 0]
	\draw[<-] (0.28,-.3) to (-0.28,.4);
	\draw[->] (-0.28,-.3) to (0.28,.4);
   \end{tikzpicture}
}\\
\mathord{
\begin{tikzpicture}[baseline = 1mm]
	\node at (0,.6){$\phantom.$};
\draw[<-] (0.4,0) to[out=90, in=0] (0.1,0.4);
	\draw[-] (0.1,0.4) to[out = 180, in = 90] (-0.2,0);
\end{tikzpicture}
}\\
\mathord{
\begin{tikzpicture}[baseline = 1mm]
	\node at (0,.6){$\phantom.$};
	\draw[<-] (0.4,0) to[out=90, in=0] (0.1,0.4);
	\draw[-] (0.1,0.4) to[out = 180, in = 90] (-0.2,0);
      \node at (-0.15,0.2) {$\dot$};
\end{tikzpicture}
}\\
\vdots\:\:\:\:\\
\!\!\!\!\mathord{
\begin{tikzpicture}[baseline = 1mm]
	\node at (0,.5){$\phantom.$};
	\draw[<-] (0.4,0) to[out=90, in=0] (0.1,0.4);
	\draw[-] (0.1,0.4) to[out = 180, in = 90] (-0.2,0);
     \node at (-0.5,0.2) {$\scriptstyle{k-1}$};
      \node at (-0.15,0.2) {$\dot$};
\end{tikzpicture}
}\:\end{array}
\right]
:
\up \otimes \down\stackrel{\sim}{\rightarrow}
\down \otimes \up \oplus \unit^{\oplus k}
&\text{if $k \geq
  0$},\\\\
\left[\:
\mathord{
\begin{tikzpicture}[baseline = 0]
	\draw[<-] (0.28,-.3) to (-0.28,.4);
	\draw[->] (-0.28,-.3) to (0.28,.4);
\end{tikzpicture}
}\:\:\:
\mathord{
\begin{tikzpicture}[baseline = -0.9mm]
	\draw[<-] (0.4,0.2) to[out=-90, in=0] (0.1,-.2);
	\draw[-] (0.1,-.2) to[out = 180, in = -90] (-0.2,0.2);
\end{tikzpicture}
}
\:\:\:
\mathord{
\begin{tikzpicture}[baseline = -0.9mm]
	\draw[<-] (0.4,0.2) to[out=-90, in=0] (0.1,-.2);
	\draw[-] (0.1,-.2) to[out = 180, in = -90] (-0.2,0.2);
      \node at (0.38,0) {$\dot$};
\end{tikzpicture}
}
\:\:\:\cdots
\:\:\:
\mathord{
\begin{tikzpicture}[baseline = -0.9mm]
	\draw[<-] (0.4,0.2) to[out=-90, in=0] (0.1,-.2);
	\draw[-] (0.1,-.2) to[out = 180, in = -90] (-0.2,0.2);
     \node at (0.83,0) {$\scriptstyle{-k-1}$};
      \node at (0.38,0) {$\dot$};
\end{tikzpicture}
}
\right]
:\up \otimes \down \oplus
\unit^{\oplus (-k)}
\stackrel{\sim}{\rightarrow}
 \down \otimes  \up&\text{if $k \leq
  0$}.\end{array}\right.
\end{align}
In fact, as in \cite[Definition 1.1]{B2}, $\Heis_k$ can be defined
equivalently as the strict $\k$-linear monoidal category generated by the morphisms
$\mathord{
\begin{tikzpicture}[baseline = 0]
	\draw[->] (0.08,-.1) to (0.08,.3);
      \node at (0.08,0.07) {$\dot$};
\end{tikzpicture}
}$,
$
\mathord{\begin{tikzpicture}[baseline = 0]
	\draw[->] (0.18,-.1) to (-0.18,.3);
	\draw[->] (-0.18,-.1) to (0.18,.3);
\end{tikzpicture}}\;$,
$\:\mathord{
\begin{tikzpicture}[baseline = .5mm]
	\draw[<-] (0.35,0.3) to[out=-90, in=0] (0.1,0);
	\draw[-] (0.1,0) to[out = 180, in = -90] (-0.15,0.3);
\end{tikzpicture}
}\:$
and $\:\mathord{
\begin{tikzpicture}[baseline = .5mm]
	\draw[<-] (0.35,0) to[out=90, in=0] (0.1,0.3);
	\draw[-] (0.1,0.3) to[out = 180, in = 90] (-0.15,0);
\end{tikzpicture}
}\:$ subject just to the relations (\ref{hecke})--(\ref{rightadj}) plus the requirement that the morphism
(\ref{invrel}) is invertible (where the rightward
crossing is defined as in (\ref{rotate})).
In the category defined in this way, there are then unique
morphisms
$\:\mathord{
\begin{tikzpicture}[baseline = .5mm]
	\draw[-] (0.35,0.3) to[out=-90, in=0] (0.1,0);
	\draw[->] (0.1,0) to[out = 180, in = -90] (-0.15,0.3);
\end{tikzpicture}
}\:$
and $\:\mathord{
\begin{tikzpicture}[baseline = .5mm]
	\draw[-] (0.35,0) to[out=90, in=0] (0.1,0.3);
	\draw[->] (0.1,0.3) to[out = 180, in = 90] (-0.15,0);
\end{tikzpicture}
}\:$
such that the other relations
from Definition~\ref{maindef} hold:

\begin{lemma}\label{lateaddition}
Suppose that $\mathcal C$ is a strict $\k$-linear monoidal category
containing objects
$\up$ and $\down$
and
morphisms
$\mathord{
\begin{tikzpicture}[baseline = 0]
	\draw[->] (0.08,-.1) to (0.08,.3);
      \node at (0.08,0.07) {$\dot$};
\end{tikzpicture}
}$,
$
\mathord{\begin{tikzpicture}[baseline = 0]
	\draw[->] (0.18,-.1) to (-0.18,.3);
	\draw[->] (-0.18,-.1) to (0.18,.3);
\end{tikzpicture}}\;$,
$\:\mathord{
\begin{tikzpicture}[baseline = .5mm]
	\draw[<-] (0.35,0.3) to[out=-90, in=0] (0.1,0);
	\draw[-] (0.1,0) to[out = 180, in = -90] (-0.15,0.3);
\end{tikzpicture}
}\:$
and $\:\mathord{
\begin{tikzpicture}[baseline = .5mm]
	\draw[<-] (0.35,0) to[out=90, in=0] (0.1,0.3);
	\draw[-] (0.1,0.3) to[out = 180, in = 90] (-0.15,0);
\end{tikzpicture}
}\:$
 satisfying (\ref{hecke})--(\ref{rightadj}).
If $\mathcal C$ contains morphisms
$\:\mathord{
\begin{tikzpicture}[baseline = .5mm]
	\draw[-] (0.35,0.3) to[out=-90, in=0] (0.1,0);
	\draw[->] (0.1,0) to[out = 180, in = -90] (-0.15,0.3);
\end{tikzpicture}
}\:$
and $\:\mathord{
\begin{tikzpicture}[baseline = .5mm]
	\draw[-] (0.35,0) to[out=90, in=0] (0.1,0.3);
	\draw[->] (0.1,0.3) to[out = 180, in = 90] (-0.15,0);
\end{tikzpicture}
}\:$
satisfying
(\ref{bubbles})--(\ref{sideways})
(for the sideways crossings and the
negatively dotted bubbles defined via
(\ref{rotate})--(\ref{f1})) then these two morphisms are uniquely determined.
\end{lemma}

\begin{proof}
This follows by the argument from the penultimate paragraph of the proof of
\cite[Theorem 1.2]{B2}.
\end{proof}

We will need various other relations in $\Heis_k$,
most of which are derived in \cite[Theorem 1.3]{B2}.
The relation (\ref{rightadj}) means that $\down$ is a right dual to $\up$.
It is also a left dual since the following relations hold:
\begin{align}\label{leftadj}
\mathord{
\begin{tikzpicture}[baseline = 0]
  \draw[-] (0.3,0) to (0.3,-.4);
	\draw[-] (0.3,0) to[out=90, in=0] (0.1,0.4);
	\draw[-] (0.1,0.4) to[out = 180, in = 90] (-0.1,0);
	\draw[-] (-0.1,0) to[out=-90, in=0] (-0.3,-0.4);
	\draw[-] (-0.3,-0.4) to[out = 180, in =-90] (-0.5,0);
  \draw[->] (-0.5,0) to (-0.5,.4);
\end{tikzpicture}
}
&=
\mathord{\begin{tikzpicture}[baseline=0]
  \draw[->] (0,-0.4) to (0,.4);
\end{tikzpicture}
}\:,
&\mathord{
\begin{tikzpicture}[baseline = 0]
  \draw[-] (0.3,0) to (0.3,.4);
	\draw[-] (0.3,0) to[out=-90, in=0] (0.1,-0.4);
	\draw[-] (0.1,-0.4) to[out = 180, in = -90] (-0.1,0);
	\draw[-] (-0.1,0) to[out=90, in=0] (-0.3,0.4);
	\draw[-] (-0.3,0.4) to[out = 180, in =90] (-0.5,0);
  \draw[->] (-0.5,0) to (-0.5,-.4);
\end{tikzpicture}
}
&=
\mathord{\begin{tikzpicture}[baseline=0]
  \draw[<-] (0,-0.4) to (0,.4);
\end{tikzpicture}
}\:.
\end{align}
This means that $\Heis_k$ is {\em rigid}.
Moreover, it is {\em strictly pivotal}: 
rotating diagrams through $180^\circ$ defines a 
strict $\k$-linear
monoidal isomorphism
\begin{equation}
*:\Heis_k \rightarrow \left(\left(\Heis_k\right)^{\op}\right)^{\rev},
\end{equation}
where $\op$ (resp., $\rev$) denotes the monoidal category with the same horizontal
composition and the opposite vertical composition (resp., the reversed
horizontal composition and the same vertical composition).
This follows due to the relations
\begin{align}
\mathord{
\begin{tikzpicture}[baseline = 0]
	\draw[<-] (0.08,-.3) to (0.08,.4);
      \node at (0.08,0.05) {$\dot$};
\end{tikzpicture}
}&:=
\mathord{
\begin{tikzpicture}[baseline = 0]
  \draw[->] (0.3,0) to (0.3,-.4);
	\draw[-] (0.3,0) to[out=90, in=0] (0.1,0.4);
	\draw[-] (0.1,0.4) to[out = 180, in = 90] (-0.1,0);
	\draw[-] (-0.1,0) to[out=-90, in=0] (-0.3,-0.4);
	\draw[-] (-0.3,-0.4) to[out = 180, in =-90] (-0.5,0);
  \draw[-] (-0.5,0) to (-0.5,.4);
   \node at (-0.1,0) {$\dot$};
\end{tikzpicture}
}=\:
\mathord{
\begin{tikzpicture}[baseline = 0]
  \draw[-] (0.3,0) to (0.3,.4);
	\draw[-] (0.3,0) to[out=-90, in=0] (0.1,-0.4);
	\draw[-] (0.1,-0.4) to[out = 180, in = -90] (-0.1,0);
	\draw[-] (-0.1,0) to[out=90, in=0] (-0.3,0.4);
	\draw[-] (-0.3,0.4) to[out = 180, in =90] (-0.5,0);
  \draw[->] (-0.5,0) to (-0.5,-.4);
   \node at (-0.1,0) {$\dot$};
\end{tikzpicture}
}\:,\\
\mathord{
\begin{tikzpicture}[baseline = 0]
	\draw[<-] (0.28,-.3) to (-0.28,.4);
	\draw[<-] (-0.28,-.3) to (0.28,.4);
\end{tikzpicture}
}
&:=
\mathord{
\begin{tikzpicture}[baseline = 0]
\draw[->] (1.3,.4) to (1.3,-1.2);
\draw[-] (-1.3,-.4) to (-1.3,1.2);
\draw[-] (.5,1.1) to [out=0,in=90] (1.3,.4);
\draw[-] (-.35,.4) to [out=90,in=180] (.5,1.1);
\draw[-] (-.5,-1.1) to [out=180,in=-90] (-1.3,-.4);
\draw[-] (.35,-.4) to [out=-90,in=0] (-.5,-1.1);
\draw[-] (.35,-.4) to [out=90,in=-90] (-.35,.4);
        \draw[-] (-0.35,-.5) to[out=0,in=180] (0.35,.5);
        \draw[->] (0.35,.5) to[out=0,in=90] (0.8,-1.2);
        \draw[-] (-0.35,-.5) to[out=180,in=-90] (-0.8,1.2);
\end{tikzpicture}
}
=\mathord{
\begin{tikzpicture}[baseline = 0]
	\draw[<-] (0.3,-.5) to (-0.3,.5);
	\draw[-] (-0.2,-.2) to (0.2,.3);
        \draw[-] (0.2,.3) to[out=50,in=180] (0.5,.5);
        \draw[->] (0.5,.5) to[out=0,in=90] (0.9,-.5);
        \draw[-] (-0.2,-.2) to[out=230,in=0] (-0.6,-.5);
        \draw[-] (-0.6,-.5) to[out=180,in=-90] (-0.9,.5);
\end{tikzpicture}
}=
\mathord{
\begin{tikzpicture}[baseline = 0]
	\draw[->] (0.3,.5) to (-0.3,-.5);
	\draw[-] (-0.2,.2) to (0.2,-.3);
        \draw[-] (0.2,-.3) to[out=130,in=180] (0.5,-.5);
        \draw[-] (0.5,-.5) to[out=0,in=270] (0.9,.5);
        \draw[-] (-0.2,.2) to[out=130,in=0] (-0.6,.5);
        \draw[->] (-0.6,.5) to[out=180,in=-270] (-0.9,-.5);
\end{tikzpicture}
}=
\mathord{
\begin{tikzpicture}[baseline = 0]
\draw[->] (-1.3,.4) to (-1.3,-1.2);
\draw[-] (1.3,-.4) to (1.3,1.2);
\draw[-] (-.5,1.1) to [out=180,in=90] (-1.3,.4);
\draw[-] (.35,.4) to [out=90,in=0] (-.5,1.1);
\draw[-] (.5,-1.1) to [out=0,in=-90] (1.3,-.4);
\draw[-] (-.35,-.4) to [out=-90,in=180] (.5,-1.1);
\draw[-] (-.35,-.4) to [out=90,in=-90] (.35,.4);
        \draw[-] (0.35,-.5) to[out=180,in=0] (-0.35,.5);
        \draw[->] (-0.35,.5) to[out=180,in=90] (-0.8,-1.2);
        \draw[-] (0.35,-.5) to[out=0,in=-90] (0.8,1.2);
\end{tikzpicture}
}\:.
\end{align}
Informally, these relations mean that dots and crossings slide over
cups and caps.
Applying $*$ to the relations (\ref{hecke}) and (\ref{curls})
gives
\begin{align}\label{symmetricalt}
\mathord{
\begin{tikzpicture}[baseline = -1mm]
	\draw[-] (0.28,0) to[out=90,in=-90] (-0.28,.6);
	\draw[-] (-0.28,0) to[out=90,in=-90] (0.28,.6);
	\draw[<-] (0.28,-.6) to[out=90,in=-90] (-0.28,0);
	\draw[<-] (-0.28,-.6) to[out=90,in=-90] (0.28,0);
\end{tikzpicture}
}&=
\mathord{
\begin{tikzpicture}[baseline = -1mm]
	\draw[<-] (0.18,-.6) to (0.18,.6);
	\draw[<-] (-0.18,-.6) to (-0.18,.6);
\end{tikzpicture}
}\:,\qquad\quad
\mathord{
\begin{tikzpicture}[baseline = -1mm]
	\draw[->] (0.45,.6) to (-0.45,-.6);
	\draw[<-] (0.45,-.6) to (-0.45,.6);
        \draw[<-] (0,-.6) to[out=90,in=-90] (-.45,0);
        \draw[-] (-0.45,0) to[out=90,in=-90] (0,0.6);
\end{tikzpicture}
}
=
\mathord{
\begin{tikzpicture}[baseline = -1mm]
	\draw[->] (0.45,.6) to (-0.45,-.6);
	\draw[<-] (0.45,-.6) to (-0.45,.6);
        \draw[<-] (0,-.6) to[out=90,in=-90] (.45,0);
        \draw[-] (0.45,0) to[out=90,in=-90] (0,0.6);
\end{tikzpicture}
}\:,&
\mathord{
\begin{tikzpicture}[baseline = -1mm]
	\draw[->] (0.25,.3) to (-0.25,-.3);
	\draw[<-] (0.25,-.3) to (-0.25,.3);
     \node at (0.12,0.135) {$\dot$};
\end{tikzpicture}
}
&=
\mathord{
\begin{tikzpicture}[baseline = -1mm]
	\draw[->] (0.25,.3) to (-0.25,-.3);
	\draw[<-] (0.25,-.3) to (-0.25,.3);
 \node at (-0.12,-0.145) {$\dot$};
\end{tikzpicture}}
+\:\mathord{
\begin{tikzpicture}[baseline = -1mm]
 	\draw[<-] (0.08,-.3) to (0.08,.3);
	\draw[<-] (-0.28,-.3) to (-0.28,.3);
\end{tikzpicture}
}
\:,\\
\mathord{
\begin{tikzpicture}[baseline = -0.5mm]
	\draw[-] (0,0.6) to (0,0.3);
	\draw[-] (0,0.3) to [out=-90,in=180] (.3,-0.2);
	\draw[-] (0.3,-0.2) to [out=0,in=-90](.5,0);
	\draw[-] (0.5,0) to [out=90,in=0](.3,0.2);
	\draw[-] (0.3,.2) to [out=180,in=90](0,-0.3);
	\draw[->] (0,-0.3) to (0,-0.6);
\end{tikzpicture}
}&=\delta_{k,0}
\:\mathord{
\begin{tikzpicture}[baseline = -0.5mm]
	\draw[<-] (0,0.6) to (0,-0.6);
\end{tikzpicture}
}
\:\text{ if $k \leq 0$,}
&\mathord{
\begin{tikzpicture}[baseline = -0.5mm]
	\draw[-] (0,0.6) to (0,0.3);
	\draw[-] (0,0.3) to [out=-90,in=0] (-.3,-0.2);
	\draw[-] (-0.3,-0.2) to [out=180,in=-90](-.5,0);
	\draw[-] (-0.5,0) to [out=90,in=180](-.3,0.2);
	\draw[-] (-0.3,.2) to [out=0,in=90](0,-0.3);
	\draw[->] (0,-0.3) to (0,-0.6);
\end{tikzpicture}
}&=
\delta_{k,0}
\:\mathord{
\begin{tikzpicture}[baseline = -0.5mm]
	\draw[<-] (0,0.6) to (0,-0.6);
\end{tikzpicture}
}
\:\text{ if $k \geq 0$}.\label{curls2}
\end{align}
There is another useful symmetry
\begin{equation}\label{Omega}
\Omega_k:\Heis_{k} \stackrel{\sim}{\rightarrow} \left(\Heis_{-k}\right)^{\operatorname{op}},
\end{equation}
which
sends a morphism in
$\Heis_{k}$ represented by some string diagram to the morphism
in $\Heis_{-k}$ obtained
by reflecting this diagram in a horizontal axis then multiplying
by $(-1)^{x+y}$,
where $x$
is the total number of crossings and $y$ is the total number of
leftward cups and caps in the diagram (including ones in fake
bubbles);
see \cite[Lemma 2.1]{B2}.

\begin{remark}\label{spicyramen}
Using $*$ and $\Omega_k$, one can deduce several more
equivalent presentations for $\Heis_k$. For example, it may be defined
by the same generating objects and morphisms as in
Definition~\ref{maindef}
subject to the relations (\ref{hecke}), (\ref{leftadj}),
(\ref{bubbles}), (\ref{curls}) and (\ref{sideways}); i.e., we have
traded the right adjunction relation (\ref{rightadj}) for the left
adjunction relation
(\ref{leftadj}).
Alternatively, one could replace the generating morphisms given by the
upward dot and crossing with the downward dot and crossing, taking the
relations (\ref{symmetricalt}), 
 (\ref{leftadj}),
(\ref{bubbles}), (\ref{curls2}) and (\ref{sideways}), 
where the sideways
crossings are obtained by rotating the downward one in an analogous
way to (\ref{rotate}).
There are also alternative versions of both of these presentations based on
an ``inversion relation'' along the
lines of the presentation explained after (\ref{invrel}).
\end{remark}

Since the relations (\ref{symmetric}) and (\ref{heckea}) hold in
$\Heis_k$,
there is a
strict $\k$-linear monoidal functor $\imath:\AH \rightarrow \Heis_k$
sending diagrams in $\AH$ to the same diagrams viewed instead as
morphisms in $\Heis_k$; this is actually an inclusion thanks to the basis theorem established in
Theorem~\ref{basis} below, but we will not use this fact here.
In particular, this means that there is an algebra homomorphism
\begin{equation}\label{im}
\imath_n:\dH_n
\rightarrow \End_{\Heis_k}(\up^{\otimes n})
\end{equation}
sending
$s_i$ to the crossing of the $i$th and $(i+1)$th strings, and $x_j$ to
the dot on the $j$th string.
Using (\ref{symmetricalt}), one sees also that there is an algebra homomorphism
\begin{equation}\label{jm}
\jmath_n:\dH_n
\rightarrow \End_{\Heis_k}(\down^{\otimes n})
\end{equation}
sending
$-s_i$ to the crossing of the $i$th and $(i+1)$th strings, and $x_j$ to
the dot on the $j$th string.
Note $\imath_n$ and $\jmath_n$ are related by the formula $\jmath_n = \Omega_k \circ \imath_n
\circ \tau$ 
where 
$\tau:\dH_n \rightarrow \dH_n$ is the antiautomorphism which is the
identity on each of the generators $s_i$ and $x_j$.

The bubbles (both genuine and fake) satisfy the
{\em infinite Grassmannian relations}:
\begin{align}
\mathord{
\begin{tikzpicture}[baseline = 1.25mm]
  \draw[->] (0.2,0.2) to[out=90,in=0] (0,.4);
  \draw[-] (0,0.4) to[out=180,in=90] (-.2,0.2);
\draw[-] (-.2,0.2) to[out=-90,in=180] (0,0);
  \draw[-] (0,0) to[out=0,in=-90] (0.2,0.2);
   \node at (0.2,0.2) {$\dot$};
   \node at (0.4,0.2) {$\scriptstyle{a}$};
\end{tikzpicture}
}&=\delta_{a,-k-1}1_\unit\text{ if $a < -k$,}\:\:
\mathord{
\begin{tikzpicture}[baseline = 1.25mm]
  \draw[<-] (0,0.4) to[out=180,in=90] (-.2,0.2);
  \draw[-] (0.2,0.2) to[out=90,in=0] (0,.4);
 \draw[-] (-.2,0.2) to[out=-90,in=180] (0,0);
  \draw[-] (0,0) to[out=0,in=-90] (0.2,0.2);
   \node at (-0.2,0.2) {$\dot$};
   \node at (-0.4,0.2) {$\scriptstyle{a}$};
\end{tikzpicture}
}=
-\delta_{a,k-1}1_\unit\text{ if $a < k$,}
&
\sum_{b\in \Z}
\begin{array}{l}
\mathord{
\begin{tikzpicture}[baseline = 1.25mm]
  \draw[<-] (0,0.4) to[out=180,in=90] (-.2,0.2);
  \draw[-] (0.2,0.2) to[out=90,in=0] (0,.4);
 \draw[-] (-.2,0.2) to[out=-90,in=180] (0,0);
  \draw[-] (0,0) to[out=0,in=-90] (0.2,0.2);
   \node at (-0.2,0.2) {$\dot$};
   \node at (-.38,0.2) {$\scriptstyle{b}$};
\end{tikzpicture}
}\\
\mathord{
\begin{tikzpicture}[baseline = 1.25mm]
  \draw[->] (0.2,0.2) to[out=90,in=0] (0,.4);
  \draw[-] (0,0.4) to[out=180,in=90] (-.2,0.2);
\draw[-] (-.2,0.2) to[out=-90,in=180] (0,0);
  \draw[-] (0,0) to[out=0,in=-90] (0.2,0.2);
   \node at (0.2,0.2) {$\dot$};
   \node at (.7,0.2) {$\scriptstyle{a-b-2}$};
\end{tikzpicture}
}\end{array}\!\!\!
&=-\delta_{a,0} 1_\unit,\label{ig}
\end{align}
for any $a \in \Z$.
For an indeterminate $w$, let
\begin{align}\label{igq}
\anticlockleft\,(w) &:= \sum_{n\in\Z}
\mathord{
\begin{tikzpicture}[baseline = 1.25mm]
  \draw[-] (0,0.4) to[out=180,in=90] (-.2,0.2);
  \draw[->] (0.2,0.2) to[out=90,in=0] (0,.4);
 \draw[-] (-.2,0.2) to[out=-90,in=180] (0,0);
  \draw[-] (0,0) to[out=0,in=-90] (0.2,0.2);
   \node at (0.2,0.2) {$\dot$};
   \node at (0.4,0.2) {$\scriptstyle{n}$};
\end{tikzpicture}
}\: w^{-n-1} \in w^k 1_\unit + w^{k-1}\End_{\Heis_k}(\unit)\llbracket w^{-1}\rrbracket,\\
\clockright\,(w)&:= -\sum_{n \in\Z}
\mathord{
\begin{tikzpicture}[baseline = 1.25mm]
  \draw[<-] (0,0.4) to[out=180,in=90] (-.2,0.2);
  \draw[-] (0.2,0.2) to[out=90,in=0] (0,.4);
 \draw[-] (-.2,0.2) to[out=-90,in=180] (0,0);
  \draw[-] (0,0) to[out=0,in=-90] (0.2,0.2);
   \node at (-0.2,0.2) {$\dot$};
   \node at (-0.4,0.2) {$\scriptstyle{n}$};
\end{tikzpicture}
} \:w^{-n-1}\in w^{-k} 1_\unit +w^{-k-1}\End_{\Heis_k}(\unit)\llbracket w^{-1}\rrbracket.\label{igp}
\end{align}
Then the infinite Grassmannian relation implies that
\begin{equation}\label{igproper}
\anticlockleft\,(w)\; \clockright\,(w) = 1_\unit.
\end{equation}
Up to the choice of normalization, this is the well-known identity from \cite[(I.2.6)]{Mac}
relating elementary and complete symmetric functions.
It follows that there is a
well-defined algebra homomorphism
\begin{align}\label{beta}
\beta:\Sym &\rightarrow \End_{\Heis_k}(\unit),
&\e_n&\mapsto
\mathord{

}\:,
\end{align}
where $\Sym := \k\otimes_{\Z} \SymZ$ denotes the algebra of
symmetric functions over $\k$.
Using this dictionary, one can also make sense of the determinantal formulae (\ref{f1})
used to define the fake bubbles in Definition~\ref{maindef}: they are
are a formal consequence of another well-known symmetric functions
identity from \cite[Exercise I.2.8]{Mac}.

There are three more essential relations:
the {\em curl relations}
\begin{align}\label{dog1}
\mathord{
%
}
\end{align}
for all $a \geq 0$,
the {\em alternating braid relation}
\begin{align}
\mathord{
%
}
\end{align}
for $a \in \Z$.

For the remainder of the section,
we assume that $k = \blue{l} + \red{m}$ for integers $\blue{l},\red{m}
\in \Z$.
Let $\blue{\Heis_l} \odot \red{\Heis_m}$ be the symmetric product
defined as in Section~\ref{newsec}. Now there are additional two-color crossings
$$
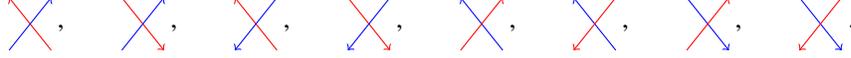
\mathord{
%
}\:\right)^{-1}\:.
$$
These satisfy many commuting relations such as (\ref{braidlike}),
(\ref{dotlike}), and also ``pitchfork relations" like the following:
\begin{align}\label{cuplike}
\mathord{
%
}
\:.
\end{align}
We extend the notation (\ref{dashit}) to
$\blue{\Heis_l}\odot\red{\Heis_m}$ in the obvious way.
Let
$\blue{\Heis_l} \;\overline{\odot}\; \red{\Heis_m}$ be the strict
$\k$-linear monoidal category obtained from
$\blue{\Heis_l}\odot\red{\Heis_m}$ by localizing at
$\mathord{
%
}
$.
As in (\ref{dumb}), we denote the two-sided inverse of this morphism
by a solid dumbbell.
Then we introduce the following shorthands
which we refer to as {\em internal bubbles}:
\begin{align}
\mathord{
%
}
\:.\label{odd2}
\end{align}
The category
$\blue{\Heis_l} \;\overline{\odot}\; \red{\Heis_m}$ is strictly
pivotal with duality functor
\begin{equation}
*:
\blue{\Heis_l} \;\overline{\odot}\; \red{\Heis_m}
\rightarrow
\left(\left(\blue{\Heis_l} \;\overline{\odot}\;
    \red{\Heis_m}\right)^\op\right)^\rev
\end{equation}
defined by rotating diagrams through $180^\circ$.
In particular, the left mates of the internal bubbles
(\ref{odd1})--(\ref{odd2}) are equal to their right mates:
\[
    %
    \ .
\]
We denote these mates by
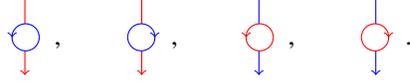
\begin{equation*}
\mathord{
%
}\:.
\end{equation*}
This definition ensures that internal bubbles commute past cups
and caps in all possible configurations.
For example:
\begin{align*}
\mathord{
%
}
\:.
\end{align*}
The dotted line notation (\ref{dashit}) obviously makes sense between points on downward
strings as well as on upward strings. We extend the dumbbell notation to
such situations
 by setting
\begin{align*}
\mathord{
%
}\:.
\end{align*}
These are two-sided inverses of the morphisms represented by the same
diagrams but with dotted lines replacing the solid ones in the dumbbells.
The morphisms on the right hand sides of the
following are also such inverses, hence, these are true equations
due to the
uniqueness of inverses:
\begin{align*}
\mathord{
%
}\:.
\end{align*}
From this discussion, it follows that dumbbells commute over cups and
caps in any configuration.
One can then deduce many other commuting relations, such as:
\begin{align*}
\mathord{
%
}
\:.
\end{align*}
We will appeal to these sorts of relation without further mention.
Finally,
there are two more useful symmetries
\begin{align}
\eta:
\blue{\Heis_l}\;\overline{\odot}\;\red{\Heis_m}
&\stackrel{\sim}{\rightarrow}
\blue{\Heis_m}\;\overline{\odot}\;\red{\Heis_l},\\
\Omega_{\blue{l}|\red{m}}:
\blue{\Heis_{l}}\;\overline{\odot}\;\red{\Heis_{m}}&\stackrel{\sim}{\rightarrow}
\left(\blue{\Heis_{-l}}\;\overline{\odot}\;\red{\Heis_{-m}}\right)^{\operatorname{op}}.
\end{align}
The first of these is
defined on diagrams by switching the colors blue and red
then multiplying by $(-1)^z$ where $z$ is the total
number of dumbbells in the picture;
it interchanges the internal bubbles in (\ref{odd1}) with the
ones in (\ref{odd2}).
The second takes a diagram to its mirror image in a horizontal axis
multiplied by
$(-1)^{x+y}$ where $x$ is the number of one-colored crossings and $y$
is the number of leftward cups and caps (including ones in fake and
internal bubbles). The only additional thing that needs to be used to
see that this is well defined beyond what was already checked for
(\ref{Omega}) is that
$\mathord{
\begin{tikzpicture}[baseline = -1mm]
 	\draw[<-,redp] (0.18,-.25) to (0.18,.2);
	\draw[<-,bluep] (-0.38,-.25) to (-0.38,.2);
	\draw[-,dotted] (-0.38,0.01) to (0.18,0.01);
     \node at (0.18,0) {$\dot$};
     \node at (-0.38,0.01) {$\dot$};
\end{tikzpicture}}
$
is invertible.
All of the symmetries $*$, $\eta$ and $\Omega_{\blue{l}|\red{m}}$
extend canonically to the Karoubi envelope.

\begin{theorem}\label{comult}
For $k = \blue{l}+\red{m}$
as above, there is a unique strict $\k$-linear monoidal functor
$$\Delta_{\blue{l}|\red{m}}:
\Heis_k \rightarrow
\Add\left(\blue{\Heis_l} \;\overline{\odot}\; \red{\Heis_m}\right)$$
such that $\up \mapsto \blueup \oplus \redup$,
$\down \mapsto \bluedown\oplus\reddown$, and on morphisms
\begin{align}\label{com0}
\mathord{
\begin{tikzpicture}[baseline = -.6mm]
	\draw[->] (0.08,-.3) to (0.08,.3);
      \node at (0.08,0) {$\dot$};
\end{tikzpicture}
}
&\mapsto
\mathord{
\begin{tikzpicture}[baseline = -.6mm]
	\draw[->,bluep] (0.08,-.3) to (0.08,.3);
      \node at (0.08,0) {$\color{blue}\dot$};
\end{tikzpicture}
}
+
\mathord{
\begin{tikzpicture}[baseline = -.6mm]
	\draw[->,redp] (0.08,-.3) to (0.08,.3);
      \node at (0.08,0) {$\color{red}\dot$};
\end{tikzpicture}
}\:,
\qquad\quad
\mathord{
\begin{tikzpicture}[baseline = 1mm]
	\draw[<-] (0.4,0) to[out=90, in=0] (0.1,0.4);
	\draw[-] (0.1,0.4) to[out = 180, in = 90] (-0.2,0);
\end{tikzpicture}
}\mapsto
\mathord{
\begin{tikzpicture}[baseline = 1mm]
	\draw[<-,bluep] (0.4,0) to[out=90, in=0] (0.1,0.4);
	\draw[-,bluep] (0.1,0.4) to[out = 180, in = 90] (-0.2,0);
\end{tikzpicture}
}+
\mathord{
\begin{tikzpicture}[baseline = 1mm]
	\draw[<-,redp] (0.4,0) to[out=90, in=0] (0.1,0.4);
	\draw[-,redp] (0.1,0.4) to[out = 180, in = 90] (-0.2,0);
\end{tikzpicture}
}\:,
\qquad\qquad
\mathord{
\begin{tikzpicture}[baseline = 1mm]
	\draw[<-] (0.4,0.4) to[out=-90, in=0] (0.1,0);
	\draw[-] (0.1,0) to[out = 180, in = -90] (-0.2,0.4);
\end{tikzpicture}
}
\mapsto
\mathord{
\begin{tikzpicture}[baseline = 1mm]
	\draw[<-,bluep] (0.4,0.4) to[out=-90, in=0] (0.1,0);
	\draw[-,bluep] (0.1,0) to[out = 180, in = -90] (-0.2,0.4);
\end{tikzpicture}
}+\mathord{
\begin{tikzpicture}[baseline = 1mm]
	\draw[<-,redp] (0.4,0.4) to[out=-90, in=0] (0.1,0);
	\draw[-,redp] (0.1,0) to[out = 180, in = -90] (-0.2,0.4);
\end{tikzpicture}
}\:,
\\\label{com1}
\mathord{
\begin{tikzpicture}[baseline = -.6mm]
	\draw[->] (0.28,-.3) to (-0.28,.3);
	\draw[->] (-0.28,-.3) to (0.28,.3);
\end{tikzpicture}
}&\mapsto
\mathord{
\begin{tikzpicture}[baseline = -.6mm]
	\draw[->,bluep] (0.28,-.3) to (-0.28,.3);
	\draw[->,bluep] (-0.28,-.3) to (0.28,.3);
\end{tikzpicture}
}+
\mathord{
\begin{tikzpicture}[baseline = -.6mm]
	\draw[->,redp] (0.28,-.3) to (-0.28,.3);
	\draw[->,redp] (-0.28,-.3) to (0.28,.3);
\end{tikzpicture}
}
+
\mathord{
\begin{tikzpicture}[baseline = -.6mm]
	\draw[->,redp] (0.28,-.3) to (-0.28,.3);
	\draw[->,bluep] (-0.28,-.3) to (0.28,.3);
\end{tikzpicture}
}+
\mathord{
\begin{tikzpicture}[baseline = -.6mm]
	\draw[->,bluep] (0.28,-.3) to (-0.28,.3);
	\draw[->,redp] (-0.28,-.3) to (0.28,.3);
\end{tikzpicture}
}-
\mathord{
\begin{tikzpicture}[baseline = -.6mm]
	\draw[->,redp] (0.28,-.3) to (-0.28,.3);
	\draw[->,bluep] (-0.28,-.3) to (0.28,.3);
	\draw[-] (-0.15,-.16) to (0.15,-.16);
     \node at (0.15,-.17) {$\dot$};
     \node at (-0.15,-.17) {$\dot$};
\end{tikzpicture}
}
+
\mathord{
\begin{tikzpicture}[baseline = -.6mm]
	\draw[->,bluep] (0.28,-.3) to (-0.28,.3);
	\draw[->,redp] (-0.28,-.3) to (0.28,.3);
	\draw[-] (-0.15,-.16) to (0.15,-.16);
     \node at (0.15,-.17) {$\dot$};
     \node at (-0.15,-.17) {$\dot$};
\end{tikzpicture}
}
-\mathord{
\begin{tikzpicture}[baseline =-.6mm]
 	\draw[->,redp] (0.2,-.3) to (0.2,.3);
	\draw[->,bluep] (-0.2,-.3) to (-0.2,.3);
	\draw[-] (-0.2,0.01) to (0.2,0.01);
     \node at (0.2,0.0) {$\dot$};
     \node at (-0.2,0.0) {$\dot$};
\end{tikzpicture}
}+
\mathord{
\begin{tikzpicture}[baseline = -0.6mm]
 	\draw[->,bluep] (0.2,-.3) to (0.2,.3);
	\draw[->,redp] (-0.2,-.3) to (-0.2,.3);
	\draw[-] (-0.2,0.01) to (0.2,0.01);
     \node at (0.2,0.0) {$\dot$};
     \node at (-0.2,0.0) {$\dot$};
\end{tikzpicture}
}\:.
\end{align}
We have that 
$\eta \circ \Delta_{\blue{l}|\red{m}} =
\Delta_{\blue{m}|\red{l}}$.
Moreover, $\Delta_{\blue{l}|\red{m}}$ satisfies the following for all $a \in \Z$:
\begin{align}\label{com4}
\mathord{
\begin{tikzpicture}[baseline = 1mm]
	\draw[-] (0.4,0) to[out=90, in=0] (0.1,0.4);
	\draw[->] (0.1,0.4) to[out = 180, in = 90] (-0.2,0);
\end{tikzpicture}
}&\mapsto
\mathord{
\begin{tikzpicture}[baseline = 1mm]
	\draw[-,bluep] (0.4,-.2) to[out=90, in=0] (0.1,0.6);
	\draw[->,bluep] (0.1,0.6) to[out = 180, in = 90] (-0.2,-.2);
     \node at (0.4,.3) {$\anticlockleftred$};
\end{tikzpicture}
}+
\mathord{
\begin{tikzpicture}[baseline = 1mm]
	\draw[-,redp] (0.4,-.2) to[out=90, in=0] (0.1,0.6);
	\draw[->,redp] (0.1,0.6) to[out = 180, in = 90] (-0.2,-.2);
     \node at (0.4,.3) {$\anticlockleftblue$};
\end{tikzpicture}
}\:,
&\mathord{
\begin{tikzpicture}[baseline = 1mm]
	\draw[-] (0.4,0.4) to[out=-90, in=0] (0.1,0);
	\draw[->] (0.1,0) to[out = 180, in = -90] (-0.2,0.4);
\end{tikzpicture}
}&\mapsto
-\mathord{
\begin{tikzpicture}[baseline = 1mm]
	\draw[-,bluep] (0.4,0.6) to[out=-90, in=0] (0.1,-0.2);
	\draw[->,bluep] (0.1,-0.2) to[out = 180, in = -90] (-0.2,0.6);
     \node at (-0.2,.2) {$\clockrightred$};
\end{tikzpicture}
}-\mathord{
\begin{tikzpicture}[baseline = 1mm]
	\draw[-,redp] (0.4,0.6) to[out=-90, in=0] (0.1,-0.2);
	\draw[->,redp] (0.1,-0.2) to[out = 180, in = -90] (-0.2,0.6);
     \node at (-0.2,.2) {$\clockrightblue$};
\end{tikzpicture}
}\:,\\\label{com5}
\mathord{\begin{tikzpicture}[baseline=-1mm]
\node at (0.04,0.05) {$\anticlockleft\,{\scriptstyle a}$};
\node at (.1,0) {$\dot$};
\end{tikzpicture}}
&\mapsto
\sum_{b \in \Z}
\begin{array}{l}
\mathord{
\begin{tikzpicture}[baseline = 1.25mm]
  \draw[-,bluep] (0,0.4) to[out=180,in=90] (-.2,0.2);
  \draw[->,bluep] (0.2,0.2) to[out=90,in=0] (0,.4);
 \draw[-,bluep] (-.2,0.2) to[out=-90,in=180] (0,0);
  \draw[-,bluep] (0,0) to[out=0,in=-90] (0.2,0.2);
   \node at (-0.2,0.2) {$\color{blue}\dot$};
   \node at (-.38,0.2) {$\color{blue}\scriptstyle{b}$};
\end{tikzpicture}
}\\
\mathord{
\begin{tikzpicture}[baseline = 1.25mm]
  \draw[->,redp] (0.2,0.2) to[out=90,in=0] (0,.4);
  \draw[-,redp] (0,0.4) to[out=180,in=90] (-.2,0.2);
\draw[-,redp] (-.2,0.2) to[out=-90,in=180] (0,0);
  \draw[-,redp] (0,0) to[out=0,in=-90] (0.2,0.2);
   \node at (0.2,0.2) {$\color{red}\dot$};
   \node at (.65,0.2) {$\color{red}\scriptstyle{a-b-1}$};
\end{tikzpicture}
}\end{array},&
\mathord{\begin{tikzpicture}[baseline=-1mm]
\node at (-0.04,.05) {${\scriptstyle a}\,\clockright$};
\node at (-.1,0) {$\dot$};
\end{tikzpicture}}
&\mapsto
-\sum_{b \in \Z}
\begin{array}{l}
\mathord{
\begin{tikzpicture}[baseline = 1.25mm]
  \draw[<-,bluep] (0,0.4) to[out=180,in=90] (-.2,0.2);
  \draw[-,bluep] (0.2,0.2) to[out=90,in=0] (0,.4);
 \draw[-,bluep] (-.2,0.2) to[out=-90,in=180] (0,0);
  \draw[-,bluep] (0,0) to[out=0,in=-90] (0.2,0.2);
   \node at (-0.2,0.2) {$\color{blue}\dot$};
   \node at (-.38,0.2) {$\color{blue}\scriptstyle{b}$};
\end{tikzpicture}
}\\
\mathord{
\begin{tikzpicture}[baseline = 1.25mm]
  \draw[-,redp] (0.2,0.2) to[out=90,in=0] (0,.4);
  \draw[<-,redp] (0,0.4) to[out=180,in=90] (-.2,0.2);
\draw[-,redp] (-.2,0.2) to[out=-90,in=180] (0,0);
  \draw[-,redp] (0,0) to[out=0,in=-90] (0.2,0.2);
   \node at (0.2,0.2) {$\color{red}\dot$};
   \node at (.65,0.2) {$\color{red}\scriptstyle{a-b-1}$};
\end{tikzpicture}
}\end{array}.
\end{align}
Finally, extending
 $\Delta_{\blue{l}|\red{m}}$ to the
Karoubi envelopes in the canonical way, 
we have that 
\begin{align}\label{magic}
\Delta_{\blue{l}|\red{m}}(H_n^{\pm}) &\cong \bigoplus_{r=0}^n
{\color{blue}H_{n-r}^{\pm}} \otimes {\color{red}H_{r}^{\pm}},&
\Delta_{\blue{l}|\red{m}}(E_n^{\pm}) &\cong \bigoplus_{r=0}^n
{\color{blue}E_{n-r}^{\pm}} \otimes {\color{red}E_{r}^{\pm}}.
\end{align}
\end{theorem}

\begin{remark}\label{altsss}
As in Remarks~\ref{alts} and \ref{altss}, the categorical comultiplication
$\Delta_{\blue{l}|\red{m}}$ is coassociative in the appropriate sense.
It also seems worth pointing out that $\Delta_{\blue{l}|\red{m}}$ does {\em
  not} commute either with the duality $*$ or the involution
$\Omega$.
In fact, either of the monoidal functors $* \circ \Delta_{\blue{l}|\red{m}}
\circ *$ or
$\Omega_{\blue{-l}|\red{-m}}
\circ \Delta_{\blue{-l}|\red{-m}} \circ  \Omega_k$ could be used as different
(but equally natural) choices for the categorical comultiplication map.
Yet another possibility would be to define $\Delta_{\blue{l}|\red{m}}$
in the same way as in (\ref{com0})--(\ref{com1}) on the upward dot and crossing,
but to adopt the following on cups and caps
\begin{align*}
\mathord{
\begin{tikzpicture}[baseline = 1mm]
	\draw[<-] (0.4,0) to[out=90, in=0] (0.1,0.4);
	\draw[-] (0.1,0.4) to[out = 180, in = 90] (-0.2,0);
\end{tikzpicture}
}&\mapsto
\mathord{
\begin{tikzpicture}[baseline = 1mm]
	\draw[<-,bluep] (0.4,0) to[out=90, in=0] (0.1,0.4);
	\draw[-,bluep] (0.1,0.4) to[out = 180, in = 90] (-0.2,0);
\end{tikzpicture}
}-
\mathord{
\begin{tikzpicture}[baseline = 1mm]
	\draw[<-,redp] (0.4,-.2) to[out=90, in=0] (0.1,0.6);
	\draw[-,redp] (0.1,0.6) to[out = 180, in = 90] (-0.2,-.2);
     \node at (-0.2,.3) {$\clockrightblue$};
\end{tikzpicture}
}\:,&
\mathord{
\begin{tikzpicture}[baseline = 1mm]
	\draw[<-] (0.4,0.4) to[out=-90, in=0] (0.1,0);
	\draw[-] (0.1,0) to[out = 180, in = -90] (-0.2,0.4);
\end{tikzpicture}
}
&\mapsto
\mathord{
\begin{tikzpicture}[baseline = 1mm]
	\draw[<-,bluep] (0.4,0.4) to[out=-90, in=0] (0.1,0);
	\draw[-,bluep] (0.1,0) to[out = 180, in = -90] (-0.2,0.4);
\end{tikzpicture}
}+\mathord{
\begin{tikzpicture}[baseline = 1mm]
	\draw[<-,redp] (0.4,0.6) to[out=-90, in=0] (0.1,-0.2);
	\draw[-,redp] (0.1,-0.2) to[out = 180, in = -90] (-0.2,0.6);
     \node at (0.4,.2) {$\anticlockleftblue$};
\end{tikzpicture}
}\:,\\
 \mathord{
\begin{tikzpicture}[baseline = 1mm]
	\draw[-] (0.4,0) to[out=90, in=0] (0.1,0.4);
	\draw[->] (0.1,0.4) to[out = 180, in = 90] (-0.2,0);
\end{tikzpicture}
}&\mapsto
\mathord{
\begin{tikzpicture}[baseline = 1mm]
	\draw[-,bluep] (0.4,-.2) to[out=90, in=0] (0.1,0.6);
	\draw[->,bluep] (0.1,0.6) to[out = 180, in = 90] (-0.2,-.2);
     \node at (0.4,.3) {$\anticlockleftred$};
\end{tikzpicture}
}+
\mathord{
\begin{tikzpicture}[baseline = 1mm]
	\draw[-,redp] (0.4,0) to[out=90, in=0] (0.1,0.4);
	\draw[->,redp] (0.1,0.4) to[out = 180, in = 90] (-0.2,0);
\end{tikzpicture}
}\:,
&\mathord{
\begin{tikzpicture}[baseline = 1mm]
	\draw[-] (0.4,0.4) to[out=-90, in=0] (0.1,0);
	\draw[->] (0.1,0) to[out = 180, in = -90] (-0.2,0.4);
\end{tikzpicture}
}&\mapsto
-\mathord{
\begin{tikzpicture}[baseline = 1mm]
	\draw[-,bluep] (0.4,0.6) to[out=-90, in=0] (0.1,-0.2);
	\draw[->,bluep] (0.1,-0.2) to[out = 180, in = -90] (-0.2,0.6);
     \node at (-0.2,.2) {$\clockrightred$};
\end{tikzpicture}
}+\mathord{
\begin{tikzpicture}[baseline = 1mm]
	\draw[-,redp] (0.4,0.4) to[out=-90, in=0] (0.1,0);
	\draw[->,redp] (0.1,0) to[out = 180, in = -90] (-0.2,0.4);
\end{tikzpicture}
}\:.
\end{align*}
This no longer has the property that $\eta \circ \Delta_{\blue{l}|\red{m}} =
\Delta_{\blue{m}|\red{l}}$, but instead
$\eta \circ \Delta_{\blue{l}|\red{m}} \circ * =
* \circ \Delta_{\blue{m}|\red{l}}$.
\end{remark}

The proof of Theorem~\ref{comult} will be explained at the end of the
section.
The main work is to verify that all of the
defining relations from Definition~\ref{maindef} are satisfied in
$\blue{\Heis_l} \;\overline{\odot}\; \red{\Heis_m}$. To prepare for
this, we first establish a series of
lemmas.

\begin{lemma}\label{l1}
We have that
$
\mathord{

}
\right)^{-1}.
$
\end{lemma}

\begin{proof}
We note first that
\begin{align*}
\mathord{
%
}
\:.
\end{align*}
Using this and the definition (\ref{odd2}) gives that
\begin{align*}
\mathord{
%
}\:.
\end{align*}
Noting that internal bubbles on the same string commute, this implies
the result.
 \end{proof}

\begin{lemma}\label{l2}
We have that
$\displaystyle\mathord{
}
=
\sum_{b \in \Z}
\begin{array}{l}
\mathord{
\begin{tikzpicture}[baseline = 1.25mm]
  \draw[<-,bluep] (0,0.4) to[out=180,in=90] (-.2,0.2);
  \draw[-,bluep] (0.2,0.2) to[out=90,in=0] (0,.4);
 \draw[-,bluep] (-.2,0.2) to[out=-90,in=180] (0,0);
  \draw[-,bluep] (0,0) to[out=0,in=-90] (0.2,0.2);
   \node at (-0.2,0.2) {$\color{blue}\dot$};
   \node at (-.38,0.2) {$\color{blue}\scriptstyle{b}$};
\end{tikzpicture}
}\\
\mathord{
\begin{tikzpicture}[baseline = 1.25mm]
  \draw[-,redp] (0.2,0.2) to[out=90,in=0] (0,.4);
  \draw[<-,redp] (0,0.4) to[out=180,in=90] (-.2,0.2);
\draw[-,redp] (-.2,0.2) to[out=-90,in=180] (0,0);
  \draw[-,redp] (0,0) to[out=0,in=-90] (0.2,0.2);
   \node at (0.2,0.2) {$\color{red}\dot$};
   \node at (.65,0.2) {$\color{red}\scriptstyle{a-b-1}$};
\end{tikzpicture}
}\end{array}
$
for any $a \geq 0$.
\end{lemma}

\begin{proof}
By the definitions (\ref{odd1})--(\ref{odd2}), the left hand side is
\begin{multline*}
\sum_{b \geq 0}
\mathord{
%
}
\:.
\end{multline*}
This simplifies to produce the single summation on the right hand side.
\end{proof}

\begin{lemma}\label{l3a}
We have that
$\displaystyle\:\mathord{
%
}
\:.
$
\end{lemma}

\begin{proof}
By the definition (\ref{odd2}), the left hand side equals
\begin{align*}
\sum_{a \geq 0}
\mathord{
%
}\:.\end{align*}
Using (\ref{odd2}) once again, this is equal to the right hand side.
\end{proof}

\begin{lemma}\label{l4}
We have that
$\:
\displaystyle
\mathord{
}
\begin{array}{l}
\mathord{
\begin{tikzpicture}[baseline = 1.25mm]
  \draw[<-,bluep] (0,0.4) to[out=180,in=90] (-.2,0.2);
  \draw[-,bluep] (0.2,0.2) to[out=90,in=0] (0,.4);
 \draw[-,bluep] (-.2,0.2) to[out=-90,in=180] (0,0);
  \draw[-,bluep] (0,0) to[out=0,in=-90] (0.2,0.2);
   \node at (-0.2,0.2) {$\color{blue}\dot$};
   \node at (-.38,0.2) {$\color{blue}\scriptstyle{b}$};
\end{tikzpicture}
}\\
\mathord{
\begin{tikzpicture}[baseline = 1.25mm]
  \draw[-,redp] (0.2,0.2) to[out=90,in=0] (0,.4);
  \draw[<-,redp] (0,0.4) to[out=180,in=90] (-.2,0.2);
\draw[-,redp] (-.2,0.2) to[out=-90,in=180] (0,0);
  \draw[-,redp] (0,0) to[out=0,in=-90] (0.2,0.2);
   \node at (0.2,0.2) {$\color{red}\dot$};
   \node at (.7,0.2) {$\color{red}\scriptstyle{-a-b-2}$};
\end{tikzpicture}
}\end{array}
$.
\end{lemma}

\begin{proof}
We apply Lemma~\ref{l3a} to commute the internal bubble in the term on
the left hand side past the crossing. The left hand side becomes
\begin{align*}
\mathord{
%
}\:,
\end{align*}
which equals the right hand side.
\end{proof}

\begin{lemma}\label{l6}
We have that $\: \mathord{
%
}
$.
\end{lemma}

\begin{proof}
By Lemma~\ref{l3a}, the left hand side equals
\begin{align*}
\mathord{
%
}\!\!\!\!\!\!\!\!.
\end{align*}
It remains to observe that the three summations at the end simplify
to the single summation on the right hand side of the formula we are
trying to prove.
\end{proof}

\begin{lemma}\label{l5}
We have that
$\:
\displaystyle\mathord{
%
}
$.
\end{lemma}

\begin{proof}
Applying (\ref{c1})--(\ref{c2}), the left hand side equals
\begin{align*}
\displaystyle\mathord{
%
}\:.
\end{align*}
Then we apply Lemma~\ref{l4} (and the relation obtained by applying
$\eta$ to its $180^\circ$ rotation) to rewrite this expression as
\begin{align*}
-\sum_{\substack{a\geq 0 \\ b \in \Z}}
\mathord{
%
}
\:.
\end{align*}
Finally an application of (\ref{teleporting}) gives the result.
\end{proof}

\begin{lemma}\label{l3b}
We have that
$\:\mathord{
%
}$.
\end{lemma}

\begin{proof}
By (\ref{odd2}), the left hand side is
\begin{align*}
\sum_{a \geq 0}
&\mathord{
%
}
\end{align*}
which is equal to the right hand side by (\ref{odd2}) once again.
\end{proof}

\begin{proof}[Proof of Theorem~\ref{comult}]
Once $\Delta_{\blue{l}|\red{m}}$ has been constructed, 
the part about $\eta$ is obvious. Also (\ref{magic}) for the sign $+$
follows from Theorem~\ref{upgraded}, noting that the formulae in
(\ref{com0old}) are the same as here; then for the sign
$-$ it follows by taking right duals (using the rightward cups and caps).

In the remainder of the proof, we are going
to use the presentation from Definition~\ref{maindef} to establish the
existence of $\Delta_{\blue{l}|\red{m}}$.
Thus we define $\Delta_{\blue{l}|\red{m}}$ on the generating morphisms from that definition
by the formulae (\ref{com0})--(\ref{com4}), and
must check that
the images of the defining relations
(\ref{hecke})--(\ref{sideways})
all hold in
$\Add(\blue{\Heis_l}\;\overline{\odot}\;\red{\Heis_m})$.
Moreover we will show that $\Delta_{\blue{l}|\red{m}}$ satisfies (\ref{com5}).
In view of Lemma~\ref{lateaddition}, this is enough to prove the theorem as stated.
We already checked the relations (\ref{hecke}) in the proof of Theorem~\ref{comultah}.
Also the check of the relation (\ref{rightadj}) is quite trivial since all
of the matrices involved are diagonal.

Next we check (\ref{com5}).
Assume that $k \geq 0$ and
consider the clockwise bubble
$\mathord{
\begin{tikzpicture}[baseline = 1.25mm]
  \draw[<-] (0,0.4) to[out=180,in=90] (-.2,0.2);
  \draw[-] (0.2,0.2) to[out=90,in=0] (0,.4);
 \draw[-] (-.2,0.2) to[out=-90,in=180] (0,0);
  \draw[-] (0,0) to[out=0,in=-90] (0.2,0.2);
   \node at (-0.2,0.2) {$\dot$};
   \node at (-0.38,0.2) {$\scriptstyle{a}$};
\end{tikzpicture}
}$.
If it is a fake bubble, i.e., $a < 0$, it is a
scalar (usually zero) by the definition (\ref{f1})
and the assumption on $k$. Hence, it is
quite trivial to see that
(\ref{com5}) is satisfied.
When $a \geq 0$, the image of
$\mathord{
\begin{tikzpicture}[baseline = 1.25mm]
  \draw[<-] (0,0.4) to[out=180,in=90] (-.2,0.2);
  \draw[-] (0.2,0.2) to[out=90,in=0] (0,.4);
 \draw[-] (-.2,0.2) to[out=-90,in=180] (0,0);
  \draw[-] (0,0) to[out=0,in=-90] (0.2,0.2);
   \node at (-0.2,0.2) {$\dot$};
   \node at (-0.38,0.2) {$\scriptstyle{a}$};
\end{tikzpicture}
}$ under $\Delta_{\blue{l}|\red{m}}$ is
$
-\displaystyle\mathord{
\begin{tikzpicture}[baseline=-1mm]
\draw[-,bluep] (0,-0.2) to[out=180,in=-90] (-.2,0);
\draw[-,bluep] (-0.2,0) to[out=90,in=180] (0,0.2);
\draw[->,bluep] (0,0.2) to[out=0,in=90] (0.2,0);
\draw[-,bluep] (0.2,0) to[out=-90,in=0] (0,-0.2);
     \node at (-.17,-.13) {$\smallclockred$};
     \node at (-.17,.12) {$\color{blue}\dot$};
     \node at (-.34,.13) {$\color{blue}\scriptstyle{a}$};
\end{tikzpicture}}
\:-\displaystyle\mathord{
\begin{tikzpicture}[baseline=-1mm]
\draw[-,redp] (0,-0.2) to[out=180,in=-90] (-.2,0);
\draw[-,redp] (-0.2,0) to[out=90,in=180] (0,0.2);
\draw[->,redp] (0,0.2) to[out=0,in=90] (0.2,0);
\draw[-,redp] (0.2,0) to[out=-90,in=0] (0,-0.2);
     \node at (-.17,-.13) {$\smallclockblue$};
     \node at (-.17,.12) {$\color{red}\dot$};
     \node at (-.34,.13) {$\color{red}\scriptstyle{a}$};
\end{tikzpicture}}
\:$,
which is indeed
equal to
$-\sum_{b \in \Z}
\mathord{
\begin{tikzpicture}[baseline = 1.25mm]
  \draw[<-,bluep] (0,0.4) to[out=180,in=90] (-.2,0.2);
  \draw[-,bluep] (0.2,0.2) to[out=90,in=0] (0,.4);
 \draw[-,bluep] (-.2,0.2) to[out=-90,in=180] (0,0);
  \draw[-,bluep] (0,0) to[out=0,in=-90] (0.2,0.2);
   \node at (-0.2,0.2) {$\color{blue}\dot$};
   \node at (-.38,0.2) {$\color{blue}\scriptstyle{b}$};
\end{tikzpicture}
}\:\:\:
\mathord{
\begin{tikzpicture}[baseline = 1.25mm]
  \draw[-,redp] (0.2,0.2) to[out=90,in=0] (0,.4);
  \draw[<-,redp] (0,0.4) to[out=180,in=90] (-.2,0.2);
\draw[-,redp] (-.2,0.2) to[out=-90,in=180] (0,0);
  \draw[-,redp] (0,0) to[out=0,in=-90] (0.2,0.2);
   \node at (0.2,0.2) {$\color{red}\dot$};
   \node at (.65,0.2) {$\color{red}\scriptstyle{a-b-1}$};
\end{tikzpicture}
}
$ by Lemma~\ref{l2}.
Now consider (\ref{com5}) for the counterclockwise bubble (still
assuming $k \geq 0$).
Define the generating functions
$\anticlockleftblue\,(w)$ and $\clockrightblue\,(w)$ (resp., $\anticlockleftred\,(w)$ and
$\clockrightred\,(w)$)
 in the same way as (\ref{igq})--(\ref{igp}) but using blue (resp., red) bubbles
in place of black ones.
We have proved already that \begin{equation}\label{rowing1}
\Delta_{\blue{l}|\red{m}}\left(\clockright\,(w)\right) = \clockrightblue\,(w)\;\clockrightred\,(w).
\end{equation}
Passing to the inverses of these formal power series and using
(\ref{igproper})
shows that\begin{equation}\label{rowing2}
\Delta_{\blue{l}|\red{m}}\left(\anticlockleft\,(w)\right)
= \anticlockleftblue\,(w)  \;\anticlockleftred\,(w).
\end{equation}
Equating coefficients yields the desired relation for the
counterclockwise bubble.
This completes the proof of (\ref{com5}) when $k \geq 0$.
A similar argument works when $k \leq 0$ too: one starts
off by considering the relation for the counterclockwise bubble,
using the infinite Grassmannian relation to deduce the one for
the clockwise bubble at the end. On the way, one needs to use the
relation
obtained by applying
the symmetry $\Omega_{\blue{l}|\red{m}}$ to Lemma~\ref{l2}.

The relation (\ref{bubbles}) follows easily from (\ref{com5}) using
the first two equalities in (\ref{ig}) for the
blue
and red bubbles.

Moving on to (\ref{curls}), we first consider the right curl, so $k
\geq 0$. Applying $\Delta_{\blue{l}|\red{m}}$ to the relation reveals that
we must show
$$
-
\mathord{
\begin{tikzpicture}[baseline = -1mm]
  \draw[->,bluep] (-0.45,-.4) to (-0.45,-.3)
        to[in=180,out=90] (-.2,.2) to[out=0,in=90] (0,0.05)
        to[out=-90,in=0] (-.2,-.2) to [out=180,in=-90] (-0.45,.3) to (-0.45,.4);
   \node at (0,0) {$\smallclockred$};
\end{tikzpicture}
}
+
\mathord{\begin{tikzpicture}[baseline=-1mm]
\draw[->,bluep] (-0.6,-0.4) to (-.6,0.4);
\draw[-,redp] (0,-0.2) to[out=180,in=-90] (-.2,0);
\draw[->,redp] (-0.2,0) to[out=90,in=180] (0,0.2);
\draw[-,redp] (0,0.2) to[out=0,in=90] (0.2,0);
\draw[-,redp] (0.2,0) to[out=-90,in=0] (0,-0.2);
     \node at (.2,0) {$\smallclockblue$};
	\draw[-] (-.2,0.01) to (-0.6,0.01);
     \node at (-0.2,0) {$\dot$};
     \node at (-0.6,0) {$\dot$};
\end{tikzpicture}}
-
\mathord{
\begin{tikzpicture}[baseline = -1mm]
  \draw[->,redp] (-0.45,-.4) to (-0.45,-.3)
        to[in=180,out=90] (-.2,.2) to[out=0,in=90] (0,0.05)
        to[out=-90,in=0] (-.2,-.2) to [out=180,in=-90] (-0.45,.3) to (-0.45,.4);
   \node at (0,0) {$\smallclockblue$};
\end{tikzpicture}
}
-
\mathord{\begin{tikzpicture}[baseline=-1mm]
\draw[->,redp] (-0.6,-0.4) to (-.6,0.4);
\draw[-,bluep] (0,-0.2) to[out=180,in=-90] (-.2,0);
\draw[->,bluep] (-0.2,0) to[out=90,in=180] (0,0.2);
\draw[-,bluep] (0,0.2) to[out=0,in=90] (0.2,0);
\draw[-,bluep] (0.2,0) to[out=-90,in=0] (0,-0.2);
     \node at (.2,0) {$\smallclockred$};
	\draw[-] (-.2,0.01) to (-0.6,0.01);
     \node at (-0.2,0) {$\dot$};
     \node at (-0.6,0) {$\dot$};
\end{tikzpicture}}
= \delta_{k,0}\:
\mathord{\begin{tikzpicture}[baseline=-1mm]
\draw[->,bluep] (0,-0.4) to (0,0.4);
\end{tikzpicture}}
\:+\: \delta_{k,0}\:
\mathord{\begin{tikzpicture}[baseline=-1mm]
\draw[->,redp] (0,-0.4) to (0,0.4);
\end{tikzpicture}}
\:.
$$
This follows from the identity in Lemma~\ref{l4} and its mirror image
under $\eta$. Note for this that the only non-zero term in the
summation on the right hand side of this
identity is the one with $a=0,b=-1$ due to the assumption that $k \geq
0$.
The argument to treat the case of the left curl is entirely similiar;
it depends ultimately on the identity obtained by applying the
symmetry $\Omega_{\blue{l}|\red{m}}$ to Lemma~\ref{l4} then rotating through $180^\circ$.

Finally, we must check (\ref{sideways}).
We just go through the argument for this for the first equation. The
proof of the second one is entirely similar; it depends ultimately on
three identities derived from Lemmas~\ref{l6}--\ref{l3b} by applying
$\Omega_{\blue{l}|\red{m}}$ then rotating.
By the definition (\ref{rotate}), the map $\Delta_{\blue{l}|\red{m}}$ sends
\begin{align*}
\mathord{
\begin{tikzpicture}[baseline = 0]
	\draw[<-] (0.28,-.3) to (-0.28,.4);
	\draw[->] (-0.28,-.3) to (0.28,.4);
\end{tikzpicture}
}
&\mapsto
\mathord{
\begin{tikzpicture}[baseline = 0]
	\draw[<-,bluep] (0.28,-.3) to (-0.28,.4);
	\draw[->,bluep] (-0.28,-.3) to (0.28,.4);
\end{tikzpicture}
}
+\mathord{
\begin{tikzpicture}[baseline = 0]
	\draw[<-,redp] (0.28,-.3) to (-0.28,.4);
	\draw[->,redp] (-0.28,-.3) to (0.28,.4);
\end{tikzpicture}
}
+\mathord{
\begin{tikzpicture}[baseline = 0]
	\draw[<-,bluep] (0.28,-.3) to (-0.28,.4);
	\draw[->,redp] (-0.28,-.3) to (0.28,.4);
\end{tikzpicture}
}+
\mathord{
\begin{tikzpicture}[baseline = 0]
	\draw[<-,redp] (0.28,-.3) to (-0.28,.4);
	\draw[->,bluep] (-0.28,-.3) to (0.28,.4);
\end{tikzpicture}
}
-\mathord{
\begin{tikzpicture}[baseline = 0]
	\draw[<-,bluep] (0.28,-.3) to (-0.28,.4);
	\draw[->,redp] (-0.28,-.3) to (0.28,.4);
	\draw[-] (-0.15,-.14) to (0.15,-.14);
     \node at (0.15,-.15) {$\dot$};
     \node at (-0.15,-.15) {$\dot$};
\end{tikzpicture}
}
+\mathord{
\begin{tikzpicture}[baseline = 0]
	\draw[<-,redp] (0.28,-.3) to (-0.28,.4);
	\draw[->,bluep] (-0.28,-.3) to (0.28,.4);
	\draw[-] (-0.15,-.14) to (0.15,-.14);
     \node at (0.15,-.15) {$\dot$};
     \node at (-0.15,-.15) {$\dot$};
\end{tikzpicture}
}
-
\mathord{
\begin{tikzpicture}[baseline = 0mm]
	\draw[<-,bluep] (0.4,0.4) to[out=-90, in=0] (0.1,-.1);
	\draw[-,bluep] (0.1,-.1) to[out = 180, in = -90] (-0.2,0.4);
	\draw[<-,redp] (1.4,-.3) to[out=90, in=0] (1.1,0.2);
	\draw[-,redp] (1.1,0.2) to[out = 180, in = 90] (0.8,-.3);
	\draw[-] (.34,0.06) to (0.88,0.06);
     \node at (0.34,0.05) {$\dot$};
     \node at (0.88,0.05) {$\dot$};
\end{tikzpicture}
}+
\mathord{
\begin{tikzpicture}[baseline = 0mm]
	\draw[<-,redp] (0.4,0.4) to[out=-90, in=0] (0.1,-.1);
	\draw[-,redp] (0.1,-.1) to[out = 180, in = -90] (-0.2,0.4);
	\draw[<-,bluep] (1.4,-.3) to[out=90, in=0] (1.1,0.2);
	\draw[-,bluep] (1.1,0.2) to[out = 180, in = 90] (0.8,-.3);
	\draw[-] (.34,0.06) to (0.88,0.06);
     \node at (0.34,0.05) {$\dot$};
     \node at (0.88,0.05) {$\dot$};
\end{tikzpicture}
}\:,\\
\mathord{
\begin{tikzpicture}[baseline = 0]
	\draw[->] (0.28,-.3) to (-0.28,.4);
	\draw[<-] (-0.28,-.3) to (0.28,.4);
\end{tikzpicture}
}&\mapsto
-\mathord{
\begin{tikzpicture}[baseline = 0]
	\draw[->,bluep] (0.28,-.3) to (-0.28,.4);
	\draw[<-,bluep] (-0.28,-.3) to (0.28,.4);
     \node at (.13,.21) {$\smallclockred$};
\node at (-.14,-.09) {$\smallanticlockred$};
\end{tikzpicture}
}
-\mathord{
\begin{tikzpicture}[baseline = 0]
	\draw[->,redp] (0.28,-.3) to (-0.28,.4);
	\draw[<-,redp] (-0.28,-.3) to (0.28,.4);
     \node at (.13,.21) {$\smallclockblue$};
\node at (-.14,-.09) {$\smallanticlockblue$};
\end{tikzpicture}
}
-\mathord{
\begin{tikzpicture}[baseline = 0]
	\draw[->,bluep] (0.28,-.3) to (-0.28,.4);
	\draw[<-,redp] (-0.28,-.3) to (0.28,.4);
     \node at (.13,.21) {$\smallclockblue$};
\node at (-.14,-.09) {$\smallanticlockblue$};
\end{tikzpicture}
}-
\mathord{
\begin{tikzpicture}[baseline = 0]
	\draw[->,redp] (0.28,-.3) to (-0.28,.4);
	\draw[<-,bluep] (-0.28,-.3) to (0.28,.4);
     \node at (.13,.21) {$\smallclockred$};
\node at (-.14,-.09) {$\smallanticlockred$};
\end{tikzpicture}
}
+\mathord{
\begin{tikzpicture}[baseline = 0]
	\draw[->,bluep] (0.35,-.3) to (-0.35,.4);
	\draw[<-,redp] (-0.35,-.3) to (0.35,.4);
	\draw[-] (-0.25,-.21) to (0.25,-.21);
     \node at (0.25,-.22) {$\dot$};
     \node at (-0.25,-.22) {$\dot$};
     \node at (.16,.23) {$\smallclockblue$};
\node at (-.11,-.06) {$\smallanticlockblue$};
\end{tikzpicture}
}
-\mathord{
\begin{tikzpicture}[baseline = 0]
	\draw[->,redp] (0.35,-.3) to (-0.35,.4);
	\draw[<-,bluep] (-0.35,-.3) to (0.35,.4);
	\draw[-] (-0.25,-.21) to (0.25,-.21);
     \node at (0.25,-.22) {$\dot$};
     \node at (-0.25,-.22) {$\dot$};
     \node at (.16,.23) {$\smallclockred$};
\node at (-.11,-.06) {$\smallanticlockred$};
\end{tikzpicture}
}
+
\mathord{
\begin{tikzpicture}[baseline = 0mm]
	\draw[-,redp] (1.4,0.4) to[out=-90, in=0] (1.1,-.1);
	\draw[->,redp] (1.1,-.1) to[out = 180, in = -90] (0.8,0.4);
	\draw[-,bluep] (.4,-.3) to[out=90, in=0] (.1,0.2);
	\draw[->,bluep] (.1,0.2) to[out = 180, in = 90] (-0.2,-.3);
	\draw[-] (.34,0.06) to (0.88,0.06);
\node at (1.35,.15) {$\smallclockblue$};
\node at (-.15,-.04) {$\smallanticlockred$};
     \node at (0.34,0.05) {$\dot$};
     \node at (0.88,0.05) {$\dot$};
\end{tikzpicture}
}-
\mathord{
\begin{tikzpicture}[baseline = 0mm]
	\draw[-,bluep] (1.4,0.4) to[out=-90, in=0] (1.1,-.1);
	\draw[->,bluep] (1.1,-.1) to[out = 180, in = -90] (0.8,0.4);
	\draw[-,redp] (.4,-.3) to[out=90, in=0] (.1,0.2);
	\draw[->,redp] (.1,0.2) to[out = 180, in = 90] (-0.2,-.3);
	\draw[-] (.34,0.06) to (0.88,0.06);
     \node at (0.34,0.05) {$\dot$};
     \node at (0.88,0.05) {$\dot$};
\node at (1.35,.15) {$\smallclockred$};
\node at (-.15,-.04) {$\smallanticlockblue$};
\end{tikzpicture}
}\:.
\end{align*}
With this, it is straightforward to compute the image under $\Delta_{\blue{l}|\red{m}}$
of the left
hand side of (\ref{sideways}). To compute the image of the right hand
side, one also needs to use (\ref{com5}).
Then one looks at
the various matrix entries of the resulting equation to reduce
to checking the following three identities
\begin{align*}
-\mathord{
\begin{tikzpicture}[baseline=-.5mm]
	\draw[->,bluep] (0,0.6) to[out=-90, in=90] (.5,0) to
        [out=-90,in=90] (0,-.6);
	\draw[<-,bluep] (.5,0.6) to[out=-90, in=90] (0,0) to
        [out=-90,in=90] (0.5,-.6);
     \node at (.39,-.15) {$\smallclockred$};
\node at (0.11,-.39) {$\smallanticlockred$};
\end{tikzpicture}
}-
\mathord{
\begin{tikzpicture}[baseline=-.5mm]
	\draw[<-,bluep] (.3,0.6) to[out=-90, in=0] (0,.1);
	\draw[-,bluep] (0,.1) to[out = 180, in = -90] (-.3,0.6);
	\draw[-,bluep] (.3,-.6) to[out=90, in=0] (0,-0.1);
	\draw[->,bluep] (0,-0.1) to[out = 180, in = 90] (-0.3,-.6);
\draw[-,redp] (.9,-0.38) to[out=0,in=-90] (1.18,0);
\draw[-,redp] (1.18,0) to[out=90,in=0] (.9,0.38);
\draw[<-,redp] (.9,0.38) to[out=180,in=90] (.62,0);
\draw[-,redp] (.622,0.02) to[out=-86,in=180] (.9,-0.38);
     \node at (1.18,0) {$\smallclockblue$};
\draw[-] (.71,-.24) to (.22,-.24);
\draw[-] (.71,.26) to (.22,.26);
      \node at (.71,-0.25) {$\dot$};
   \node at (.71,0.25) {$\dot$};
      \node at (.22,-0.25) {$\dot$};
   \node at (0.22,0.25) {$\dot$};
\node at (-0.25,-.33) {$\smallanticlockred$};
\end{tikzpicture}}
&=
\:\mathord{
\begin{tikzpicture}[baseline = -1mm]
  \draw[->,bluep] (0,-.6) to (0,.6);
  \draw[<-,bluep] (-.5,-.6) to (-.5,.6);
\end{tikzpicture}}
\:-\;\displaystyle
\sum_{\substack{a,b \geq 0 \\ c \in \Z}}
\!\!\!\!\!\begin{array}{l}
\;\quad\mathord{
\begin{tikzpicture}[baseline = 1.25mm]
  \draw[<-,bluep] (0,0.4) to[out=180,in=90] (-.2,0.2);
  \draw[-,bluep] (0.2,0.2) to[out=90,in=0] (0,.4);
 \draw[-,bluep] (-.2,0.2) to[out=-90,in=180] (0,0);
  \draw[-,bluep] (0,0) to[out=0,in=-90] (0.2,0.2);
   \node at (-0.2,0.2) {$\color{blue}\dot$};
   \node at (-.38,0.2) {$\color{blue}\scriptstyle{c}$};
\end{tikzpicture}
}\\
\mathord{
\begin{tikzpicture}[baseline = 1.25mm]
  \draw[-,redp] (0.2,0.2) to[out=90,in=0] (0,.4);
  \draw[<-,redp] (0,0.4) to[out=180,in=90] (-.2,0.2);
\draw[-,redp] (-.2,0.2) to[out=-90,in=180] (0,0);
  \draw[-,redp] (0,0) to[out=0,in=-90] (0.2,0.2);
   \node at (-0.2,0.2) {$\color{red}\dot$};
   \node at (-.85,0.2) {$\color{red}\scriptstyle{-a-b-c-3}$};
\end{tikzpicture}
}\end{array}
\mathord{
\begin{tikzpicture}[baseline=-.5mm]
	\draw[<-,bluep] (0.3,0.6) to[out=-90, in=0] (0,.1);
	\draw[-,bluep] (0,.1) to[out = 180, in = -90] (-0.3,0.6);
      \node at (0.44,-0.3) {$\color{blue}\scriptstyle{b}$};
	\draw[-,bluep] (0.3,-.6) to[out=90, in=0] (0,-0.1);
	\draw[->,bluep] (0,-0.1) to[out = 180, in = 90] (-0.3,-.6);
      \node at (0.27,-0.3) {$\color{blue}\dot$};
   \node at (0.27,0.3) {$\color{blue}\dot$};
   \node at (.43,.3) {$\color{blue}\scriptstyle{a}$};
\node at (-0.25,-.33) {$\smallanticlockred$};
\end{tikzpicture}}
\:,\end{align*}\begin{align*}
\mathord{
\begin{tikzpicture}[baseline = -3.8mm]
	\draw[<-,bluep] (0.4,0.4) to[out=-90, in=0] (0.1,-.1);
	\draw[-,bluep] (0.1,-.1) to[out = 180, in = -90] (-0.2,0.4);
	\draw[<-,redp] (.8,-.7) to [out=90,in=-90] (1.4,-.1) to[out=90, in=0] (1.1,0.2);
	\draw[-,redp] (1.1,0.2) to[out = 180, in = 90] (0.8,-.1)
        to[out=-90,in=90] (1.4,-.7);
	\draw[-] (.34,0.06) to (0.85,0.06);
     \node at (0.34,0.05) {$\dot$};
     \node at (0.85,0.05) {$\dot$};
\node at (1.38,-.22) {$\smallclockblue$};
\node at (.9,-.62) {$\smallanticlockblue$};
\end{tikzpicture}
}-\mathord{
\begin{tikzpicture}[baseline = 0mm]
	\draw[-,bluep] (.8,.8) to [out=-90,in=90] (1.4,0.2) to[out=-90, in=0] (1.1,-.1);
	\draw[->,bluep] (1.1,-.1) to[out = 180, in = -90]
        (0.8,0.2) to [out=90,in=-90] (1.4,.8);
	\draw[-,redp] (.4,-.3) to[out=90, in=0] (.1,0.2);
	\draw[->,redp] (.1,0.2) to[out = 180, in = 90] (-0.2,-.3);
	\draw[-] (.34,0.06) to (0.85,0.06);
     \node at (0.34,0.05) {$\dot$};
     \node at (0.85,0.05) {$\dot$};
\node at (1.35,.15) {$\smallclockred$};
\node at (-.15,-.04) {$\smallanticlockblue$};
\end{tikzpicture}
}&=-\sum_{\substack{a,b \geq 0 \\ c \in \Z}}
\!\!\!\!\!\begin{array}{l}
\;\quad\mathord{
\begin{tikzpicture}[baseline = 1.25mm]
  \draw[<-,bluep] (0,0.4) to[out=180,in=90] (-.2,0.2);
  \draw[-,bluep] (0.2,0.2) to[out=90,in=0] (0,.4);
 \draw[-,bluep] (-.2,0.2) to[out=-90,in=180] (0,0);
  \draw[-,bluep] (0,0) to[out=0,in=-90] (0.2,0.2);
   \node at (-0.2,0.2) {$\color{blue}\dot$};
   \node at (-.38,0.2) {$\color{blue}\scriptstyle{c}$};
\end{tikzpicture}
}\\
\mathord{
\begin{tikzpicture}[baseline = 1.25mm]
  \draw[-,redp] (0.2,0.2) to[out=90,in=0] (0,.4);
  \draw[<-,redp] (0,0.4) to[out=180,in=90] (-.2,0.2);
\draw[-,redp] (-.2,0.2) to[out=-90,in=180] (0,0);
  \draw[-,redp] (0,0) to[out=0,in=-90] (0.2,0.2);
   \node at (-0.2,0.2) {$\color{red}\dot$};
   \node at (-.85,0.2) {$\color{red}\scriptstyle{-a-b-c-3}$};
\end{tikzpicture}
}\end{array}
\mathord{
\begin{tikzpicture}[baseline=-.5mm]
	\draw[<-,bluep] (0.3,0.6) to[out=-90, in=0] (0,.1);
	\draw[-,bluep] (0,.1) to[out = 180, in = -90] (-0.3,0.6);
      \node at (0.44,-0.3) {$\color{red}\scriptstyle{b}$};
	\draw[-,redp] (0.3,-.6) to[out=90, in=0] (0,-0.1);
	\draw[->,redp] (0,-0.1) to[out = 180, in = 90] (-0.3,-.6);
      \node at (0.27,-0.3) {$\color{red}\dot$};
   \node at (0.27,0.3) {$\color{blue}\dot$};
   \node at (.43,.3) {$\color{blue}\scriptstyle{a}$};
\node at (-.25,-.3) {$\smallanticlockblue$};
\end{tikzpicture}}\:,&
-\mathord{
\begin{tikzpicture}[baseline=-.5mm]
	\draw[->,bluep] (0,0.6) to[out=-90, in=90] (.5,0) to
        [out=-90,in=90] (0,-.6);
	\draw[<-,redp] (.5,0.6) to[out=-90, in=90] (0,0) to
        [out=-90,in=90] (0.5,-.6);
     \node at (.39,-.15) {$\smallclockred$};
\node at (0.11,-.39) {$\smallanticlockred$};
\end{tikzpicture}
}+
\mathord{
\begin{tikzpicture}[baseline=-.5mm]
	\draw[->,bluep] (0,0.6) to[out=-90, in=90] (.5,0) to
        [out=-90,in=90] (0,-.6);
	\draw[<-,redp] (.5,0.6) to[out=-90, in=90] (0,0) to
        [out=-90,in=90] (0.5,-.6);
     \node at (.39,-.15) {$\smallclockred$};
\node at (0.11,-.39) {$\smallanticlockred$};
\draw[-] (.44,.42) to (.07,.42);
      \node at (.44,0.41) {$\dot$};
   \node at (.07,0.41) {$\dot$};
\draw[-] (.44,.16) to (.07,.16);
      \node at (.44,0.15) {$\dot$};
   \node at (.07,0.15) {$\dot$};
\end{tikzpicture}
}&=
\:\mathord{
\begin{tikzpicture}[baseline = -1mm]
  \draw[->,redp] (0,-.6) to (0,.6);
  \draw[<-,bluep] (-.5,-.6) to (-.5,.6);
\end{tikzpicture}}
\:.
\end{align*}
plus their images under the symmetry $\eta$.
To prove the first two of these,
simplify them by
multiplying the bottom left string with a
clockwise internal bubble and using Lemma~\ref{l1}; the resulting
identities then follow from Lemmas~\ref{l6} and
\ref{l5}, respectively.
For the final one, use Lemma~\ref{l3b} to commute the clockwise
internal bubble in the first diagram past the crossing below it, then
use Lemma~\ref{l1} and a commuting relation.
\end{proof}

\section{A new proof of the basis theorem}\label{sbt}

By a {\em module category} over $\Heis_k$, we mean a $\k$-linear category $\mathcal V$ together with a $\k$-linear monoidal functor
$\Heis_k \rightarrow \mathcal{E}nd_\k(\mathcal V)$, where $\mathcal{E}nd_\k(\mathcal V)$ denotes the strict $\k$-linear 
monoidal category consisting of $\k$-linear endofunctors and natural transformations.
Suppose that $\mathcal{V}$ and $\mathcal{W}$ are two
$\k$-linear categories.
Let $\mathcal{V}
\boxtimes \mathcal{W}$ be the $\k$-linear category
with objects that are pairs $(X, Y)$ of objects $X \in \mathcal{V}$
and $Y \in \mathcal{W}$, and morphisms defined from
$$
\Hom_{\mathcal{V}\boxtimes\mathcal{W}}((X_1,Y_1), (X_2, Y_2))
:= \Hom_{\mathcal{V}}(X_1, X_2) \otimes_\k \Hom_{\mathcal{W}}(Y_1, Y_2).
$$
The rule for
composition of morphisms in $\mathcal{V}\boxtimes\mathcal{W}$ is
 $(e\otimes f) \circ (g \otimes h) :=
(e \circ g) \otimes (f \circ h).$
If $\mathcal V$ and $\mathcal W$ are module categories
over $\blue{\Heis_l}$ and $\red{\Heis_m}$, respectively, then $\mathcal V
\boxtimes \mathcal W$ is naturally a module category over the
symmetric product
$\blue{\Heis_l}\odot\red{\Heis_m}$.
If in addition the morphism
$\mathord{
\begin{tikzpicture}[baseline = -1mm]
 	\draw[->,redp] (0.18,-.2) to (0.18,.25);
	\draw[->,bluep] (-0.38,-.2) to (-0.38,.25);
	\draw[-,dotted] (-0.38,0.01) to (0.18,0.01);
     \node at (0.18,0) {$\dot$};
     \node at (-0.38,0.01) {$\dot$};
\end{tikzpicture}}=\;
\mathord{
\begin{tikzpicture}[baseline = -1mm]
 	\draw[->,redp] (0.1,-.2) to (0.1,.25);
	\draw[->,bluep] (-0.3,-.2) to (-0.3,.25);
     \node at (0.1,0) {$\red{\dot}$};
\end{tikzpicture}
}-
\mathord{
\begin{tikzpicture}[baseline = -1mm]
 	\draw[->,redp] (0.1,-.2) to (0.1,.25);
	\draw[->,bluep] (-0.3,-.2) to (-0.3,.25);
     \node at (-0.3,0) {$\blue{\dot}$};
\end{tikzpicture}}\:
$
acts invertibly on all objects of $\mathcal V
\boxtimes \mathcal W$, then this categorical action
extends to an action of the localization
$\blue{\Heis_l} \;\overline{\odot}\;\red{\Heis_m}$
from Section~\ref{sdhc}.
Hence, we can use the categorical comultiplication $\Delta_{\blue{l}|\red{m}}$ from Theorem~\ref{comult}
to make $\Add(\mathcal V\boxtimes \mathcal W)$ into a module category over
$\Heis_{k}$ where $k = \blue{l}+\red{m}$.
In this section, we are going to use this idea to give an efficient proof of the basis
theorem for the morphism spaces in $\Heis_k$ from \cite[Theorem 1.6]{B2}.

To get started, we need a source of Heisenberg module categories.
These come from degenerate cyclotomic Hecke algebras.
Assume that $f(w), g(w)$
are monic polynomials in $\k[w]$
of degrees $l,m \geq 0$, respectively.
The {\em degenerate cyclotomic Hecke algebra}
$\cH_n^f$ associated to the polynomial $f(w)$ is the quotient $\dH_n / (f(x_1))$; in case $n=0$ we have that $\cH_0^f = \dH_0 = \k$ by convention.
This algebra has the well-known basis
\begin{equation}\label{base1}
\left\{
x_1^{a_1} \cdots x_n^{a_n} \pi \:\big|\:\pi \in \S_n, 0 \leq a_1,\dots,a_n <
l\right\};
\end{equation}
see \cite[Section 7.5]{Kbook}.
In particular, one sees from this that the natural
homomorphism $\cH_n^f \rightarrow \cH_{n+1}^f$ is injective.
The following elementary lemma is well known; cf. \cite[Proposition
2.2.2]{Kbook}.
It implies that the eigenvalues of all $x_i$ on any $\cH_n^f$-module
lie in the same cosets of $\k$ modulo $\Z$ as the roots of the
polynomial $f(w)$.

\begin{lemma}\label{stupid}
Assume that $V$ is a finite-dimensional $\dH_2$-module.
All eigenvalues of $x_2$ on $V$ are of the form
$\lambda, \lambda+1$ or
$\lambda-1$ for eigenvalues $\lambda$ of $x_1$ on $V$.
\end{lemma}

\begin{proof}
For the proof, we may assume that the ground field $\k$ is algebraically closed.
Let
$v \in V$ be a simultaneous eigenvector for the commuting operators
$x_1$ and $x_2$ of
eigenvalues $\lambda_1$ and $\lambda_2$, respectively.
If $s_1 v = v$ (resp., $s_1 v = - v$)
then $\lambda_2 = \lambda_1+1$ (resp., $\lambda_2 = \lambda_1-1$), as
follows easily from the relation $x_2 v = (s_1 x_1+1) s_1 v$.
Otherwise, $v$ and $s_1 v$ are linearly independent, in which case
the matrix describing the action of
$x_1$ on the subspace with basis $\{v, s_1 v\}$ is
$\left(\begin{array}{rr}\lambda_1&-1\\0&\lambda_2\end{array}\right).$
So $\lambda_2$ is another eigenvalue of $x_1$ on $V$.
\end{proof}

To the polynomials $f(w)$ and $g(w)$, we are going to associate a
$\Heis_k$-module category $\mathcal{V}(f|g)$. As
a $\k$-linear category, this is defined 
from
\begin{equation*}
\mathcal{V}(f|g) := \Add(\mathcal{V}(f)\boxtimes \mathcal{V}(g)^\vee)
\end{equation*}
where
\begin{equation}\label{radio}
\mathcal{V}(f) := \bigoplus_{n \geq 0} \cH_n^f \pmd,
\qquad
\mathcal{V}(g)^\vee := \bigoplus_{n \geq 0} \cH_n^g \pmd.
\end{equation}
To make $\mathcal{V}(f|g)$ into a module category, we first make $\mathcal{V}(f)$ and $\mathcal{V}(g)^\vee$ into $\Heis_{-l}-$ and $\Heis_m$-module categories, respectively.
According to \cite[(1.23)]{B2}, there is a strict $\k$-linear monoidal functor
\begin{equation}\label{psif}
\Psi_f:\Heis_{-l} \rightarrow \mathcal{E}nd_\k\left(\mathcal{V}(f)\right)
\end{equation}
sending $\up$ (resp., $\down$) to the endofunctor defined on $M \in
\cH_n^f\pmd$ by the induction functor $\ind_n^{n+1} = \cH_{n+1}^f \otimes_{\cH_n^f} -$ (resp., the restriction functor $\res^n_{n-1}$).
On generating morphisms, $\Psi_f$ sends
\begin{itemize}
\item
$\mathord{
\begin{tikzpicture}[baseline = -.5mm]
	\draw[->] (0.08,-.2) to (0.08,.3);
      \node at (0.08,0.05) {$\dot$};
\end{tikzpicture}
}\:$
to the natural transformation defined on
a projective $\cH_n^f$-module $M$
by the map
$
\cH_{n+1}^f\otimes_{\cH_n^f} M \rightarrow
\cH_{n+1}^f\otimes_{\cH_n^f} M, h \otimes v \mapsto h x_{n+1}
\otimes v$;
\item
$\mathord{
\begin{tikzpicture}[baseline = -.5mm]
	\draw[->] (0.2,-.2) to (-0.2,.3);
	\draw[->] (-0.2,-.2) to (0.2,.3);
\end{tikzpicture}
}\:$
to the natural transformation defined on
a projective
$\cH_n^f$-module $M$ by
the map
$\cH_{n+2}^f \otimes_{\cH_n^f} M \rightarrow
\cH_{n+2}^f \otimes_{\cH_n^f} M,
h \otimes v \mapsto h s_{n+1}
\otimes v$;
\item
$\mathord{
\begin{tikzpicture}[baseline = .5mm]
	\draw[<-] (0.3,0.3) to[out=-90, in=0] (0.1,0);
	\draw[-] (0.1,0) to[out = 180, in = -90] (-0.1,0.3);
\end{tikzpicture}
}\:$ and
$\:\mathord{
\begin{tikzpicture}[baseline = .5mm]
	\draw[<-] (0.3,0) to[out=90, in=0] (0.1,0.3);
	\draw[-] (0.1,0.3) to[out = 180, in = 90] (-0.1,0);
\end{tikzpicture}
}\:$
to the natural transformations defined
by the
unit and counit of the canonical adjunction making $(\ind_n^{n+1},
\res_n^{n+1})$ into an adjoint pair.
\end{itemize}
Thus we have made $\mathcal{V}(f)$ into a module category over
$\Heis_{-l}$.
Similarly, switching the roles of induction and restriction using
$\jmath_n$ in place of $\imath_n$,
we make $\mathcal{V}(g)^\vee$ into a
$\Heis_m$-module category via the strict
$\k$-linear monoidal functor
\begin{equation}\label{psig}
\Psi_g^\vee:\Heis_{m} \rightarrow \mathcal{E}nd_\k\left(\mathcal{V}(g)^\vee\right)
\end{equation}
sending $\down$ (resp., $\up$) to the endofunctor defined on $M \in
\cH_n^g\pmd$ by the induction functor $\ind_n^{n+1} = \cH_{n+1}^g \otimes_{\cH_n^g} -$ (resp., the restriction functor $\res^n_{n-1}$).
On generating morphisms, $\Psi_g^\vee$ sends
\begin{itemize}
\item
$\mathord{
\begin{tikzpicture}[baseline = -.5mm]
	\draw[<-] (0.08,-.2) to (0.08,.3);
      \node at (0.08,0.05) {$\dot$};
\end{tikzpicture}
}\:$
to the natural transformation defined on
a projective $\cH_n^g$-module $M$
by the map
$
\cH_{n+1}^g\otimes_{\cH_n^g} M \rightarrow
\cH_{n+1}^f\otimes_{\cH_n^g} M, h \otimes v \mapsto h x_{n+1}
\otimes v$;
\item
$\mathord{
\begin{tikzpicture}[baseline = -.5mm]
	\draw[<-] (0.2,-.2) to (-0.2,.3);
	\draw[<-] (-0.2,-.2) to (0.2,.3);
\end{tikzpicture}
}\:$
to the natural transformation defined on
a projective
$\cH_n^g$-module $M$ by
the map
$\cH_{n+2}^g \otimes_{\cH_n^g} M \rightarrow
\cH_{n+2}^g \otimes_{\cH_n^g} M,
h \otimes v \mapsto -h s_{n+1}
\otimes v$;
\item
$\mathord{
\begin{tikzpicture}[baseline = .5mm]
	\draw[-] (0.3,0.3) to[out=-90, in=0] (0.1,0);
	\draw[->] (0.1,0) to[out = 180, in = -90] (-0.1,0.3);
\end{tikzpicture}
}\:$ and
$\:\mathord{
\begin{tikzpicture}[baseline = .5mm]
	\draw[-] (0.3,0) to[out=90, in=0] (0.1,0.3);
	\draw[->] (0.1,0.3) to[out = 180, in = 90] (-0.1,0);
\end{tikzpicture}
}\:$
to the natural transformations defined
by the
unit and counit of the canonical adjunction making $(\ind_n^{n+1},
\res_n^{n+1})$ into an adjoint pair.
\end{itemize}
The proof of this is similar to the argument explained in \cite[(1.23)]{B2},
using one of the alternative presentations for $\Heis_m$ from Remark~\ref{spicyramen}.

\begin{lemma}\label{starbucks}
Suppose that $f(w) = (w-\lambda_1)\cdots(w-\lambda_l)$
and $g(w) = (w-\mu_1)\cdots(w-\mu_m)$ for $\lambda_i, \mu_j \in \k$ such that $\lambda_i - \mu_j \notin \Z$ for all $i,j$.
In the categorical action of
$\blue{\Heis_{-l}} \odot \red{\Heis_m}$ on $\mathcal{V}(f) \boxtimes
\mathcal V(g)^\vee$ arising from (\ref{psif})--(\ref{psig}),
 $\mathord{
\begin{tikzpicture}[baseline = -1mm]
 	\draw[->,redp] (0.18,-.2) to (0.18,.25);
	\draw[->,bluep] (-0.38,-.2) to (-0.38,.25);
	\draw[-,dotted] (-0.38,0.01) to (0.18,0.01);
     \node at (0.18,0) {$\dot$};
     \node at (-0.38,0.01) {$\dot$};
\end{tikzpicture}}$
acts invertibly on every object.
\end{lemma}

\begin{proof}
Lemma~\ref{stupid} and the genericity assumption imply that
the set of 
eigenvalues of $x_1,\dots,x_n$ on any
finite-dimensional $H_n^f$-module is disjoint from the 
set of eigenvalues of $x_1,\dots,x_m$ on any finite-dimensional $H_m^g$-module.
Consequently, the commuting endomorphisms
defined by evaluating
$\:\mathord{
\begin{tikzpicture}[baseline = -1mm]
 	\draw[->,redp] (0.18,-.2) to (0.18,.25);
	\draw[->,bluep] (-0.18,-.2) to (-0.18,.25);
     \node at (0.18,0) {$\red{\dot}$};
\end{tikzpicture}}
$
and
$\mathord{
\begin{tikzpicture}[baseline = -1mm]
 	\draw[->,redp] (0.18,-.2) to (0.18,.25);
	\draw[->,bluep] (-0.18,-.2) to (-0.18,.25);
     \node at (-0.18,0) {$\blue{\dot}$};
\end{tikzpicture}}\:
$
on an object of
$\mathcal{V}(f) \boxtimes
\mathcal V(g)^\vee$
have disjoint spectra.
Hence, all eigenvalues of the endomorphism
defined by
$\mathord{
\begin{tikzpicture}[baseline = -1mm]
 	\draw[->,redp] (0.18,-.2) to (0.18,.25);
	\draw[->,bluep] (-0.38,-.2) to (-0.38,.25);
	\draw[-,dotted] (-0.38,0.01) to (0.18,0.01);
     \node at (0.18,0) {$\dot$};
     \node at (-0.38,0.01) {$\dot$};
\end{tikzpicture}}=
\: \mathord{
\begin{tikzpicture}[baseline = -1mm]
 	\draw[->,redp] (0.18,-.2) to (0.18,.25);
	\draw[->,bluep] (-0.18,-.2) to (-0.18,.25);
     \node at (0.18,0) {$\red{\dot}$};
\end{tikzpicture}}
-\mathord{
\begin{tikzpicture}[baseline = -1mm]
 	\draw[->,redp] (0.18,-.2) to (0.18,.25);
	\draw[->,bluep] (-0.18,-.2) to (-0.18,.25);
     \node at (-0.18,0.01) {$\blue{\dot}$};
\end{tikzpicture}}
\:$
lie in $\k^\times$.
Consequently, this endomorphism is invertible.
\end{proof}

As explained in the opening paragraph of the section, it follows
that
there is a strict $\k$-linear monoidal functor
$\blue{\Heis_{-l}}\;\overline{\odot}\;\red{\Heis_m}\rightarrow
\mathcal{E}nd_\k\left(\mathcal{V}(f)\boxtimes \mathcal{V}(g)^\vee)\right)$
for $f(w), g(w)$ satisfying the genericity assumption from Lemma~\ref{starbucks}.
Passing to the additive envelope and composing with the
categorical comultiplication $\Delta_{\blue{-l}|\red{m}}$, we obtain a strict
$\k$-linear monoidal functor
\begin{equation}\label{hollow}
\Psi_{f|g}:\Heis_{m-l}\rightarrow
\mathcal{E}nd_\k\left(\mathcal{V}(f|g)\right).
\end{equation}
Thus we have made $\mathcal{V}(f|g)$ into a module category
over $\Heis_{m-l}$.

\begin{lemma}\label{music}
In the categorical action of $\Heis_{m-l}$ on
$\mathcal{V}(f|g)$
just defined,
the generating functions
$\anticlockleft\,(w)$ and
$\clockright\,(w)$ from (\ref{igq})--(\ref{igp}) act on
$(H_0^f, H_0^g)$
by multiplication by
$g(w) / f(w) \in w^{m-l}\k\llbracket w^{-1}\rrbracket$ and
$f(w)/g(w)\in w^{l-m}\k\llbracket w^{-1}\rrbracket$, respectively.
\end{lemma}

\begin{proof}
Applying \cite[Lemma 1.8]{B2} with $f'(w) = 1$,
we get that
\begin{align*}\Psi_f\left(\anticlockleftblue\,(w)\right)_{H_0^f}
  &=f(w)^{-1},
&\Psi_f\left(\clockrightblue\,(w)\right)_{H_0^f}&=f(w).
\end{align*}
Similarly, applying it with $f(w) = 1$,
we get that
\begin{align*}
\Psi_g^\vee\left(\anticlockleftred\,(w)\right)_{H_0^g} &=g(w),
&
\Psi_g^\vee\left(\clockrightred\,(w)\right)_{H_0^g}&=g(w)^{-1}.
\end{align*}
Now use the identities (\ref{rowing1})--(\ref{rowing2}).
\end{proof}

Now we can prove the basis theorem.
To recall its statement,
let $X = X_r\otimes \cdots\otimes X_1$ and $Y = Y_s \otimes\cdots
\otimes Y_1$ be objects of
$\Heis_k$ for $X_i, Y_j \in \{\up,\down\}$.
An {\em $(X,Y)$-matching} is a bijection between the
sets $\{i\:|\:X_i = \up\}\sqcup\{j\:|\:Y_j = \down\}$
and $\{i\:|\:X_i = \down\}\sqcup\{j\:|\:Y_j = \up\}$.
A {\em reduced lift} of an $(X,Y)$-matching means a diagram
representing a morphism $X \rightarrow Y$ such that
\begin{itemize}
\item
the endpoints of each string are points which
correspond under the given matching;
\item
there are no floating bubbles and no dots on any string;
\item there are no self-intersections of strings and
no two strings cross each other more than once.
\end{itemize}
Fix a set $B(X,Y)$ consisting of a choice of reduced lift for each of the
$(X,Y)$-matchings.
Let $B_{\circ}(X,Y)$ be the set of all morphisms that can be obtained
from the elements of $B(X,Y)$ by adding dots labelled with
non-negative integer multiplicities
near to the terminus of each string.
Recall the homomorphism $\beta$ from (\ref{beta}).

\begin{theorem} \label{basis}
For
$X, Y \in \Heis_k$,
the morphism space $\Hom_{\Heis_k}(X,Y)$ is a free right
$\Sym$-module with basis $B_{\circ}(X,Y)$, where the right
$\Sym$-module structure is defined by $\phi \theta := \phi \otimes
\beta(\theta)$.
\end{theorem}

\begin{proof}
We just prove this when $k \leq 0$; the result for $k \geq 0$ then follows
by applying $\Omega_k$.
Let $X = X_r\otimes \cdots \otimes X_1$ and $Y = Y_s\otimes
\cdots\otimes Y_1$ be two objects.

We first observe that $B_{\circ}(X,Y)$ spans
$\Hom_{\Heis_k}(X,Y)$ as a right $\Sym$-module.
This is because there is a ``straightening rule'' allowing
any diagram representing a morphism $X \rightarrow Y$ as a linear
combination of the ones in $B_{\circ}(X,Y)$. This proceeds by
induction on the number of crossings.
Dots can be moved past crossings modulo a correction term with fewer
crossings, so we can assume that all dots are at the termini of their
strings.
Also we can use the relations
(\ref{hecke}), (\ref{sideways}), (\ref{dog1}) and (\ref{altbraid}) to move strings into the
same configuration as one of the chosen reduced lifts. Again this may produce
correction terms with fewer crossings plus some floating
bubbles. Finally floating bubbles can be moved to the right hand
edge using (\ref{bs1}), where they become scalars in $\Sym$.

It remains to prove the linear independence.
The main step is to do this in
the special case that $X = Y =\: \up^{\otimes n}$.
Take a linear relation $\sum_{i =1}^N \phi_i \otimes \beta(\theta_i)=0$
for $\phi_i \in B_\circ(X,Y)$ and $\theta_i \in \Sym$.
Choose $l \geq m \gg 0$ so that
\begin{itemize}
\item
$k=m-l$;
\item
the multiplicities of dots in all $\phi_i$ arising in this linear relation
are $< l$;
\item
 all of the symmetric functions $\theta_i \in
\Sym$
are polynomials in
the elementary symmetric functions
$\e_1,\dots,\e_m$.
\end{itemize}
Let $u_1,\dots,u_m$ be indeterminates and let
$\mathbb{K}$ be the algebraic closure of
$\k(u_1,\dots,u_m)$.
We are going to work now with algebras/categories that are linear over
$\mathbb{K}$
(instead of the usual $\k$), adding a subscript $\mathbb{K}$ to our notation
as we do
to avoid any confusion.
Consider the
cyclotomic Hecke algebras ${_\mathbb{K}}H_n^f$ and ${_\mathbb{K}}H_n^g$ over $\mathbb{K}$
associated to the polynomials
\begin{align*}
f(w) &:= w^l,
&
g(w) = w^m+u_1 w^{m-1}+\cdots+u_m.
\end{align*}
Using the functor ${_\mathbb{K}}\Psi_{f|g}$ from (\ref{hollow}),
we make
${_\mathbb{K}}\mathcal{V}(f|g)$ into a
${_\mathbb{K}}\Heis_k$-module category.
Since $\k\hookrightarrow \mathbb{K}$, there is a canonical $\k$-linear
monoidal
functor $\Heis_{k} \rightarrow {_\mathbb{K}}\Heis_{k}$, allowing us to
view
${_\mathbb{K}}\mathcal{V}(f|g)$ also as a module
category over $\Heis_{k}$.
Now we evaluate the relation
$\sum_{i=1}^N \phi_i \otimes \beta(\theta_i)=0$ on
$({_\mathbb{K}}H^f_0,{_\mathbb{K}}H^g_0)\in {_\mathbb{K}}\mathcal{V}(f|g)$
to obtain a relation in
${_\mathbb{K}} H_n^f$. By the basis theorem for ${_\mathbb{K}}H_n^f$ from
(\ref{base1}) and the choice of $l$,
the images of $\phi_1,\dots,\phi_N$ 
in ${_\mathbb{K}}H_n^f$ are linearly independent over $\mathbb{K}$, so we deduce
that
the image of $\beta(\theta_i)$ in $\mathbb{K}$ is zero for each $i$.
To deduce from this that $\theta_i = 0$, we know by the choice of $m$
that $\theta_i$ is a polynomial in
$\e_1,\dots,\e_m$. So we need to show that the images of $\beta(\e_1),
\dots, \beta(\e_m)$ in $\mathbb{K}$
are algebraically independent.
In fact, these images are the
indeterminates
$u_1,\dots,u_m$, respectively, as follows from Lemma~\ref{music}
on noting that
$g(w)/f(w) = w^k + u_1 w^{k-1} + \cdots+u_m w^{k-m}$.

We have now proved the linear independence when $X = Y = \up^{\otimes
  n}$.
The general case reduces to this special case in just the same way as
indicated in the proof of \cite[Proposition 5]{K}.
Let us give some more details.
First, we can use the canonical isomorphism
$\Hom_{\Heis_k}(X,Y) \cong \Hom_{\Heis_k}(\unit, X^* \otimes Y)$
arising from rigidity to reduce the proof of linear independence to
the
case that $X = \unit$. Assume this from
now on.
The set $B_\circ(\unit, Y)$ is empty
unless $Y$ has the same number $n$ of $\up$'s as $\down$'s. Also we have already proved the linear
independence in the case $Y = \down^{\otimes n} \otimes
\up^{\otimes n}$. So we may assume that $Y$ has
a subword $\up \otimes \down$.
Let $Z$ be $Y$ with the two letters in the subword
interchanged.
By induction, we may assume the linear independence has already been
established for $B_\circ(\unit,Z)$.
Now take a linear relation $\sum_{i=1}^N \phi_i \otimes \beta(\theta_i)$ for $\phi_i \in
B_\circ(\unit, Y)$ and $\theta_i \in \Sym$.
Recalling the isomorphism
$\up\otimes\down\oplus\: \unit^{\oplus (-k)}
\stackrel{\sim}{\rightarrow}\:
\down\otimes \up$ from
(\ref{invrel}), multiplying the subword $\up\otimes\down$ on top by
the sideways crossing $\mathord{
\begin{tikzpicture}[baseline = -.5mm]
	\draw[<-] (0.2,-.2) to (-0.2,.3);
	\draw[->] (-0.2,-.2) to (0.2,.3);
\end{tikzpicture}
}$
defines a $\Sym$-linear map
$$
s:\Hom_{\Heis_k}(\unit, Y) \hookrightarrow \Hom_{\Heis_k}(\unit, Z).
$$
Unfortunately, $s$ does not send $B_\circ(\unit, Y)$ into
$B_\circ(\unit, Z)$, so we need to argue a little further.
For $\phi \in B_\circ(\unit,Y)$,
there are three possibilities:

\begin{enumerate}[wide=0pt, widest=99,leftmargin=*, labelsep=0pt]
\item 
If $\phi$ has a leftward cup labelled with $a$ dots joining the letters in the
subword
then $s(\phi)$ has a dotted curl in this position, which can be rewritten using
the relation
$$
\mathord{
\begin{tikzpicture}[baseline = 0]
	\draw[-] (-0.25,.6) to[out=300,in=90] (0.25,-0);
	\draw[-] (0.25,-0) to[out=-90, in=0] (0,-0.25);
	\draw[-] (0,-0.25) to[out = 180, in = -90] (-0.25,-0);
	\draw[<-] (0.25,.6) to[out=240,in=90] (-0.25,-0);
	    \node at (-0.45,0) {$\scriptstyle{a}$};
      \node at (-0.25,0) {$\dot$};
\end{tikzpicture}
}
=
\mathord{
\begin{tikzpicture}[baseline = 1mm]
	\draw[<-] (0.4,0.4) to[out=-90, in=0] (0.1,0);
	\draw[-] (0.1,0) to[out = 180, in = -90] (-0.2,0.4);
      \node at (0.65,0.18) {$\scriptstyle{a-k}$};
      \node at (0.37,0.18) {$\dot$};
\end{tikzpicture}
}
- \sum_{b=0}^{a-k-1}
\mathord{
\begin{tikzpicture}[baseline = 1mm]
	\draw[<-] (0.4,0.4) to[out=-90, in=0] (0.1,0);
	\draw[-] (0.1,0) to[out = 180, in = -90] (-0.2,0.4);
      \node at (0.53,0.18) {$\scriptstyle{b}$};
      \node at (0.37,0.18) {$\dot$};
\end{tikzpicture}
}
\mathord{\begin{tikzpicture}[baseline = .8mm]
  \draw[<-] (0.2,0.2) to[out=90,in=0] (0,.4);
  \draw[-] (0,0.4) to[out=180,in=90] (-.2,0.2);
\draw[-] (-.2,0.2) to[out=-90,in=180] (0,0);
  \draw[-] (0,0) to[out=0,in=-90] (0.2,0.2);
      \node at (-.2,0.2) {$\dot$};
      \node at (-0.65,0.2) {$\scriptstyle{a-b-1}$};
 \end{tikzpicture}
}
$$
from (\ref{dog1}).
Thus $s(\phi) = \phi^\dagger + (*)$ where $\phi^\dagger$ is $\phi$
with the leftward cup labelled by $a$ dots replaced with a rightward cup labelled by
$a-k$ dots, and $(*)$
is a linear combination of similar-looking diagrams but with strictly
fewer dots on the rightward cup
and a clockwise bubble. This bubble
may be moved to the right hand edge using \eqref{bs1}, where it becomes a scalar
in $\Sym$; this process produces extra diagrams which have
additional dots on the strings along
the way.
We may assume further that $B(\unit,Z)$ was chosen so that
$\phi^\dagger \in B_\circ(\unit,Z)$.
Let $B_1 = \bigcup_{b \geq 0} B_{1,b}$ where
$B_{1,b}$ is
the set of all $\psi \in B_\circ(\unit,Z)$
which have a rightward cup labelled by a dot of multiplicity $b$
joining the letters in the subword.
Then we have shown that $s(\phi)  = \phi^\dagger + (**)$
for $\phi^\dagger \in B_{1,a-k}$ and $(**)$ that is a linear
combination of terms in $B_{1,b}$ for $0 \leq b < a-k$.

\item 
If $\phi$ has two non-intersecting strings at the letters $\down$ and $\up$ of
the subword, we can slide any dots on the $\up$-string
of $s(\phi)$
to the terminus
to obtain $\phi^\dagger + (*)$
where $\phi^\dagger$ is a diagram that has intersecting strings at the letters of the
subword, and
$(*)$ is a linear combination of diagrams which have a
dotted rightward cup at the subword.
Again, we may assume that $\phi^\dagger \in B_\circ(\unit,Z)$ by the
choice of $B(\unit,Z)$.
Let $B_2$ be all elements of $B_\circ(\unit, Z)$ with
intersecting strings at the subword.
Rewriting the error terms $(*)$ in terms of the basis,  we deduce that
$s(\phi) = \phi^\dagger + (**)$
for $\phi^\dagger \in B_2$ and $(**)$ that is a linear combination of
terms in $B_{1}$.

\item 
If $\phi$ has two intersecting strings at the letters $\down$ and $\up$,
then $s(\phi)$ will have two strings that cross each other twice.
Again, we slide dots to the terminus, producing also an error term $(*)$
which is
a linear combination of terms in
$B_{1}$. Then we use (\ref{sideways}) (and possibly
some other braid relations if there are other strings in
between)
to eliminate the crossings of the two
strings in the leading term.
Making a suitable choice of $B(\unit,Z)$ and letting $B_3$ be the set
of all elements of $B_\circ(\unit,Z)$ with non-intersecting strings
at the subword,
we thus have that
$s(\phi)  = \phi^\dagger + (**)$ for $\phi^\dagger \in B_3$ and
and $(**)$ that
is a linear combination of terms in $B_1 \cup B_{2}$.
\end{enumerate}

\vspace{1mm}
We have that $\sum s(\phi_i) \otimes \beta(\theta_i) = 0$.
Ordering $B_\circ(\unit,Z)$ so that
$B_{1,0} < B_{1,1} < B_{1,2} < \cdots < B_2 < B_3$, we have shown that
$s(\phi_i)  = \phi_i^\dagger + (*)$
for
$\phi_i^\dagger \in B_1\cup B_2 \cup B_3$ and $(*)$ that is a linear
combination of smaller $g \in B_1 \cup B_2 \cup B_3$.
Also the elements $\phi_1^\dagger,\dots,\phi_N^\dagger$ are all
different.
Hence, the known linear independence of $B_\circ(\unit,Z)$ implies that $\theta_i = 0$
for all $i$, as required to complete the argument.
\end{proof}

\begin{corollary}\label{cor:beta-iso}
The homomorphism
$\beta:\Sym\rightarrow\End_{\Heis_k}(\unit)$
is an isomorphism.
\end{corollary}

\begin{remark}
As noted before Definition~\ref{maindef}, the category $\Heis_k$ can
be defined more generally over any commutative ground ring
$\k$. 
Theorem~\ref{basis} is easily extended to this situation: the proof
of the spanning part of the result works for any $\k$;
the linear independence in general may be deduced from the known linear
independence 
over $\Q$
by standard base change arguments.
\end{remark}

\section{Proofs of Theorems~\ref{t1} and \ref{t2}}\label{sp}

Recall the objects $S_\lambda^{\pm} \in
\Kar(\Heis_k)$ for each $\lambda \in \mathcal P$
defined by (\ref{goodtimes}). We note that 
\begin{equation}\label{transp}
\Omega_k(S_\lambda^{\pm}) \cong S_{\lambda^T}^{\mp},
\end{equation}
with the transpose partition appearing because of the sign 
when $\Omega_k$ is applied to a crossing.
The following provides the final
important ingredient
needed to prove the main results. The argument
depends essentially on
Theorems~\ref{Grothsplit}, \ref{starting} and \ref{basis}.

\begin{theorem}\label{fatherbrown}
The Grothendieck group $K_0(\Kar(\Heis_k))$ is free as a $\Z$-module,
with basis given by the elements
$\big\{[S_\mu^- \otimes S_\lambda^+]\:\big|\:\lambda,\mu \in \mathcal
P\big\}$
if $k \geq 0$
or
$\big\{[S_\mu^+ \otimes S_\lambda^-]\:\big|\:\lambda,\mu \in \mathcal
P\big\}$ if $k \leq 0$.
Moreover, 
$[X] = 0 \Rightarrow X = 0$ for $X \in \Kar(\Heis_k)$.
\end{theorem}

\begin{proof}
It suffices to treat the case $k \geq 0$; then the case $k \leq 0$
follows using (\ref{transp}).
We make four elementary reductions
which were suggested in \cite[Section 5.1]{K}:

\begin{enumerate}[wide=0pt, widest=99,leftmargin=*, labelsep=0pt]
\item
Let $A$ be the locally unital algebra 
with distinguished idempotents $\{1_X\:|\:X \in
\mathbb{A}\}$  that
arises from the $\k$-linear
category 
$\Heis_k$ 
 as in (\ref{dictionary}). So $\mathbb{A}$, the object set of $\Heis_k$, is the set of {\em words} obtained as finite tensor products of the symbols $\up, \down$.
In view of the contravariant equivalence (\ref{yon}), it suffices to show
that $[P]=0\Rightarrow P = 0$ for all $P \in A\pmd$, and that
$K_0(A)$ has basis $\big\{[A e_{\mu|\lambda}]\:\big|\:\lambda,\mu \in
\mathcal P\big\}$ where
$e_{\mu|\lambda} := 
\jmath_{|\mu|}(e_{\mu})\otimes \imath_{|\lambda|}(e_{\lambda})$ for
$\imath_n$ and $\jmath_n$ as in (\ref{im})--(\ref{jm}).
Note these are the objects of $A\pmd$ which correspond to the objects
 $S_{\mu^T}^-\otimes S_\lambda^+ \in
\Kar(\Heis_k)$.

\item
For $d \in \Z$, let $\mathbb{A}_d$ be the set of all words $X \in \mathbb{A}$ such that the
number of letters $\up$ minus the number of letters $\down$ is equal
to $d$. 
Let 
$$
A^{(d)} := \bigoplus_{X, Y \in \mathbb{A}_d} 1_X A 1_Y.
$$
Noting that $1_X A 1_Y = 0$ for $X \in \mathbb{A}_d, Y \in
\mathbb{A}_e$ and $d \neq e$, we have that
$A = \bigoplus_{d \in \Z} A^{(d)}$, hence,
$K_0(A\pmd) = \bigoplus_{d \in \Z} K_0(A^{(d)}\pmd)$.
Therefore it is enough to show that 
$[P]=0\Rightarrow P = 0$ for all $P \in A^{(d)}\pmd$,
and
that $K_0(A^{(d)})$ has basis
$\left\{[A^{(d)} e_{\mu|\lambda}]\:\big|\:\lambda, \mu \in \mathcal P,
  |\lambda|-|\mu| = d\right\}$.

\item
Since $\up\otimes\down \:\cong \:
\down\otimes \up\:
\oplus\: \unit^{\oplus k}$, the left ideal $A^{(d)} 1_X$ 
for $X \in \mathbb{A}_d$
is isomorphic to a
direct sum of left ideals $A^{(d)} 1_Y$ for words $Y\in\mathbb{A}_d$ in which all letters
$\down$ appear to the left of the letters $\up$. 
Letting $\mathbb{A}_d^+$
denote the set of all such $Y$, 
this means that $A^{(d)}$ is Morita equivalent to the locally unital
algebra
$$
B^{(d)} := \bigoplus_{X, Y \in \mathbb{A}_d^+} 1_X A^{(d)} 1_Y.
$$
Hence, we just need to show that $[P]=0\Rightarrow P=0$ for all $P \in B^{(d)}\pmd$,
and that
$K_0(B^{(d)}\pmd)$ has basis
$\left\{[B^{(d)} e_{\mu|\lambda}]\:\big|\:\lambda, \mu \in \mathcal P,
  |\lambda|-|\mu| = d\right\}$.

\vspace{1mm}
\item
Next, we
let $1^{(d)}_n := \sum_X 1_X$ summing over all words $X \in \mathbb{A}_d^+$
of length $\leq(2n+|d|)$.
Then let 
$$
B^{(d)}_n := 1^{(d)}_n B^{(d)} 1^{(d)}_n.
$$
This defines a direct system of locally unital algebras
$0 = B^{(d)}_{-1} \subset B^{(d)}_0 \subset B^{(d)}_1 \subset \cdots$
whose union is $B^{(d)}$. Moreover, each $B^{(d)}_n$ is actually unital.
As any idempotent in $B^{(d)}$ belongs to $B^{(d)}_n$ for some sufficiently large $n$, we have that
$$
K_0(B^{(d)}\pmd) = \varinjlim K_0(B^{(d)}_n\pmd).
$$
Using this, we are reduced to checking for each $n$ that 
$B_n^{(d)}$
is
stably finite, and that
$K_0(B_n^{(d)}\pmd)$ 
has basis
$\left\{[B_n^{(d)} e_{\mu|\lambda}]\:\big|\:\lambda, \mu \in \mathcal P,
  |\lambda|-|\mu| = d, |\lambda|+|\mu| \leq 2n+|d|\right\}$.
\end{enumerate}

\vspace{1mm}
To complete the proof of the theorem, we establish the truth of
the statement just made by induction on $n=-1,0,1,\dots$.
For the induction step, take $n \geq 0$, 
set $R := B_n^{(d)}$ and $e := 1^{(d)}_{n-1}$.
Note that 
$eRe = B_{n-1}^{(d)}$. By induction, we know
already that $eRe$ is stably finite and that
$K_0(eRe)$ has basis
$\left\{[B_{n-1}^{(d)} e_{\mu|\lambda}]\:\big|\: \lambda, \mu \in \mathcal P,
  |\lambda|-|\mu| = d, |\lambda|+|\mu| < 2n+|d|\right\}.$
Let $n_1,n_2\geq 0$ be defined from $n_1-n_2=d$ and $n_1+n_2=2n+|d|$.
By Theorem~\ref{basis}, 
the quotient $S := R / ReR$ has basis given by the elements
$\pi(\phi \theta)$
for $\phi \in B_\circ(\down^{\otimes n_2}\otimes \up^{\otimes n_1},
\down^{\otimes n_2}\otimes \up^{\otimes n_1})$ involving
no cups or caps and $\theta$ running over a
basis for $\Sym$, where $\pi:R \twoheadrightarrow S$ is the quotient map.
It follows that there is an isomorphism 
$AH_{n_1} \otimes_\k AH_{n_2} \otimes_\k \Sym
\stackrel{\sim}{\rightarrow} S,\:
\phi_1 \otimes \phi_2 \otimes \theta
\mapsto
\jmath_{n_2}(\phi_2) \otimes \imath_{n_1}(\phi_1) \otimes \beta(\theta).$
Moreover,
$\sigma:S \rightarrow R,
\pi(\phi \theta) \mapsto \phi \theta+e$
 is a {\em unital} algebra homomorphism. Since we obviously have that
 $\pi \circ \sigma = \operatorname{id}_S$,
this puts us in a position to apply Theorem~\ref{Grothsplit}. We
deduce that the induction step follows from the assertions that
$AH_{n_1}\otimes_\k AH_{n_2}\otimes_\k \Sym$ is stably finite and
$K_0(AH_{n_1} \otimes_\k AH_{n_2} \otimes_\k \Sym)$ has basis
$\big\{[AH_{n_1} e_\lambda \otimes_\k AH_{n_2} e_\mu \otimes_\k \Sym]\:\big|\:
\lambda, \mu \in \mathcal P, |\lambda|=n_1, |\mu|=n_2\big\}$.
The first of these statements follows from Lemma~\ref{beerisgood} (together with the fact that $AH_n$ is finitely generated as a module over its center $\Sym_n$), and the
second from Theorem~\ref{starting}.
\end{proof}

To prove Theorem~\ref{t1}, we are going to categorify some representations
of $\rh_k$. The {\em basic representation} of $\rh_{-1}$ is the
ring $\SymZ$ of symmetric funtions viewed as a $\rh_{-1}$-module so that
for $f \in \SymZ$
the element $f^+$ 
acts
by left
multiplication by $f$, and $f^-$ acts by the adjoint operator with 
respect to the usual form  $\langle -,-\rangle$ on $\SymZ$, i.e., $\langle s_\lambda,
s_\mu\rangle := \delta_{\lambda,\mu}$.
In particular, the generators of $\rh_{-1}$ act on the basis of Schur
functions as follows:
\begin{itemize}
\item
$h_n^+
s_\lambda = \sum s_\mu$ summing over all partitions $\mu$ whose Young
diagram is obtained
from that of $\lambda$ by adding a box to the end of $n$ different
columns;
\item
$e_n^- s_\lambda =\sum s_\mu$ summing over partitions $\mu$
whose Young diagram is obtained from that of $\lambda$ by removing a
box from the end of $n$ different rows.
\end{itemize}
Let $\SymZ^\vee$ be the $\rh_1$-module obtained from $\SymZ$ using
$\omega_1 \colon \rh_1\stackrel{\sim}{\rightarrow} \rh_{-1}$; see \eqref{sedai}.
Thus, denoting $s_\lambda$ instead by $s_\lambda^\vee$ to avoid
confusion, 
the action of $\rh_1$ on $\SymZ^\vee$ satisfies
\begin{itemize}
\item
$h_n^+
s^\vee_\lambda = \sum s^\vee_\mu$ summing over all partitions $\mu$ whose Young
diagram is obtained
from that of $\lambda$ by removing a box from the end of $n$ different
columns;
\item
$e_n^- s^\vee_\lambda =\sum s^\vee_\mu$ summing over partitions $\mu$
whose Young diagram is obtained from that of $\lambda$ by adding a
box to the end of $n$ different rows.
\end{itemize}
More generally, for $l,m \geq 0$ and $k := m-l$,
the tensor product
$V(l|m)
:= \SymZ^{\otimes l} \otimes_\Z \left(\SymZ^\vee\right)^{\otimes m}$
is naturally a $\rh_k$-module.
It has a natural monomial basis 
indexed by $(l+m)$-tuples of partitions. 
The associated representation
\begin{equation}\label{sat}
\psi_{l|m}:\rh_k \rightarrow \End_\Z\left(V(l|m)\right)
\end{equation}
is faithful as soon as $l+m > 0$; the proof of faithfulness is
particularly easy when both $l > 0$ and $m > 0$ which is all that we
use below.

For monic $f(w) \in \k[w]$ of degree one, the inclusion $\k \S_n \hookrightarrow H_n^f$ is
actually an algebra isomorphism.
Thus, the $\Heis_{-1}$-module category $\mathcal V(f)$ from (\ref{psif})
is the semisimple Abelian 
category $\bigoplus_{n \geq 0} \k \S_n\pmd$, and there is an isomorphism
$\SymZ
\stackrel{\sim}{\rightarrow} K_0(\mathcal V(f)),
s_\lambda \mapsto [S(\lambda)]$ of $\Z$-modules.
Similar statements hold for the $\Heis_1$-module category
$\mathcal V(g)^\vee$ from
(\ref{psig})
when $g(w) \in \k[w]$ is of degree one.
More generally, for $u_1,\dots,u_l,v_1,\dots,v_m \in \k$, the
category
$$
\mathcal V(u_1,\dots,u_l|v_1,\dots,v_m) :=
\Add(\mathcal V(w-u_1) \boxtimes \cdots \boxtimes \mathcal V(w-u_l)
\boxtimes \mathcal V(w-v_1)^\vee \boxtimes \cdots \boxtimes \mathcal
V(w-v_m)^\vee)
$$
is a semisimple Abelian category, and there is a $\Z$-module
isomorphism
\begin{align}\label{obv}
V(l|m)&\stackrel{\sim}{\rightarrow} 
K_0\left(\mathcal V(u_1,\dots,u_l|v_1,\dots,v_m)\right),\\
s_{\lambda^{(1)}}\otimes\cdots\otimes s_{\lambda^{(l)}}\otimes
s_{\mu^{(1)}}^\vee\otimes\cdots\otimes s_{\mu^{(m)}}^\vee&\mapsto 
\left[\left(S(\lambda^{(1)}),\dots,S(\lambda^{(l)}),S(\mu^{(1)}),\dots,S(\mu^{(m)})\right)\right].\notag
\end{align}
This is
a module category over 
$\Heis_{-1}\odot\cdots\odot \Heis_{-1}\odot
\Heis_1\odot\cdots\odot \Heis_1$.
If we assume in addition that $u_1,\dots,u_l, v_1,\dots,v_m$ are generic in the
sense that their images in $\k / \Z$
are all different,
then we can argue as in Lemma~\ref{starbucks}
to see that the action extends to
the localization
$\Heis_{-1}\;\overline{\odot}\;\cdots\;\overline{\odot}\; \Heis_{-1}\;\overline{\odot}\;
\Heis_1\;\overline{\odot}\;\cdots\;\overline{\odot}\; \Heis_1$.
Using the iterated categorical comultiplication from Theorem~\ref{comult}
(and the
coassociativity noted in Remark~\ref{altsss}), it becomes
a module category over $\Heis_k$.
Thus, there is a strict $\k$-linear monoidal functor
\begin{equation}
\Psi_{l|m}:\Heis_k \rightarrow \mathcal{E}nd_\k\left(\mathcal
  V(u_1,\dots,u_l|v_1,\dots,v_m)\right).
\end{equation}
Since $\mathcal
  V(u_1,\dots,u_l|v_1,\dots,v_m)$ is Abelian, this extends to a functor
  from $\Kar(\Heis_k)$, which we denote by the same notation $\Psi_{l|m}$.
The following shows that this functor categorifies (\ref{sat}).

\begin{theorem}\label{atlast}
There is a ring isomorphism $
\gamma_k:\rh_k \rightarrow
K_0(\Kar(\Heis_k))$ sending
$s_\lambda^{\pm} \mapsto
[S_\lambda^{\pm}]$
for each $\lambda \in \mathcal P$. 
Moreover,
for generic $u_1,\dots,u_l,v_1,\dots, v_m$ with $k=m-l$, the diagram
\begin{equation}\label{boardgames}
\begin{diagram}
\node{\rh_k}\arrow{s,l,A,J}{\gamma_k}\arrow{e,t}{\psi_{l|m}}\node{\End_\Z(V(l|m))}\arrow{s,r,A,J}{c_k}\\
\node{K_0(\Kar(\Heis_k))}\arrow{e,b}{[\Psi_{l|m}(-)]}\node{\End_\Z(K_0(\mathcal V(u_1,\dots,u_l|v_1,\dots,v_m))}
\end{diagram}
\end{equation}
commutes, where $c_k$ is the ring isomorphism defined by conjugating with (\ref{obv}),
and the bottom map is the ring homomorphism  $[X]\mapsto[\Psi_{l|m}(X)]$.
\end{theorem}

\begin{proof}
Theorem~\ref{fatherbrown} shows there is a
$\Z$-module isomorphism
$\gamma_k:\rh_k \stackrel{\sim}{\rightarrow}
K_0(\Kar(\Heis_k))$
sending
$s_\mu^-s_\lambda^+\mapsto
[S_\mu^-\otimes S_\lambda^+]$ if $k \geq 0$
or
$s_\mu^+s_\lambda^-\mapsto
[S_\mu^+\otimes S_\lambda^-]$ if $k \leq 0$,
although 
we do not yet know that this is a ring homomorphism. 
Taking this as the definition of the left hand map, we are going to show
in the next paragraph
that the diagram (\ref{boardgames}) commutes for all
generic $u_i, v_j$.
This is all that is needed to complete the proof:
since the top and bottom maps in (\ref{boardgames}) are ring homomorphisms,
the right hand map is a ring isomorphism,
and moreover $\psi_{l|m}$ is injective for any sufficiently large $l$ and
$m$, the commutativity of the diagram implies that the left hand map $\gamma_k$ is
a ring homomorphism too.

To see that the diagram commutes, it suffices to check that it
commutes on each of the basis vectors via which $\gamma_k$ has been
defined.
This reduces easily to checking that
\begin{equation}\label{nearly}
c_k(s_\lambda^\pm v) = [S_\lambda^{\pm}] c_k(v)
\end{equation}
for each $\lambda \in \mathcal P$ and $v \in V(l|m)$.
The restrictions $\gamma_k^-$ and $\gamma_k^+$ of the map $\gamma_k$ to the subalgebras
$\rh_k^- = \SymZ \otimes 1$ and $\rh_k^+ = 1 \otimes \SymZ$, respectively,
are both ring homomorphisms.
This follows for $\gamma_k^+$ because
$\gamma_k^+(s_\lambda^+) = [\imath](\gamma(s_\lambda))$,
where $\gamma$ is the ring isomorphism from (\ref{bargamma}) and
$[\imath]$ is the ring homomorphism induced by
the monoidal functor $\imath:\SYM \rightarrow \Heis_k$ defined just
before (\ref{im}).
To see it for $\gamma_k^-$, use instead that $\gamma_k^-(s_\lambda^-) = [\jmath](\gamma(s_{\lambda^T}))
$ for the monoidal functor $\jmath:\SYM \rightarrow \Heis_k$ arising
from (\ref{jm}).
In view of this and the fact that $\rh_k^\pm$
is generated by $\{h_n^{\pm}\:|\:n \geq 1\}$,
we deduce that (\ref{nearly}) follows if we can establish just that
\begin{equation}\label{atlonglast}
c_k(h_n^\pm v) = [H_n^\pm] c_k(v).
\end{equation} 
By the definition of (\ref{psif}), the object $H_n^+\in \Kar(\Heis_{-1})$
acts on $S(\lambda) \in \k \S_m\pmd$
by 
$$
H_n^+ 
S(\lambda) = 
\ind_{\S_m\times \S_n}^{\S_{m+n}} S(\lambda)\boxtimes
\operatorname{triv}_n,
$$
which is the image of $h_n s_\lambda$ under the isomorphism
$\SymZ \stackrel{\sim}{\rightarrow}
K_0(\mathcal V(w-u_i))$. Thus $h_n^+$ and $[H_n^+]$ act in the same way
under this isomorphism.
Since $H_n^- = (H_n^+)^*$, we deduce from this that
$h_n^-$ and $[H_n^-]$ 
act in the same way too.
Similar statements hold for the action on
$\SymZ\cong[\mathcal V(w-v_j)^\vee]$
for each $j$.
Recalling (\ref{comultiplication}) and
(\ref{talent}), 
$\psi_{l|m}(h_n^\pm)$ 
is
multiplication by
$$
\sum_{r_1+\dots + r_{l+m}=n} 
h_{r_1}^\pm\otimes\cdots\otimes h^\pm_{r_{l+m}}.
$$
Now (\ref{atlonglast}) follows from (\ref{magic}), as that shows that
$\Psi_{l|m}(H_n^{\pm})$ satisfies an analogous formula.
\end{proof}

\begin{proof}[Proof of Theorem~\ref{t1}]
The isomorphism $\gamma_k$ is constructed in Theorem~\ref{atlast}.
The final part follows from the final part of
Theorem~\ref{fatherbrown}.
\end{proof}

\begin{proof}[Proof of Theorem~\ref{t2}]
As the maps involved are ring homomorphisms,
it suffices to show that the diagram commutes on the generators
$h_n^+$ and $e_n^-$ of $\rh_k$, which follows from (\ref{magic}).
\end{proof}

\section{Proof of Theorem~\ref{t3}}\label{sfin}

To prove Theorem~\ref{t3}, we need some explicit maps. 
To write these down, we
use some ``thick calculus'' in the same spirit as
\cite{KLMS}.
For $X, Y \in \Heis_k$ and idempotents $e_X:X\rightarrow X$ and
$e_Y:Y \rightarrow Y$ we have that 
$\Hom_{\Kar(\Heis_k)}((X,e_X), (Y,e_Y)) = e_Y \Hom_{\Heis_k}(X,Y) e_X$
by the definition of Karoubi envelope.
We will denote the identity endomorphisms of the objects
$H_n^+ = (\up^{\otimes n}, \imath_n(e_{(n)}))$ and
$E_n^- = (\down^{\otimes n}, \jmath_n(e_{(n)}))$ by
thick strings labelled by $n$, upward for $H_n^+$ and downward for $E_n^-$.
We stress that these objects are {\em not}
duals (unless $n=1$). Instead, 
in view of the definitions (\ref{im})--(\ref{jm}),
they are
interchanged by the symmetry $\Omega_k$.

We introduce more diagrammatic shorthands:
\begin{align}\label{thick1}
 &:= \binom{\,n\,}{\,r\,}\;\jmath_{n}(e_{(n)}):
E_{n-r}^-\otimes E_r^-\rightarrow E_n^-,
\label{thick3}
\end{align}
for $0 \leq r \leq n$ and $f \in \Sym_n$.
Again, the downward morphisms in (\ref{thick1})--(\ref{thick3}) are
the images of the upward
ones under $\Omega_k$.
Also, the merge and split 
morphisms are associative in an obvious sense allowing
their
definition to be extended to more strings, e.g., for three strings:
\begin{align*}
%
\:,
\end{align}
where $x^\lambda := x_1^{\lambda_1} \cdots x_n^{\lambda_n}$.

There are thick upward crossings which are defined recursively so that
\begin{equation}
%
\label{picard}
\end{equation}
for any $0 < r < m, 0 < s < n$.
There are thickened versions of the braid and quadratic relations (\ref{symmetric}).
Moreover, the braid relation implies further relations:
\begin{align}
%
\:.\label{butternutsquash}
 \end{align}
Similarly, there are thick downward, rightward and leftward crossings,
defined by the same pictures as (\ref{picard}) with different
orientations. Analogs of (\ref{butternutsquash}) hold for all
possible orientations.

Finally, there are thick cups and caps, which we define recursively by setting
\begin{align}
%
\:.
\end{align}
Symmetric polynomials commute across thick cups and caps, so we may
also draw them at the critical point without ambiguity.
By (\ref{thickcapsneedthis}), we have for $\lambda =
(\lambda_1,\dots,\lambda_n)\in\N^n$ that 
\begin{align}
%
\:,\label{nearlylunchtime}
\end{align}
where $\jonschi_\lambda$ is the (signed) Schur polynomial from
(\ref{chidef})
and $\rho=(n-1,\cdots,1,0) \in \Z^n$.
Although not needed here, we point out also that
\begin{align}\label{isthisclear}
%
\:.
\end{align}
These identities can both be proved by induction on $n$. 
For example, in the special case $n=2$, (\ref{isthisclear})
is equivalent to the assertion that
$$
\imath_2\left(4 e_{(2)} x_1 e_{(1^2)} x_2 e_{(2)}\right)
= \imath_2\left(e_{(2)} (1-(x_1-x_2)^2) e_{(2)}\right)=
\imath_2\left(4 e_{(2)} x_2 e_{(1^2)} x_1 e_{(2)}\right).
$$
This may be checked by replacing $e_{(1^2)}$ by $\frac{1}{2}(1-s_1)$,
commuting $s_1$ past $x_2$ noting that
$s_1 e_{(2)} = e_{(2)} = e_{(2)} s_1$, 
then symmetrizing the result by calculations involving (\ref{numbers}).

\begin{lemma}
\label{squishier}
Assume that $k \geq 0$ and $m,n > 0$.
Then
$$
\mathord{
\begin{tikzpicture}[baseline = -.5mm]
	\draw[->, line width=2pt] (0.28,0) to[out=90,in=-60] (-0.28,.6);
	\draw[-, line width=2pt] (-0.28,0) to[out=90,in=-120] (0.28,.6);
	\draw[<-, line width=2pt] (0.28,-.6) to[out=120,in=-90] (-0.28,0);
	\draw[-, line width=2pt] (-0.28,-.6) to[out=60,in=-90] (0.28,0);
        \node at (-0.3,-.7) {$\scriptstyle m$};
        \node at (0.3,.7) {$\scriptstyle n$};
\end{tikzpicture}
}\:=\:
\mathord{
\begin{tikzpicture}[baseline = -.5mm]
	\draw[<-, line width=2pt] (0.28,-.6) to (0.28,.6);
	\draw[->, line width=2pt] (-0.18,-.6) to (-0.18,.6);
        \node at (-0.18,-.7) {$\scriptstyle m$};
        \node at (0.28,.7) {$\scriptstyle n$};
\end{tikzpicture}
}\:
+\:
\sum_{a=0}^{k-1}
\mathord{
\begin{tikzpicture}[baseline = -.5mm]
	\draw[-, line width=2pt] (0.3,.3) to (0.3,.6);
	\draw[-, line width=1.6pt] (0.308,-.5) to (0.308,.5);
	\draw[<-, line width=2pt] (0.3,-.6) to (0.3,-.4);
	\draw[->, line width=2pt] (-0.3,.3) to (-0.3,.6);
	\draw[-, line width=1.6pt] (-0.308,-.5) to (-0.308,.5);
	\draw[-, line width=2pt] (-0.3,-.6) to (-0.3,-.4);
	\draw[-] (-0.272,-.41) to [out=90,in=180] (0,-.15) to[out=0,in=90] (.272,-.41);
        \node at (0,.15) [draw,fill=lightgray,rounded corners,inner sep=.9pt] {$\scriptstyle{\phantom{omething}}$};
        \node at (-0.3,-.7) {$\scriptstyle m$};
        \node at (0.3,-.7) {$\scriptstyle n$};
        \node at (-0.3,.7) {$\scriptstyle m$};
        \node at (0.3,.7) {$\scriptstyle n$};
        \node at (0,-.17) {$\dot$};
        \node at (0,-.3) {$\scriptstyle a$};
\end{tikzpicture}
}
$$
where a shaded box indicates a morphism which will not be
determined precisely.
\end{lemma}

\begin{proof}
We proceed by induction on $n$. The base case will be discussed in the
next paragraph.
For the induction step, assuming
$n > 1$, the induction hypothesis gives us that
\begin{align*}
\mathord{
\begin{tikzpicture}[baseline = -.5mm]
	\draw[->, line width=2pt] (0.28,0) to[out=90,in=-60] (-0.28,.6);
	\draw[-, line width=2pt] (-0.28,0) to[out=90,in=-120] (0.28,.6);
	\draw[<-, line width=2pt] (0.28,-.6) to[out=120,in=-90] (-0.28,0);
	\draw[-, line width=2pt] (-0.28,-.6) to[out=60,in=-90] (0.28,0);
        \node at (-0.3,-.7) {$\scriptstyle m$};
        \node at (0.3,.7) {$\scriptstyle n$};
\end{tikzpicture}
}\:&=
\frac{1}{n}
\mathord{
\begin{tikzpicture}[baseline = -.5mm]
	\draw[->, line width=2pt] (0.28,0) to[out=90,in=-60] (-0.28,.6);
	\draw[-, line width=2pt] (-0.28,-.6) to[out=60,in=-90] (0.28,0);
	\draw[<-, line width=2pt] (0.28,-.6) to[out=120,in=-40] (-0.14,-.27);
	\draw[-, line width=1.6pt] (-0.116,.284) to[out=-145,in=145] (-0.122,-.293);
	\draw[-] (-0.15,.28) to[out=-40,in=40] (-0.16,-.29);
	\draw[-, line width=2pt] (0.28,.6) to[out=-120,in=40] (-0.135,.26);
        \node at (-0.3,-.7) {$\scriptstyle m$};
        \node at (0.3,.7) {$\scriptstyle n$};
\end{tikzpicture}
}
=\frac{1}{n}
\mathord{
\begin{tikzpicture}[baseline = -.5mm]
	\draw[->, line width=2pt] (-0.28,-.6) to (-.28,-.3) to [out=90,in=-90]
        (0.38,0) to [out=90,in=-90] (-.28,.3) to (-.28,.6);
	\draw[<-, line width=2pt] (0.28,-.6) to[out=120,in=-40] (0.06,-.37);
	\draw[-, line width=1.6pt] (0.08,.38) to[out=-145,in=145] (0.071,-.387);
	\draw[-] (0.05,.38) to[out=-50,in=50] (0.04,-.39);
	\draw[-, line width=2pt] (0.28,.6) to[out=-120,in=40] (0.065,.36);
        \node at (-0.3,-.7) {$\scriptstyle m$};
        \node at (0.3,.7) {$\scriptstyle n$};
\end{tikzpicture}
}
=\frac{1}{n}
\mathord{
\begin{tikzpicture}[baseline = -.5mm]
	\draw[->, line width=2pt] (-0.28,-.6) to (-.28,-.3) to [out=90,in=-90]
        (0.05,0) to [out=90,in=-90] (-.28,.3) to (-.28,.6);
	\draw[<-, line width=2pt] (0.28,-.6) to[out=120,in=-40] (0.06,-.37);
	\draw[-, line width=1.6pt] (0.08,.38) to[out=-145,in=145] (0.071,-.387);
	\draw[-] (0.05,.38) to[out=-50,in=50] (0.04,-.39);
	\draw[-, line width=2pt] (0.28,.6) to[out=-120,in=40] (0.065,.36);
        \node at (-0.3,-.7) {$\scriptstyle m$};
        \node at (0.3,.7) {$\scriptstyle n$};
\end{tikzpicture}
}
+
\sum_{a=0}^{k-1}
\mathord{
\begin{tikzpicture}[baseline = -.5mm]
	\draw[-, line width=2pt] (0.1,.3) to (0.1,.6);
	\draw[->, line width=2pt] (-0.38,.3) to (-0.38,.6);
	\draw[<-, line width=2pt] (0.3,-.6) to (0.3,-.4);
	\draw[-, line width=2pt] (-0.3,-.6) to (-0.3,-.5) to [out=90,in=-135] (0,-.1);
        \draw[-] (.015,-.122) to [out=20,in=180] (.18,-0.1) to
        [out=0,in=60] (.324,-.41);
        \draw[-,line width=1.6pt] (-0.008,-.12) [out=75,in=-120] to (.05,.2);
        \draw[-,line width=1.6pt] (0.3,-.42) [out=120,in=-90] to
        (-.3,-0.1) to [out=90,in=-90](-.31,0.15);
        \node at (-0.3,-.7) {$\scriptstyle m$};
        \node at (0.3,-.7) {$\scriptstyle n$};
        \node at (-0.38,.7) {$\scriptstyle m$};
        \node at (0.1,.7) {$\scriptstyle n$};
        \node at (-0.13,.2) [draw,fill=lightgray,rounded corners,inner
        sep=.9pt] {$\scriptstyle{\phantom{oething}}$};
        \node at (0.3,-0.18) {$\dot$};
        \node at (0.45,-0.14) {$\scriptstyle a$};
\end{tikzpicture}
}\\
&=\frac{1}{n}
\mathord{
\begin{tikzpicture}[baseline = -.5mm]
	\draw[->, line width=2pt] (-0.4,-.6) to (-.4,.6);
	\draw[<-, line width=2pt] (0.28,-.6) to (0.28,-.32);
	\draw[-, line width=1.6pt] (0.29,.43) to[out=-145,in=145] (0.29,-.337);
	\draw[-] (0.29,.43) to[out=-50,in=50] (0.29,-.34);
	\draw[-, line width=2pt] (0.28,.55) to (0.28,.41);
        \node at (-0.4,-.7) {$\scriptstyle m$};
        \node at (0.28,.7) {$\scriptstyle n$};
       \node at (-0.12,0.05) {$\scriptstyle n-1$};
\end{tikzpicture}
}
+\sum_{a=0}^{k-1}
\mathord{
\begin{tikzpicture}[baseline = -.5mm]
	\draw[-, line width=2pt] (0.38,.3) to (0.38,.6);
	\draw[-, line width=1.2pt] (0.301,-.5) to (0.301,.2);
	\draw[<-, line width=2pt] (0.3,-.8) to (0.3,-.57);
	\draw[-, line width=1.63pt] (0.2935,-.571) to (0.2935,-.4);
	\draw[->, line width=2pt] (-0.3,.3) to (-0.3,.6);
	\draw[-, line width=1.6pt] (-0.307,-.5) to (-0.307,.5);
	\draw[-, line width=2pt] (-0.3,-.8) to (-0.3,-.4);
	\draw[-] (-0.272,-.41) to [out=90,in=180] (0,-.15) to[out=0,in=90] (.272,-.41);
        \draw[-] (.45,0.1) to [out=-60,in=45] (.3,-.6);
        \node at (0.07,.15) [draw,fill=lightgray,rounded corners,inner
        sep=.9pt] {$\scriptstyle{\phantom{omethngm}}$};
        \node at (-0.3,-.9) {$\scriptstyle m$};
        \node at (0.3,-.9) {$\scriptstyle n$};
        \node at (-0.3,.7) {$\scriptstyle m$};
        \node at (0.3,.7) {$\scriptstyle n$};
        \node at (0,-.15) {$\dot$};
        \node at (0,-.3) {$\scriptstyle a$};
\end{tikzpicture}
}
+
\sum_{a=0}^{k-1}
\mathord{
\begin{tikzpicture}[baseline = .5mm]
	\draw[-, line width=2pt] (0.2,.5) to (0.2,.8);
	\draw[->, line width=2pt] (-0.3,.5) to (-0.3,.8);
	\draw[<-, line width=2pt] (0.3,-.6) to (0.3,-.38);
	\draw[-, line width=2pt] (-0.3,-.6) to (-0.3,-.5) to [out=90,in=-135] (-0.2,-.3);
        \draw[-] (-.191,-.325) to [out=25,in=180] (.18,-0.05) to
        [out=0,in=60] (.315,-.41);
        \draw[-,line width=1.6pt] (-0.21,-.32) [out=90,in=-100] to (.22,.4);
        \draw[-,line width=1.6pt] (0.31,-.4) [out=120,in=-90] to (-.3,.4);
        \node at (-0.3,-.7) {$\scriptstyle m$};
        \node at (0.3,-.7) {$\scriptstyle n$};
        \node at (-0.3,.9) {$\scriptstyle m$};
        \node at (0.2,.9) {$\scriptstyle n$};
        \node at (-0.05,.4) [draw,fill=lightgray,rounded corners,inner
        sep=.9pt] {$\scriptstyle{\phantom{oething}}$};
        \node at (0.3,-0.15) {$\dot$};
        \node at (0.47,-0.1) {$\scriptstyle a$};
\end{tikzpicture}
}=
\mathord{
\begin{tikzpicture}[baseline = -.5mm]
	\draw[<-, line width=2pt] (0.28,-.6) to (0.28,.6);
	\draw[->, line width=2pt] (-0.18,-.6) to (-0.18,.6);
        \node at (-0.18,-.7) {$\scriptstyle m$};
        \node at (0.28,.7) {$\scriptstyle n$};
\end{tikzpicture}
}\:
+\:
\sum_{a=0}^{k-1}
\mathord{
\begin{tikzpicture}[baseline = -.5mm]
	\draw[-, line width=2pt] (0.3,.3) to (0.3,.6);
	\draw[-, line width=1.6pt] (0.307,-.5) to (0.307,.5);
	\draw[<-, line width=2pt] (0.3,-.6) to (0.3,-.4);
	\draw[->, line width=2pt] (-0.3,.3) to (-0.3,.6);
	\draw[-, line width=1.6pt] (-0.307,-.5) to (-0.307,.5);
	\draw[-, line width=2pt] (-0.3,-.6) to (-0.3,-.4);
	\draw[-] (-0.272,-.41) to [out=90,in=180] (0,-.15) to[out=0,in=90] (.272,-.41);
        \node at (0,.15) [draw,fill=lightgray,rounded corners,inner sep=.9pt] {$\scriptstyle{\phantom{omething}}$};
        \node at (-0.3,-.7) {$\scriptstyle m$};
        \node at (0.3,-.7) {$\scriptstyle n$};
        \node at (-0.3,.7) {$\scriptstyle m$};
        \node at (0.3,.7) {$\scriptstyle n$};
        \node at (0,-.17) {$\dot$};
        \node at (0,-.31) {$\scriptstyle a$};
\end{tikzpicture}
}+
\sum_{a=0}^{k-1}
\mathord{
\begin{tikzpicture}[baseline = -.5mm]
	\draw[-, line width=2pt] (0.1,.3) to (0.1,.6);
	\draw[->, line width=2pt] (-0.4,.3) to (-0.4,.6);
	\draw[<-, line width=2pt] (0.3,-.6) to (0.3,-.38);
	\draw[-, line width=2pt] (-0.3,-.6) to (-0.3,-.5) to [out=90,in=-135] (-0.2,-.3);
        \draw[-] (-.193,-.325) to [out=25,in=180] (.18,-0.05) to
        [out=0,in=60] (.315,-.41);
        \draw[-,line width=1.6pt] (-0.21,-.33) [out=110,in=-90] to (-.3,.2);
        \draw[-,line width=1.6pt] (0.294,-.4) [out=135,in=-90] to [out=90,in=-90]
        (-0,0) to [in=-90,out=90] (.2,.25);
        \node at (-0.3,-.7) {$\scriptstyle m$};
        \node at (0.3,-.7) {$\scriptstyle n$};
        \node at (-0.4,.7) {$\scriptstyle m$};
        \node at (0.1,.7) {$\scriptstyle n$};
        \node at (-0.11,.2) [draw,fill=lightgray,rounded corners,inner
        sep=.9pt] {$\scriptstyle{\phantom{oethings}}$};
        \node at (0.3,-0.13) {$\dot$};
        \node at (0.45,-0.12) {$\scriptstyle a$};
\end{tikzpicture}
}\:.
\end{align*}
Now we commute the $a$ dots in the final term to the left
past the crossing. This also produces correction terms, but these
all have strictly fewer than $a$ dots on the cap so are allowed.

It just remains to treat the base case $n=1$. This proceeds by
induction on $m=1,2,\dots$. The case $m=1$ follows from
(\ref{sideways}) and (\ref{bubbles}).
The induction step follows by a calculation which is the mirror image
in a vertical axis
of the calculation in the previous
paragraph, starting by splitting the string of thickness $m$ into
strings of thickness $1$ and $m-1$.
\end{proof}

\begin{corollary}\label{whatIreallywant}
For $k \geq 0$ and $m,n > 0$, we have that
$$
\mathord{
\begin{tikzpicture}[baseline =1mm]
	\draw[<-, line width=2pt] (0.28,-.6) to (0.28,.8);
	\draw[->, line width=2pt] (-0.18,-.6) to (-0.18,.8);
        \node at (-0.18,-.7) {$\scriptstyle m$};
        \node at (0.28,.9) {$\scriptstyle n$};
\end{tikzpicture}
}\:=
\: \sum_{\substack{0 \leq r \leq \min(m,n)\\\lambda \in \mathcal P_{r,k}}}
\mathord{
\begin{tikzpicture}[baseline = 1mm]
	\draw[-, line width=2pt] (0.3,.5) to (0.3,.8);
	\draw[-, line width=1pt] (0.318,-.41) to (0.318,-.25) to [out=90,in=-60] (-0.25,.3);
	\draw[<-, line width=2pt] (0.3,-.6) to (0.3,-.4);
	\draw[->, line width=2pt] (-0.3,.5) to (-0.3,.8);
	\draw[-, line width=1pt] (-0.318,-.41) to (-.318,-.25) to [out=90,in=-120] (0.25,.3);
	\draw[-, line width=2pt] (-0.3,-.6) to (-0.3,-.4);
	\draw[-, line width=1pt] (-0.2825,-.41) to [out=90,in=180] (0,-.15) to[out=0,in=90] (.2825,-.41);
        \node at (0,.4) [draw,fill=lightgray,rounded corners,inner sep=.9pt] {$\scriptstyle{\phantom{omething}}$};
        \node at (-0.3,-.7) {$\scriptstyle m$};
        \node at (0.3,-.7) {$\scriptstyle n$};
        \node at (-0.3,.9) {$\scriptstyle m$};
        \node at (0.3,.9) {$\scriptstyle n$};
       \node at (-0.5,0) {$\scriptstyle m-r$};
        \node at (0,-.15) {$\bigdot$};
        \node at (0.02,-.33) {$\scriptstyle \chi_{\!_{\lambda}}$};
\end{tikzpicture}
}\:.
$$
\end{corollary}

\begin{proof}
Rearrange the identity from Lemma~\ref{squishier} to get the $r=0$
term in the sum exactly, then use induction on
$\min(m,n)$ plus (\ref{nearlylunchtime}) to get the other terms.
\end{proof}

\begin{proof}[Proof of Theorem~\ref{t3}]
We just treat the case $k \geq 0$; the result for $k \leq 0$ then
follows easily by applying $\Omega_k$ (also transposing matrices).
The thick upward (resp., downward) crossing gives a canonical isomorphism $H_m^+ \otimes
H_n^+ \stackrel{\sim}{\rightarrow} H_n^+ \otimes H_m^+$
(resp.,
$E_m^- \otimes
E_n^- \stackrel{\sim}{\rightarrow} E_n^- \otimes E_m^-$).
For the remaining relation, we must construct an isomorphism between the objects
\begin{align*}
P&:=H_m^+ \otimes E_n^-,
&Q&:=\bigoplus_{r=0}^{\min(m,n)}\bigoplus_{\lambda \in
    \mathcal P_{r,k}}
\E_{n-r}^- \otimes H_{m-r}^+.
\end{align*}
Corollary~\ref{whatIreallywant} shows that the 
morphism
$\theta_{m,n}:P\rightarrow Q$
defined by the
column vector
\begin{equation}
\left[
\begin{tikzpicture}[baseline = 0]
\draw[-,line width=1pt] (-0.29,-.4) to [out=90,in=-90](-.29,-.25) to
[out=90,in=180] (0,0.05) to [out=0,in=90] (.29,-.25) to [out=-90,in=90] (0.29,-.35);
        \node at (-0.3,-.48) {$\scriptstyle m$};
        \node at (-0.3,.68) {$\scriptstyle n-r$};
        \node at (-0.55,.15) {$\scriptstyle m-r$};
        \node at (0,0.05) {$\bigdot$};
        \node at (0.02,-.13) {$\scriptstyle \chi_{\!_{\lambda}}$};
\draw[-,line width=1pt] (-.32,-.4) to (-.31,-0.05);
\draw[-,line width=1pt] (.32,-.35) to (.31,-0.05);
\draw[<-,line width=2pt] (.3,-.4) to (.3,-.3);
\draw[-,line width=1pt] (-.32,.6) to [out=-90,in=90] (.31,-0.05);
\draw[<-,line width=1pt] (.32,.6) to [out=-90,in=90] (-.31,-0.05);
\end{tikzpicture}\:\:\right]_{0 \leq r \leq \min(m,n),\lambda \in
\mathcal{P}_{r,k}}
\end{equation}
has a left inverse $\phi_{m,n}$.
Moreover, thanks to Theorem~\ref{t1} and (\ref{upper}), we have that
$[P] = [Q]$ in $K_0(\Kar(\Heis_k))$.
Using the final part of Theorem~\ref{t1},
this is enough to imply that $\phi_{m,n}$ is actually the two-sided
inverse of $\theta_{m,n}$.

To explain the last assertion in more detail, we use (\ref{yon}) to
translate into a statement about projective modules over the
locally unital algebra $A$ arising from $\Heis_k$.
Remembering that this is a contravariant equivalence, 
we have finitely generated projective $A$-modules $P, Q$ such that $[P] = [Q]$,
and homomorphisms $\theta_{m,n}:Q \rightarrow P$ and $\phi_{m,n}:P
\rightarrow Q$ such that $\theta_{m,n} \circ \phi_{m,n} =
\operatorname{id}_P$, and need to show that $\theta_{m,n}$ is an
isomorphism.
Let $R := \ker \theta_{m,n}$.
Since $\theta_{m,n}$ has a right inverse, it is surjective. 
Since $Q$
is projective, we have that $Q \cong P \oplus R$.
Since $[P]=[Q]$, we deduce that $[R] = 0$. Hence, $R = 0$.
\end{proof}



\begin{thebibliography}{BCNR}

\bibitem[B]{B2}
J. Brundan,
On the definition of Heisenberg category, 
{\em Alg. Comb.} {\bf 1} (2018), 523--544.

\bibitem[BCNR]{BCNR}
J. Brundan, J. Comes, D. Nash and A. Reynolds,
A basis theorem for the affine oriented Brauer category and its
cyclotomic quotients,
{\em Quantum Topology} {\bf 8} (2017), 75--112.

\bibitem[BSW1]{qheis}
J. Brundan, A. Savage and B. Webster,
On the definition of quantum Heisenberg category,
 {\em Alg. Numb. Th.} {\bf 14} (2020), 275--321.

\bibitem[BSW2]{Foundations}
\bysame,
Foundations of Frobenius Heisenberg categories,
{\em J. Algebra} {\bf 578} (2021), 115--185.

\bibitem[BSW3]{QFrobHeis}
\bysame,
Quantum Frobenius Heisenberg categorification,
 {\em J. Pure Appl. Algebra} {\bf 226} (2022), Paper No. 106792, 
 50 pp.

\bibitem[D]{D}
A. Davydov,
$K$-theory of rings with idempotents,
{\em Math. Notes} {\bf 53} (1993), 253--259.

\bibitem[E]{Eisenbud}
D. Eisenbud,
{\em Commutative Algebra with a View toward Algebraic Geometry}, Graduate Texts in Mathematics, vol. 150, Springer-Verlag, 1995.

\bibitem[Kh]{K}
M. Khovanov,
Heisenberg algebra and a graphical calculus,
{\em Fund. Math.} {\bf 225} (2014), 169--210.

\bibitem[KLMS]{KLMS}
M. Khovanov, A. Lauda, M. Mackaay and M. St\v osi\'c,
Extended graphical calculus for categorified quantum
$\mathfrak{sl}(2)$,
{\em Mem. Amer. Math. Soc.} {\bf 219} (2012), no. 1029, 87 pp.

\bibitem[Kl]{Kbook}
A. Kleshchev,
{\em Linear and Projective Representations of Symmetric Groups},
Cambridge University Press,
2005.

\bibitem[LRS]{LRS}
A.~Licata, D.~Rosso, and A.~Savage,
A graphical calculus for the {J}ack inner product on symmetric
  functions, {\em J. Combin. Theory Ser. A} {\bf 155} (2018), 503--543.

\bibitem[M]{Mac}
I. G. Macdonald, {\em Symmetric Functions and Hall Polynomials},
Oxford Mathematical Monographs, second edition, OUP, 1995.

\bibitem[MS]{MS18}
M. Mackaay and A. Savage,
Degenerate cyclotomic {H}ecke algebras and higher level {H}eisenberg
  categorification, {\em J. Algebra} {\bf 505} (2018), 150--193.

\bibitem[Q]{Q}
D. Quillen,
Higher algebraic K-theory I, 
{\em Lect. Notes Math.} {\bf 341} (1973), 85–-147.

\bibitem[R]{Ros94}
J. Rosenberg,
{\em Algebraic {$K$}-Theory and its Applications},
Graduate Texts in Mathematics, vol. 147, Springer-Verlag, 1994.

\bibitem[S]{Sua}
D.~B. Suarez, Integral presentations of quantum lattice {H}eisenberg
algebras,
in: ``Categorification and Higher Representation Theory,''
{\em Contemp. Math.},  vol. 683, pp. 247--259, Amer. Math. Soc., Providence,
  RI, 2017.

\bibitem[TV]{TV}
V. Turaev and A. Virelizier,
{\em Monoidal Categories and Topological Field Theory},
 Progress in Mathematics, 322, Birkhäuser/Springer, 2017.

\bibitem[W]{Wunfurling}
B. Webster,
Unfurling Khovanov-Lauda-Rouquier algebras;
\arxiv{1603.06311}.
\end{thebibliography}
\end{document}